\documentclass{article}
\usepackage{amsmath}
\usepackage{amssymb}
\usepackage{amsthm}
\usepackage{algorithm}
\usepackage{algorithmic}
\usepackage{array}
\usepackage{ascmac}
\usepackage{bm}
\usepackage{comment}
\usepackage{multirow}
\usepackage[dvipdfmx]{graphicx}
\usepackage{color}
\usepackage{lscape}
\usepackage[top=20truemm,bottom=20truemm,left=16truemm,right=16truemm]{geometry}
\usepackage{subcaption}
\usepackage{here}
\usepackage{booktabs}
\def\@themcountersep{}
\newtheorem{theorem}{Theorem}[section]

\newtheorem{example}[theorem]{Example}
\newtheorem{lemma}[theorem]{Lemma}
\newtheorem{proposition}[theorem]{Proposition}
\newtheorem{remark}[theorem]{Remark}
\renewcommand{\algorithmicrequire}{\textbf{Input:}}
\renewcommand{\algorithmicensure}{\textbf{Output:}}
\newcommand{\argmax}{\mathop{\rm arg~max}\limits}
\newcommand{\argmin}{\mathop{\rm arg~min}\limits}

\usepackage[dvipsnames]{xcolor}

\newcommand{\modifyFirst}[1]{\textcolor{black}{#1}}
\definecolor{fuchsia}{rgb}{1.0, 0.0, 1.0} 
\newcommand{\modifySecond}[1]{\textcolor{black}{#1}}
\definecolor{new_orange}{rgb}{1, 0.4, 0.2}
\definecolor{new_green}{rgb}{0.1,0.5, 0.2}
\newcommand{\modifyThird}[1]{\textcolor{black}{#1}}
\usepackage{ulem}

\begin{document}

\title{A New Extension of Chubanov's Method to Symmetric Cones}

\author{Shin-ichi Kanoh\thanks{
Graduate School of Systems and Information Engineering, University of Tsukuba, Tsukuba, Ibaraki 305-8573, and Japan Society for the Promotion of Science, 5-3-1 Kojimachi, Chiyoda-ku, Tokyo 102-0083, Japan. email: s2130104@s.tsukuba.ac.jp
}
and 
Akiko Yoshise\thanks{Corresponding author. Faculty of Engineering, Information and Systems, University of Tsukuba, Tsukuba, Ibaraki 305-8573, Japan. email: yoshise@sk.tsukuba.ac.jp
}     
}

\date{October 2021 \\ Revised November 2022 \\ Revised April 2023}

\maketitle     
\begin{abstract}
We propose a new variant of Chubanov's method for solving the feasibility problem over the symmetric cone by extending Roos's method (2018) of solving the feasibility problem over the nonnegative orthant. The proposed method considers a feasibility problem associated with a norm induced by the maximum eigenvalue of an element and uses a rescaling focusing on the upper bound for the sum of eigenvalues of any feasible solution to the problem. Its computational bound is (i) equivalent to that of Roos's original method (2018) and superior to that of Louren\c{c}o et al.'s method (2019) when the symmetric cone is the nonnegative orthant, (ii) superior to that of Louren\c{c}o et al.'s method (2019) when the symmetric cone is a Cartesian product of second-order cones, (iii) equivalent to that of Louren\c{c}o et al.'s method (2019) when the symmetric cone is the simple positive semidefinite cone, and (iv) superior to that of Pena and Soheili's method (2017) for any simple symmetric cones under the feasibility assumption of the problem imposed in Pena and Soheili's method (2017).
We also conduct numerical experiments that compare the performance of our method with existing methods by generating strongly (but ill-conditioned) feasible instances.
For any of these instances, the proposed method is rather more efficient than the existing methods in terms of accuracy and execution time.
\end{abstract}

\section{Introduction}
Recently, Chubanov \cite{Chubanov2012,Chubanov2015} proposed a new polynomial-time algorithm for solving the problem (${\rm P}(A)$), 
\begin{equation}
\begin{array}{lllllll}
{\rm P}(A)	&	&\mbox{find}	&x > \bm{0}	&	&\mbox{s.t.} &  Ax = \bm{0},
\end{array}
\notag
\end{equation}
where $A$ is a given integer (or rational) matrix and $\mbox{rank} (A) = m$ and $\bm{0}$ is an $n$-dimensional vector of $0$s. The method explores the feasibility of the following problem ${\rm P}_{S_1}(A)$, which is equivalent to ${\rm P}(A)$ and given by
\begin{equation}
\begin{array}{llllllll}
{\rm P}_{S_1}(A)	&	&\mbox{find}   &x > \bm{0}	&	&\mbox{s.t.} &Ax = \bm{0},	&\bm{0} < x \leq \bm{1},
\end{array}
\notag
\end{equation}
where $\bm{1}$ is an $n$-dimensional vector of $1$s. Chubanov's method consists of two ingredients, the ``main algorithm'' and the ``basic procedure.''
\modifyFirst{
Note that the alternative problem ${\rm D}(A)$ of ${\rm P}(A)$ is given by 
\begin{equation}
\notag
\begin{array}{llllllllll}
{\rm D}(A)	&	&\mbox{find}   &y \geq \bm{0}	&	&\mbox{s.t.}	&y \in \mbox{range} A^\top,	&y \neq \bm{0},
\end{array}
\end{equation}
where ${\rm range} A^\top$ is the orthogonal complement of ${\rm ker}A$.
}
The structure of the method is as follows: In the outer iteration, the main algorithm calls the basic procedure, which generates a sequence in $\mathbb{R}^n$ using projection to the set \modifyFirst{$\mbox{ker} A := \{x \in \mathbb{R}^n \mid Ax= \bm{0}\}$.} The basic procedure terminates in a finite number of iterations returning one of the following:
(i). a solution of problem ${\rm P}(A)$, 
\modifyFirst{(ii). a solution of problem ${\rm D}(A)$}, or
(iii). a cut of ${\rm P}(A)$, i.e., an index $j \in \{ 1,2,\dots,n \}$ for which $0 < x_j \leq \frac{1}{2} \notag $ holds for any feasible solution of problem ${\rm P}_{S_1}(A)$.

If result (i) or (ii) is returned by the basic procedure, then the feasibility of problem ${\rm P}(A)$ can be determined and the main procedure stops. If result (iii) is returned, then the main procedure generates a diagonal matrix $D \in \mathbb {R}^{n \times n} $ with a $ (j, j) $ element of $2$ and all other diagonal elements of $1$ and rescales the matrix as $AD ^ {-1} $. Then, it calls the basic procedure with the rescaled matrix. Chubanov's method checks the feasibility of ${\rm P}(A) $ by repeating the above procedures.

For problem ${\rm P}(A)$, 
\cite{Roos2018} proposed a tighter cut criterion of the basic procedure than the one used in \cite{Chubanov2015}. \cite{Chubanov2015} used the fact that $x_j \leq \frac{ \sqrt{n} \|z\|_2}{y_j}$ holds for any $y \in \mathbb{R}^n$ satisfying $\sum_{i=1}^n y_i = 1 , y \geq 0$ and $y \notin \mbox{range} A^T $, $z \in \mathbb{R}^n $ obtained by projecting this $y$ onto $\mbox{ker} A$, and any feasible solution $x \in \mathbb{R}^n$ of ${\rm P}_{S_1}(A)$, and the basic procedure is terminated if a $y$ is found for which $\frac{ \sqrt{n} \|z\|_2}{y_j} \leq \frac{1}{2}$ holds for some index $j$.
On the other hand, \cite{Roos2018} showed that for $v = y - z$, $x_j \leq \min ( 1 , \bm{1}^T \left[ -v /v_j \right]^+ ) \leq \frac{ \sqrt{n} \|z\|_2}{y_j}$ holds if $v_j \neq 0$, where $\left[ -v / v_j \right]^+$ is the projection of $-v/v_j \in \mathbb{R}^n$ onto the nonnegative orthant and $\bm{1}$ is the vector of ones, and the basic procedure is terminated if a $y$ is found for which $\bm{1}^T \left[ -v/v_j \right]^+ \leq \frac{1}{2}$ holds.

Chubanov's method has also been extended to include the feasibility problem over the second-order cone \cite{Kitahara2018} and the symmetric cone \cite{Pena2017,Bruno2019}. The feasibility problem over the symmetric cone is of the form, 
\begin{equation}
\begin{array}{lllllll}
{\rm P}(\mathcal{A})	&	&\mbox{find}   &x \in \mbox{int} \hspace{0.75mm} \mathcal{K}	&	&\mbox{s.t.}	&\mathcal{A} (x) = \bm{0}, 
\end{array}
\notag
\end{equation}
where $\mathcal{A}$ is a linear operator, $\mathcal{K}$ is a symmetric cone, and $\mbox{int} \hspace{0.75mm} \mathcal{K}$ is the interior of the set $\mathcal{K}$.
As proposed in \cite{Pena2017,Bruno2019}, for problem ${\rm P}(\mathcal{A})$, the structure of Chubanov's method remains the same; i.e., the main algorithm calls the basic procedure, and the basic procedure returns one of the following in a finite number of iterations: 
(i). a solution of problem ${\rm P}(\mathcal{A})$, or
(ii). a solution of the alternative problem of problem ${\rm P}(\mathcal{A})$, or
(iii). a recommendation of scaling problem ${\rm P}(\mathcal{A})$.
If result (i) or (ii) is returned by the basic procedure, then the feasibility of the problem ${\rm P}(\mathcal{A})$ can be determined and the main procedure stops. If result (iii) is returned, the problem is scaled appropriately and the basic procedure is called again.

It should be noted that the purpose of rescaling differs between \cite{Bruno2019} and \cite{Pena2017}.
In \cite{Pena2017}, the authors devised a rescaling method so that the following value becomes larger:
\begin{equation}
\delta (\mbox{ker} \hspace{0.3mm} \mathcal{A} \cap \mathcal{K}) := \max_x \left \{ {\rm det}(x) \mid x \in \mbox{ker} \hspace{0.3mm} \mathcal{A} \cap \mathcal{K} , \|x\|_J^2 = r  \right\}, \notag
\end{equation}
where $\mbox{ker} \hspace{0.3mm} \mathcal{A} :=\{ x \mid \mathcal{A}(x) = \bm{0} \}$ and $\|x\|_J$ is the norm induced by the inner product $\langle x, y \rangle = {\rm trace}(x \circ y)$ defined in section \ref{sec: notation}.
They proposed four updating schemes to be employed in the basic procedure 
and conducted numerical experiments to compare the effect of these schemes when the symmetric cone is the nonnegative orthant \cite{Pena2019}.

In \cite{Bruno2019}, the authors assumed that the symmetric cone $\mathcal{K}$ is given by the Cartesian product of $p$ simple symmetric cones $\mathcal{K}_1, \ldots, \mathcal{K}_p$, and they investigated the feasibility of the problem (${\rm P}_{S_{1,\infty}}(\mathcal{A})$),
\begin{equation}
\begin{array}{llllllll}
{\rm P}_{S_{1,\infty}}(\mathcal{A})	&	&\mbox{find}   &x \in \mbox{int} \hspace{0.75mm} \mathcal{K}	&	&\mbox{s.t.}	&\mathcal{A} (x) =\bm{0},	&\|x\|_{1 , \infty} \leq 1,
\end{array}
\notag
\end{equation}
where for each $x = (x_1, \ldots, x_p)  \in \mathcal{K} =\mathcal{K}_1 \times \cdots \mathcal{K}_p$,  $\|x\|_{1,\infty}$ is defined by $\|x\|_{1,\infty} := \max \{ \|x_1\|_1 , \dots ,\|x_p\|_1 \}$,
and $\|x\|_1$ is the sum of the absolute values of all eigenvalues of $x$. Note that if $p=1$, then problem ${\rm P}_{S_{1,\infty}}(\mathcal{A})$ turns out to be ${\rm P}_{S_1}(\mathcal{A})$, which is equivalent to ${\rm P}(\mathcal{A})$:
\begin{equation}
\begin{array}{llllll}
{\rm P}_{S_1}(\mathcal{A})	&\mbox{find}   &x \in \mbox{int} \hspace{0.75mm} \mathcal{K}	&\mbox{s.t.}	&\mathcal{A} (x) =\bm{0},	&\|x\|_{1} \leq 1.
\end{array}
\notag
\end{equation}

The authors focused on the volume of the feasible region of ${\rm P}_{S_{1,\infty}}(\mathcal{A})$ and devised a rescaling method so that the volume becomes smaller.
Their method will stop when the feasibility of problem {${\rm P}_{S_{1,\infty}}(\mathcal{A})$ or the fact that the minimum eigenvalue of any feasible solution of problem ${\rm P}_{S_{1,\infty}}(\mathcal{A})$ is less than $\varepsilon$ is determined.

The aim of this paper is to devise a new variant of Chubanov's method for solving ${\rm P}(\mathcal{A})$ by extending Roos's method \cite{Roos2018} to the following feasibility problem (${\rm P}_{S_{\infty}}(\mathcal{A})$) over the symmetric cone $\mathcal{K}$:
\begin{equation}
\begin{array}{llllllll}
{\rm P}_{S_{\infty}}(\mathcal{A})	&	&\mbox{find}   &x \in \mbox{int} \hspace{0.75mm} \mathcal{K}	&	&\mbox{s.t.}	&\mathcal{A} (x) = \bm{0},	&\|x\|_\infty \leq 1,
\end{array}
\notag
\end{equation}
where $\|x\|_\infty$ is the maximum absolute eigenvalue of $x$.
Throughout this paper, we will assume that $\mathcal{K}$ is the Cartesian product of $p$ simple symmetric cones $\mathcal{K}_1, \dots, \mathcal{K}_p$, i.e., $\mathcal{K} = \mathcal{K}_1 \times \dots \times \mathcal{K}_p$.
Here, we should mention an important issue about Lemma 4.2 in \cite{Roos2018}, which is one of the main results of \cite{Roos2018}.
The proof of Lemma 4.2 given in the paper \cite{Roos2018} is incorrect and a correct proof is provided in the paper \cite{Wei2019}, while this study derives theoretical results without referring to the lemma.
Our method has a feature that the main algorithm works while keeping information about the minimum eigenvalue of any feasible solution of ${\rm P}_{S_{\infty}}(\mathcal{A})$ and, in this sense, it is closely related to Louren\c{c}o et al.'s method \cite{Bruno2019}. Using the norm $\| \cdot \|_\infty$ in problem ${\rm P}_{S_\infty} (\mathcal{A})$ makes it possible to 
\begin{itemize}
\item 
calculate the upper bound for the minimum eigenvalue of any feasible solution of ${\rm P}_{S_\infty} (\mathcal{A})$, 
\item
quantify the feasible region of ${\rm P} (\mathcal{A})$, and hence, 
\item
determine whether there exists a feasible solution of ${\rm P} (\mathcal{A})$ whose minimum eigenvalue is greater than $\varepsilon$ as in \cite{Bruno2019}.
\end{itemize}

Note that the symmetric cone optimization includes several types of problems (linear, second-order cone, and semi-definite optimization problems) with various settings and the computational bound of an algorithm depends on these settings. As we will describe in section \ref{sec: compare}, the theoretical computational bound of our method is 
\begin{itemize}
\item equivalent to that of Roos's original method \cite{Roos2018} and superior to that of Louren\c{c}o et al.'s method \cite{Bruno2019} when the symmetric cone is the nonnegative orthant,
\item superior to that of Louren\c{c}o et al.'s method when the symmetric cone is a Cartesian product of second-order cones, and
\item equivalent to that of Louren\c{c}o et al.'s method when the symmetric cone is the simple positive semidefinite cone, under the assumption that the costs of computing the spectral decomposition and of the minimum eigenvalue are of the same order for any given symmetric matrix.
\item superior to that of Pena and Soheili's method~\cite{Pena2017} for any simple symmetric cones under the feasibility assumption of the problem imposed in~\cite{Pena2017}.
\end{itemize}
}

Another aim of this paper is to give comprehensive numerical comparisons of the existing algorithms and our method. 
As described in section \ref{sec: numerical experiments}, we generate strongly feasible ill-conditioned instances, i.e.,   $\mbox{ker} \hspace{0.3mm} \mathcal{A} \cap \mbox{int} \hspace{0.75mm} \mathcal{K} \neq  \emptyset$ and $x \in \mbox{ker} \hspace{0.3mm} \mathcal{A} \cap \mbox{int} \hspace{0.75mm} \mathcal{K}$ has positive but small eigenvalues,
for the simple positive semidefinite cone $\mathcal{K}$,  and conduct numerical experiments.

The paper is organized as follows:
Section \ref{sec:EJA} contains a brief description of Euclidean Jordan algebras and their basic properties.
Section \ref{sec: extension} gives a collection of propositions which are necessary to extend Roos's method to  problem ${\rm P}_{S_{\infty}}(\mathcal{A})$ over the symmetric cone.
In sections \ref{sec: basic procedure} and \ref{sec: main algorithm}, we explain the basic procedure and the main algorithm of our variant of Chubanov's method.
Section \ref{sec: compare} compares the theoretical computational bounds of  Louren\c{c}o et al.'s method~\cite{Bruno2019}, Pena and Soheili's method~\cite{Pena2017} and our method.
In section \ref{sec: numerical experiments}, we conduct numerical experiments comparing our variant with the existing methods.
The conclusions are summarized in section \ref{sec: concluding remarks}.

\section{Euclidean Jordan algebras and their basic properties}
\label{sec:EJA}
In this section, we briefly introduce Euclidean Jordan algebras and symmetric cones. For more details, see~\cite{Faraut1994}. In particular, the relation between symmetry cones and Euclidean Jordan algebras is given in Chapter III (Koecher and Vinberg theorem) of~\cite{Faraut1994}.
\subsection{Euclidean Jordan algebras}
Let $\mathbb{E}$ be a real-valued vector space equipped with an inner product $\langle \cdot, \cdot \rangle$ and a bilinear operation $\circ$ : $\mathbb{E} \times \mathbb{E} \rightarrow \mathbb{E}$,
and $e$ be the identity element, i.e.,$x \circ e = e \circ x = x$ holds for any $ x \in \mathbb{E}$.
$(\mathbb{E}, \circ)$ is called a Euclidean Jordan algebra if it satisfies
\begin{align*}
x \circ y = y  \circ x,  \ \
x \circ (x^2 \circ y) = x^2 \circ (x \circ y), \ \
\langle x \circ y , z \rangle = \langle y , x \circ z \rangle
\end{align*}
for all $x,y,z \in \mathbb{E}$ and $x^2 := x \circ x$.
We denote $y \in \mathbb{E}$ as $x^{-1}$ if  $y$ satisfies $x \circ y = e$.
$c \in \mathbb{E}$ is called an {\em idempotent} if it satisfies $c \circ c = c$, and an idempotent $c$ is called {\em primitive} if it can not be written as a sum of two or more nonzero idempotents. 
A set of primitive idempotents $c_1, c_2, \ldots c_k$ is called a {\em Jordan frame} if $c_1, \ldots c_k$ satisfy
\begin{equation}
\notag
c_i \circ c_j = 0 \ (i \neq j), \ \
c_i \circ c_i = c_i \ (i= 1 , \dots , k), \ \
\sum_{i=1}^k c_i = e.
\end{equation}
For $x \in  \mathbb{E}$, the {\em degree} of $x$ is the smallest integer $d$ such that the set $\{e,x,x^2,\ldots,x^d\}$ is linearly independent.
The {\em rank} of $\mathbb{E}$ is the maximum integer $r$ of the degree of $x$ over all $x \in \mathbb{E}$.
The following properties are known.

\begin{proposition}[Spectral theorem (cf. Theorem III.1.2 of \cite{Faraut1994})]
\label{prop:spectral}
Let $(\mathbb{E}, \circ)$ be a Euclidean Jordan algebra having rank $r$. 
For any $x \in \mathbb{E} $, there exist real numbers $\lambda_1 , \dots , \lambda_r$ and a Jordan frame $c_1 , \dots , c_r $ for which the following holds:
\begin{equation}
x = \sum_{i=1}^r \lambda_i c_i \notag.
\end{equation}
The numbers $\lambda_1 , \dots , \lambda_r$ are uniquely determined {\em eigenvalues} of $x$ (with their multiplicities).
Furthermore, ${\rm trace}(x) := \sum_{i=1}^r \lambda_i$, $\det(x) := \prod_{i=1}^r \lambda_i$.
\end{proposition}


\subsection{Symmetric cone}
\label{subsec:symmetric cone}

A proper cone is symmetric if it is self-dual and homogeneous. 
It is known that the set of squares $\mathcal{K} = \{ x^2 : x \in \mathbb{E} \}$ is the symmetric cone of  $\mathbb{E}$  (cf. Theorems III.2.1 and III.3.1 of \cite{Faraut1994}).
The following properties can be derived from the results in \cite{Faraut1994}, as in  Corollary 2.3 of \cite{Yoshise2007}:
\begin{proposition}
\label{prop:lambda-Jordan}
Let $x \in \mathbb{E}$ and let $\sum_{j=1}^r \lambda_j c_j$ be a decomposition of $x$ given by Propositoin \ref{prop:spectral}. Then
\begin{description}
\item[(i)]
$x \in \mathcal{K}$ if and only if  $\lambda_j \geq 0 \ (j=1,2,\ldots,r)$, 
\item[(ii)]
$x \in {\rm int} \hspace{0.75mm} \mathcal{K}$ if and only if $\lambda_j > 0 \ (j=1,2,\ldots,r)$.
\end{description}
\end{proposition}

From Propositions \ref{prop:spectral} and \ref{prop:lambda-Jordan} for any $x \in \mathbb{E}$, its projection $P_\mathcal{K} (x)$ onto $\mathcal{K}$ can be written as an operation to round all negative eigenvalues of $x$ to $0$, i.e., 
$P_\mathcal{K} (x) = \sum_{i=1}^r [\lambda_i]^+ c_i$, where $[\cdot]^+$ denotes the projection onto the nonnegative orthant. 
Using $P_\mathcal{K}$, we can decompose any $x \in \mathbb{E}$ as follows.
\begin{lemma}
\label{lemma:-1}
Let $x \in \mathbb{E}$, and $\mathcal{K}$ be the symmetric cone corresponding to $\mathbb{E}$.
Then, $x$ can be decomposed into $x = P_\mathcal{K} (x) - P_\mathcal{K} (-x)$.
\end{lemma}
\begin{proof}
From Propositoin \ref{prop:spectral}, let $x$ be given as $x = \sum_{i=1}^r \lambda_i c_i$. 
Let $I_1$ be the set of indices such that $\lambda_i \geq 0$ and $I_2$ be the set of indices such that $\lambda_i < 0$. Then, we have
$P_\mathcal{K} (x) = \sum_{i \in I_1} \lambda_i c_i$ and $P_\mathcal{K} (-x) =  \sum_{i \in I_2} - \lambda_i c_i$,
which implies that $x = \sum_{i \in I_1} \lambda_i c_i + \sum_{i \in I_2} \lambda_i c_i = P_\mathcal{K} (x) - P_\mathcal{K} (-x)$.
\end{proof}

A Euclidean Jordan algebra $(\mathbb{E} , \circ)$ is called {\em simple} if it cannot be written as any Cartesian product of non-zero Euclidean Jordan algebras.
If the Euclidean Jordan algebra $(\mathbb{E} , \circ)$ associated with a symmetric cone $\mathcal{K}$ is simple, then we say that $\mathcal{K}$ is {\em simple}.
In this paper, we will consider that $\mathcal{K}$ is given by a Cartesian product of $p$ simple symmetric cones $\mathcal{K}_\ell$,
$\mathcal{K} := \mathcal{K}_1 \times \dots \times \mathcal{K}_p$, whose rank and identity element are $r_\ell$ and $e_\ell$ $(\ell=1, \ldots, p)$.
The rank $r$ and the identity element of  $\mathcal{K}$ are given by
\begin{equation}
\label{eq:r and e}
r = \sum_{\ell=1}^p r_\ell, \ \ e = (e_1 , \dots , e_p).
\end{equation}
In what follows, $x_\ell$ stands for the $\ell$-th block element of  $x \in \mathcal{K}$, i.e., $x = (x_1, \dots, x_p) \in \mathcal{K}_1 \times \dots \times \mathcal{K}_p$.
For each $\ell=1, \cdots, p$, we define $\lambda_{\min}(x_\ell) := \min\{ \lambda_1, \cdots, \lambda_{r_\ell} \}$ where $\lambda_1, \cdots,  \lambda_{r_\ell}$ are eigenvalues of $x_\ell$.
The minimum eigenvalue $\lambda_{\min}(x)$ of $x \in \mathcal{K}$ is given by $\lambda_{\min}(x) = \mbox{min} \{ \lambda_{\min}(x_1), \cdots, \lambda_{\min}(x_p) \}$.

Next, we consider the {\em quadratic representation}  $Q_v(x)$ defined by $Q_v(x) := 2 v \circ ( v \circ x ) - v^2 \circ x$.
For the cone $\mathcal{K} = \mathcal{K}_1 \times \dots \times \mathcal{K}_p$, the quadratic representation $Q_v(x)$ of $x \in \mathcal{K}$ is denoted by $Q_v(x) = \left(Q_{v_1} (x_1) , \dots , Q_{v_p}(x_p) \right)$.
Letting $I_\ell$ be the identity operator of the Euclidean Jordan algebra $(\mathbb{E}_\ell, \circ_\ell)$ associated with the cone $\mathcal{K}_\ell$, we have $Q_{e_\ell} = I_\ell$ for $\ell=1, \ldots, p$.
The following properties can also be retrieved from the results in \cite{Faraut1994} as in Proposition 3 of \cite{Bruno2019}:

\begin{proposition}
\label{ptop:quadratic}
For any  $v \in {\rm int}\mathcal{K}$, $ Q_v (\mathcal{K}) = \mathcal{K}$.
\end{proposition}

It is also known that the following relations hold for the quadratic representation $Q_v$ and $\det (\cdot)$ \cite{Faraut1994}.
\begin{proposition}[cf. Proposition II.3.3 and III.4.2-(i), \cite{Faraut1994}] 
\label{q-det-relation}
For any $v,x \in \mathbb{E}$,
\begin{enumerate}
\item $\det Q_v(x) = \det (v)^2 \det (x)$,
\item $Q_{Q_v(x)} = Q_v Q_x Q_v$ (i.e., if $x =e$ then $Q_{v^2} = Q_v Q_v$) .
\end{enumerate}
\end{proposition}

More detailed descriptions, including concrete examples of symmetric cone optimization, can be found in, e.g., \cite{Faraut1994,Faybusovich1997,Schmieta2003,Alizadeh2012}. 
Here, we will use concrete examples of symmetric cones to explain the biliniear operation, the identity element, the inner product, the eigenvalues, the primitive idempotents, the projection on the symmetric cone and the quadratic representation on the cone.

\begin{example}[$\mathcal{K}$ is the semidefinite cone $\mathbb{S}^n_+$]
{\rm 
Let $\mathbb{S}^n$ be the set of symmetric matrices of $n \times n$.The semidefinite cone $\mathbb{S}^n_+$ is given by $\mathbb{S}^n_+ = \{ X \in \mathbb{S}^n : X \succeq O \}$.
For any symmetric matrices $X , Y \in \mathbb{S}^n$, define the bilinear operation $\circ$ and inner product as $X \circ Y = \frac{ XY + YX } {2}$ and $\langle X , Y \rangle  = \mbox{tr}(XY) = \sum_{i=1}^n \sum_{j=1}^n X_{ij} Y_{ij}$, respectively.
For any $X \in \mathbb{S}^n$, perform the eigenvalue decomposition and let $u_1 , \dots , u_n$ be the corresponding normalized eigenvectors for the eigenvalues $\lambda_1 , \dots , \lambda_n$: $X = \sum_{i=1}^n \lambda_i u_i u_i^T$.
The eigenvalues of $X$ in the Jordan algebra are $\lambda_1 , \dots , \lambda_n$ and the primitive idempotents are $c_1 = u_1 u_1^T , \dots , c_n = u_n u_n^T$, which implies that the rank of the semidefinite cone $\mathbb{S}^n_+$ is $r=n$.
The identity element is the identity matrix $I$ and the projection onto $\mathbb{S}^n_+$ is given by $P_{\mathbb{S}^n_+}(X) = \sum_{i=1}^n [ \lambda_i ]^+ u_i u_i^T$.
The quadratic representation of $V \in \mathbb{S}^n$ is given by $Q_V(X) = V X V$.
}
\end{example}
\begin{example}[$\mathcal{K}$ is the second-order cone $\mathbb{L}_n$]
{\rm 
The second order cone is given by $\mathbb{L}_n = \{ ( x_1, \bm{\bar{x}}^\top )^\top \in \mathbb{R}^n : x_1 \geq \| \bm{\bar{x}} \|_2\}$.
For any $x , y \in \mathbb{R}^n$, define the bilinear operation $\circ$ and the inner product as 
$x \circ y = ( x^\top y,  ( x_1\bm{\bar{y}}  + y_1\bm{\bar{x}})^\top )^\top$ and $\langle x , y \rangle = 2 \sum_{i=1}^n x_i y_i$, respectively.
For any $x  \in \mathbb{R}^n$, by the decomposition 
\begin{equation}
x = \left( x_1 + \| \bm{\bar{x}}  \|_2 \right) \begin{pmatrix} 1/2 \\ \frac{\bm{\bar{x}}}{2 \| \bm{\bar{x}} \|_2} \end{pmatrix}  + \left(  x_1 - \| \bm{\bar{x}} \|_2  \right) \begin{pmatrix} 1/2 \\ -\frac{\bm{\bar{x}}}{2 \| \bm{\bar{x}} \|_2} \end{pmatrix} , \notag
\end{equation}
we obtain the eigenvalues and the primitive idempotents as follows:\begin{equation}
\lambda_1 =  x_1 + \| \bm{\bar{x}}  \|_2  \ \ , \ \ \lambda_2 =x_1 - \| \bm{\bar{x}} \|_2  , \notag
\end{equation}
\begin{equation}
c_1 =
\begin{cases}
\frac{1}{2} (1, \frac{\bm{\bar{x}^\top}}{\| \bm{\bar{x}}\|_2} )^\top & \|\bm{\bar{x}}\|_2 \neq 0 \\
\\
\frac{1}{2} (1, z^\top )^\top & \|\bm{\bar{x}}\|_2 = 0 
\end{cases}
 \ \ ,\ \ 
c_2 = 
\begin{cases}
\frac{1}{2} (1, -\frac{\bm{\bar{x}^\top}}{\| \bm{\bar{x}} \|_2})^\top & \|\bm{\bar{x}}\|_2 \neq 0 \\
\\
\frac{1}{2} (1, -z^\top)^\top & \|\bm{\bar{x}}\|_2 = 0 
\end{cases}
 \ \ .\ \ 
\notag
\end{equation}
where $z \in \mathbb{R}^{n-1}$ is an arbitrary vector satisfying $\|z\|_2=1$.
The above implies that the rank of the second-order cone $\mathbb{L}_n$ is $r=2$.
The identity element is given by 
$e = (1, \bm{0}^\top )^\top \in \mathbb{R}^n$.
The projection $P_{\mathbb{L}_n}(x)$ onto $\mathbb{L}_n$ is given by
\begin{equation}
P_{\mathbb{L}_n}(x) = \left[ x_1 + \| \bm{\bar{x}}  \|_2 \right]^+ \begin{pmatrix} 1/2 \\ \frac{\bm{\bar{x}}}{2 \| \bm{\bar{x}} \|_2} \end{pmatrix}  + \left[  x_1 - \| \bm{\bar{x}} \|_2  \right]^+ \begin{pmatrix} 1/2 \\ -\frac{\bm{\bar{x}}}{2 \| \bm{\bar{x}} \|_2} \end{pmatrix} .
\notag
\end{equation}
Letting $I_{n-1}$ be the identity matrix of order $n-1$, the quadratic representation $Q_v(\cdot)$ of $v \in \mathbb{R}^n$ is as follows:
\begin{equation}
Q_v(x) = 
\begin{pmatrix}
\| v \|_2^2 & 2 v_1 \bm{\bar{v}}^T \\
2 v_1 \bm{\bar{v}} & \mbox{det}v I_{n-1} + 2 \bm{\bar{v}} \bm{\bar{v}}^T
\end{pmatrix} x.
\notag
\end{equation}
}
\end{example}

\subsection{Notation}
\label{sec: notation}
This subsection summarizes the \modifyFirst{notations} used in this paper.
For any $x,y \in \mathbb{E}$, we define the inner product  $\langle \cdot,\cdot \rangle$ and the norm $\|\cdot\|_{J}$ as $\langle x , y \rangle := {\rm trace}(x \circ y)$ and $\| x \|_J := \sqrt{ \langle x , x \rangle }$, respectively.
For any $x \in \mathbb{E}$ having decomposition  $x = \sum_{i=1}^r \lambda_i c_i$ as in  Proposition \ref{prop:spectral}, we also define $\| x \|_1 := |\lambda_1| + \dots + |\lambda_r|$, $\|x\|_\infty := \max \{|\lambda_1| , \dots , |\lambda_r| \}$.
For $x \in \mathcal{K}$, we obtain the following equivalent representations: $\|x\|_1 = \langle e , x \rangle$, $\|x\|_\infty = \lambda_{\max} (x)$.
The following is a list of other definitions and frequently used symbols in the paper.
\begin{itemize}
\item $d$: the dimension of the Euclidean space $\mathbb{E}$ corresponding to $\mathcal{K}$,
\item $F_{{\rm P}_{S_{\infty}}(\mathcal{A})}$: the feasible region of ${\rm P}_{S_{\infty}}(\mathcal{A})$,
\item $P_{\mathcal{A}} (\cdot) $: the projection map onto $\mbox{ker} \hspace{0.3mm} \mathcal{A}$,
\item $\mathcal{P}_\mathcal{K} (\cdot) $: the projection map onto $\mathcal{K}$,
\item $\lambda(x) \in \mathbb{R}^r$: an $r$-dimensional vector composed of the eigenvalues of $ x \in \mathcal{K}$,
\item $\lambda(x_\ell) \in \mathbb{R}^{r_\ell}$: an $r_\ell$-dimensional vector composed of the eigenvalues of $x_\ell \in \mathcal{K}_\ell$ ($\ell =1,\ldots,p$),
\item $c(x_\ell)_i \in \mathcal{K}_\ell$: the $i$-th primitive idempotent of $x_\ell \in \mathbb{E}_\ell$. When $\mathcal{K}$ is simple, it is abbreviated as $c_i$.
\item $\left[ \cdot \right]^+$: the projection map onto the nonnegative orthant, and
\item $\mathcal{A}^*(\cdot)$: the adjoint operator of the linear operator $\mathcal{A}(\cdot)$, i.e., $\langle \mathcal{A}(x) , y \rangle = \left \langle x , \mathcal{A}^* (y) \right\rangle$ for all $x \in \mathcal{K}$ and $y \in \mathbb{R}^m$.
\end{itemize}

\section{Extension of Roos's method to the symmetric cone problem}
\label{sec: extension}
\subsection{Outline of the extended method}
We focus on the feasibility of the following problem ${\rm P}_{S_{\infty}}(\mathcal{A})$, which is equivalent to  ${\rm P}(\mathcal{A})$: 
\begin{equation}
\begin{array}{lllllllll}
{\rm P}_{S_{\infty}}(\mathcal{A})&	&\mbox{find}   &x \in \mbox{int} \hspace{0.75mm} \mathcal{K}	&	&\mbox{s.t.}	&\mathcal{A} (x) = \bm{0},	&\|x\|_\infty \leq 1.
\end{array}
\notag
\end{equation}
The alternative problem ${\rm D}(\mathcal{A})$ of ${\rm P}(\mathcal{A})$ is 
\begin{equation}
\notag
\begin{array}{llllllllll}
{\rm D}(\mathcal{A})	&	&\mbox{find}   &y \in \mathcal{K}	&	&\mbox{s.t.}	&y \in \mbox{range} \hspace{0.3mm} \mathcal{A}^*,	&y \neq \bm{0},
\end{array}
\end{equation}
where  $\mbox{range} \hspace{0.3mm} \mathcal{A}^*$ is the orthogonal complement of $\mbox{ker} \hspace{0.3mm} \mathcal{A}$.
As we mentioned in section \ref{subsec:symmetric cone}, 
we assume that $\mathcal{K}$ is given by a Cartesian product of $p$ simple symmetric cones $\mathcal{K}_\ell (\ell=1, \ldots, p)$, i.e., $\mathcal{K} = \mathcal{K}_1 \times \dots \times \mathcal{K}_p $.
In our method, the upper bound for the sum of eigenvalues of a feasible solution of ${\rm P}_{S_{\infty}}(\mathcal{A})$ plays a key role, whereas the existing work focuses on the volume of the set of the feasible region \cite{Bruno2019} or the condition number of a feasible solution~\cite{Pena2017}.
Before describing the theoretical results, let us outline the proposed algorithm when $\mathcal{K}$ is simple. The algorithm repeats two steps: 
(i). find a cut for ${\rm P}_{S_{\infty}}(\mathcal{A})$,  (ii) scale the problem to an isomorphic problem equivalent to ${\rm P}_{S_{\infty}}(\mathcal{A})$ such that the region narrowed by the cut is efficiently explored.
Given a feasible solution $x$ of ${\rm P}_{S_{\infty}}(\mathcal{A})$ and a constant $0<\xi<1$, our method first searches for a Jordan frame $\{ c_1 , \dots, c_r \}$ such that the following is satisfied:
\begin{equation}
\notag
\langle c_i , x \rangle  \leq  \xi  \  ( i \in H), \ \  \langle c_i , x \rangle \leq 1 \ ( i \notin H),
\end{equation}
where $H \subseteq \{1, \dots,r\}$ and $|H|>0$.
In this case, instead of ${\rm P}_{S_{\infty}}(\mathcal{A})$, we may consider ${\rm P}^{\rm Cut}_{S_{\infty}}(\mathcal{A})$ as follows:
\begin{equation}
\begin{array}{lllllllllllllll}
{\rm P}^{\rm Cut}_{S_{\infty}}(\mathcal{A})	&
	&\mbox{find}	&x \in \mbox{int} \hspace{0.75mm} \mathcal{K}
	&&\mbox{s.t.}	&\langle c_i , x \rangle  \leq  \xi  \  ( i \in H), \ \  \langle c_i , x \rangle \leq 1 \ ( i \notin H), \\
	&&&&
	&&\mathcal{A} (x) = \bm{0}, \ \|x\|_\infty \leq 1.
\end{array}
\notag
\end{equation}
Here, we define the set $SR^{\rm Cut} = \{x \in \mathbb{E} : x \in \mbox{int} \hspace{0.75mm} \mathcal{K}, \ \|x\|_\infty \leq 1, \ \langle c_i, x \rangle \leq \xi \ (i \in H), \ \langle c_i, x \rangle \leq 1 \  i \notin H) \}$ as the search range for the solutions of the problem ${\rm P}^{\rm Cut}_{S_{\infty}}(\mathcal{A})$.
The proposed method then creates a problem equivalent and isomorphic to ${\rm P}_{S_{\infty}}(\mathcal{A})$ such that $SR^{\rm Cut}$, the region narrowed by the cut, can be searched efficiently.
Such a problem is obtained as follows:
\begin{equation}
\begin{array}{llllllll}
{\rm P}_{S_{\infty}}(\mathcal{A}Q_g)	&	&\mbox{find}   &\bar{x} \in \mbox{int} \hspace{0.75mm} \mathcal{K}	&	&\mbox{s.t.} &  \mathcal{A}Q_g (\bar{x}) = \bm{0},	&\|\bar{x}\|_\infty \leq 1, \notag
\end{array}
\end{equation}
where $g$ is given by $g = \sqrt{\xi} \sum_{i \in H} c_i + \sum_{i \notin H} c_i \in {\rm int} \hspace{0.75mm} \mathcal{K}$ for which $e = Q_{g^{-1}} (u)$ holds for $u = \sum_{i \in H} \xi c_i + \sum_{i \notin H} c_i$.

In the succeeding sections, we explain how the cut for ${\rm P}_{S_{\infty}}(\mathcal{A})$ is obtained from some $v \in \mbox{range} \hspace{0.3mm} \mathcal{A}^*$; we also explain the scaling method for the problem in detail.
To simplify our discussion, we will assume that $\mathcal {K}$ is simple, i.e., $p=1$, in section \ref{subsec:simple}. Then, in section \ref{subsec:non-simple}, we will generalize our discussion to the case of $p \geq 2$.

\subsection{Simple symmetric cone case}
\label{subsec:simple}

Let us consider the case where $\mathcal {K}$ is simple. It is obvious that, for any feasible solution $x$ of ${\rm P}_{S_{\infty}}(\mathcal{A})$, the constraint $\|x\|_\infty \leq 1$ implies that $\langle e , x\rangle \leq r$, since $x \in \mathcal{K}$. In Proposition \ref{prop:upper-bound}, we show that this bound may be improved as $\langle e , x\rangle <r$ by using a point $v \in \mbox{range} \hspace{0.3mm} \mathcal{A}^*\setminus \{0\}$. To prove Proposition \ref{prop:upper-bound}, we need the following Lemma \ref{lemma:0} and Proposition \ref{prop:p-d-relation}.
\begin{lemma}
\label{lemma:0}
Let $(\mathbb{E}, \circ)$ be a Euclidean Jordan algebra with the \modifySecond{associated} symmetric cone $\mathcal{K}$.
For any $y \in \mathbb{E}$, the following equation holds:
\begin{equation}
\max_{\bm{0} \leq \lambda(x) \leq \bm{1}} \langle y ,x \rangle  =  \left \langle \mathcal{P}_\mathcal{K} \left( y \right)  ,e \right \rangle . \notag
\end{equation}
\end{lemma}
\begin{proof}
Using the decomposition $y = \sum_{i=1}^r \lambda_i c_i$ obtained by Proposition \ref{prop:spectral}, we see that 
\begin{equation}
\max_{\bm{0} \leq \lambda(x) \leq \bm{1}} \langle y ,x \rangle  = \max_{\bm{0} \leq \lambda(x) \leq \bm{1}} \left \langle \sum_{i=1}^r \lambda_i c_i , x \right \rangle = \max_{\bm{0} \leq \lambda(x) \leq \bm{1}}  \sum_{i=1}^r \lambda_i  \left \langle c_i , x \right \rangle . \label{lemma:0-eq1}
\end{equation}
Noting that $x \in \mathcal{K}, e-x \in \mathcal{K}$ from $\bm{0} \leq \lambda(x) \leq \bm{1}$, since $c_i \in \mathcal{K}$ is primitive idempotent, we find that 
$\langle c_i,x \rangle \geq 0$ and $\langle c_i, e-x \rangle \geq 0$, which implies that $0 \leq \langle c_i,x \rangle \leq 1$.
Thus, letting $I_1$ be the set of indices for which $\lambda_i \leq 0$ and $I_2$ be the set of indices for which $\lambda_i > 0$, if there exists an $x$ satisfying 
\begin{equation}
\langle c_i,x \rangle =
\begin{cases}
0 & i \in I_1 \\
1 & i \in I_2
\end{cases}, 
\label{lemma:0-eq2}
\end{equation}
then such an $x$ is an optimal solution of (\ref{lemma:0-eq1}). In fact, if we define $x^* = \sum_{i \in I_2} c_i$, then by the dedfinition of the Jordan frame, $x^*$ satisfies (\ref{lemma:0-eq2}) and $\bm{0} \leq \lambda(x) \leq \bm{1}$ and becomes an optimal solution of (\ref{lemma:0-eq1}). In this case, the optimal value of (\ref{lemma:0-eq1}) turns out to be
\begin{equation}
\max_{\bm{0} \leq \lambda(x) \leq \bm{1}}  \sum_{i=1}^r \lambda_i  \left \langle c_i , x \right \rangle = \sum_{i=1}^r \lambda_i  \left \langle c_i , x^* \right \rangle 
= \sum_{i \in I_2} \lambda_i 
= \sum_{i=1}^r [\lambda_i]^+
= \langle \mathcal{P}_\mathcal{K} (y) , e \rangle. 
\notag
\end{equation}
\end{proof}

\begin{proposition}
\label{prop:p-d-relation} 
Let $(\mathbb{E}, \circ)$ be a Euclidean Jordan Algebra with the corresponding symmetric cone  $\mathcal{K}$.
For a given $c \in \mathbb{E}$, consider the problem
\begin{equation}
\notag
\begin{array}{llllll}
\mbox{\rm max}	&\langle c , x \rangle 	&	&\mbox{\rm s.t}	&\mathcal{A}(x) = \bm{0},	&\bm{0} \leq \lambda (x) \leq \bm{1}  \notag.
\end{array}
\end{equation}
The dual problem of the above is 
\begin{equation}
\notag
\begin{array}{lllll}
\mbox{\rm min}	&\left \langle \mathcal{P}_\mathcal{K} \left( c - u  \right) , e \right \rangle	&	&\mbox{\rm s.t}	&u \in \mbox{\rm range} \hspace{0.3mm} \mathcal{A}^*. \notag
\end{array}
\end{equation}
\end{proposition}
\begin{proof}
Define the Lagrangian function $L(x,w)$ as $L(x,w) :=   \langle c , x \rangle -  w^\top \mathcal{A}(x)$ where $w \in \mathbb{R}^m$ is the Lagrange multiplier. 
\modifyThird{
Supoose that $x^*$ is an optimal sotution of the primal problem. 
Then, for any $w \in \mathbb{R}^m$,  we have $\langle c,x^* \rangle = L(w,x^*) \leq \max_{\bm{0} \leq \lambda(x) \leq \bm{1}} L(w,x)$, and hence,}
\begin{align*}
\modifyThird{\langle c, x^* \rangle} \leq \min_w \max_{\bm{0} \leq \lambda(x) \leq \bm{1}} L(x,w)
&=  \min_w \max_{\bm{0} \leq \lambda(x) \leq \bm{1}} \{  \langle c , x \rangle - \langle \mathcal{A}^*(w) ,x \rangle \}  \\
&=  \min_w \max_{\bm{0} \leq \lambda(x) \leq \bm{1}} \{  \langle c -  \mathcal{A}^*(w) ,x \rangle \}  \\
&=  \min_w   \left \langle \mathcal{P}_\mathcal{K} \left( c -  \mathcal{A}^*(w) \right)  ,e \right \rangle \ \ \ \mbox{(by  lemma \ref{lemma:0})}  \\
&= \min_{u \in \mbox{\footnotesize range} \mathcal{A}^*}  \langle \mathcal{P}_\mathcal{K} \left( c - u \right) , e \rangle.
\end{align*}
\end{proof}


\begin{proposition}
\label{prop:upper-bound}

Suppose that $v \in \mbox{\rm range} \hspace{0.3mm} \mathcal{A}^*$ is given by $v = \sum_{i=1}^r \lambda_i c_i$ as in Proposition \ref{prop:spectral}. 
For each $i \in \{ 1, \dots , r \}$ and $\alpha \in \mathbb{R}$, define $q_i(\alpha) := \left[    1-\alpha \lambda_i  \right]^+ + \sum_{j \neq i }^r \left[ -  \alpha \lambda_j  \right]^+$.
Then, the following relations hold for any $x \in F_{{\rm P}_{S_{\infty}}(\mathcal{A})}$ and $i \in \{ 1, \dots , r \}$\modifySecond{:}
\begin{equation}
\label{eq:bound_q}
\langle c_i , x \rangle \leq  
\min_{\alpha \in \mathbb{R}} q_i(\alpha) = 
\begin{cases}
\min \left\{  1 ,  \left \langle  e , \mathcal{P}_{\mathcal{K}} \left( -  \frac{1}{\lambda_i} v  \right) \right \rangle \right\} & \mbox{if $\lambda_i \neq 0$}, \\
1 & \mbox{if $\lambda_i = 0$}\modifySecond{.}
\end{cases}
\end{equation}
\end{proposition}

\begin{proof}

For each $i \in \{ 1, 2, \dots , r \}$, we have 
\begin{equation}
\mathcal{P}_\mathcal{K} \left( c_i -  \alpha v \right) = \mathcal{P}_\mathcal{K} \left( c_i - \alpha \sum_{j=1}^r \lambda_j c_j \right) = \mathcal{P}_\mathcal{K} \left( (1-\alpha \lambda_i ) c_i -   \sum_{j \neq i }^r \alpha \lambda_j c_j  \right), \notag
\end{equation}
and hence, 
\begin{equation}
\label{eq:q_i(alpha)}
\left \langle \mathcal{P}_\mathcal{K} \left( c_i -  \alpha v \right) , e \right \rangle
=
\left \langle \mathcal{P}_\mathcal{K} \left(    (1-\alpha \lambda_i ) c_i -   \sum_{j \neq i }^r \alpha \lambda_j c_j   \right) , \sum_{i=1}^r c_i \right \rangle 
= \left[    1-\alpha \lambda_i  \right]^+ + \sum_{j \neq i }^r \left[ -  \alpha \lambda_j  \right]^+  = q_i(\alpha) .
\end{equation}

Note that, since $q_i(\alpha)$ is a piece-wise linear convex function, if $\lambda_i = 0$, it attains the minimum at $\alpha = 0$ with $q_i(0) = 1$, and if $\lambda_i \neq 0$, it attains the minimum at $\alpha = 0$ with $q_i(0) = 1$ or at $\alpha =  \frac{1}{\lambda_i}$ with
\begin{equation}
\notag
q \left( \frac{1}{\lambda_i} \right) = \sum_{j \neq i}^r \left[ - \frac{\lambda_j}{\lambda_i}  \right]^+ 
= \sum_{j=1}^r \left[ - \frac{\lambda_j}{\lambda_i}  \right]^+ 
= \left \langle e , \mathcal{P}_{\mathcal{K}} \left( -  \frac{1}{\lambda_i} v  \right) \right \rangle.
\end{equation}
Thus, we obtain equivalence in (\ref{eq:bound_q}).
Since $\alpha v \in \mbox{range} \hspace{0.3mm} \mathcal{A}^*$ for all $\alpha \in \mathbb{R}$, 
for each $i \in \{ 1, \dots , r \}$, Proposition \ref{prop:p-d-relation} and  (\ref{eq:q_i(alpha)})  ensure that 
$\langle c_i , x \rangle \leq  \left \langle \mathcal{P}_\mathcal{K} \left( c_i -  \alpha v \right) , e \right \rangle = q_i(\alpha)$
for all $\alpha \in \mathbb{R}$, which implies the inequality in (\ref{eq:bound_q}).
\end{proof}

Since $\sum_{i=1}^r c_i = e$ holds, Proposition \ref{prop:upper-bound} allows us to compute upper bounds for the sum of eigenvalues of $x$. The following proposition gives us information about indices whose upper bound for $\langle c_i , x \rangle$ in Proposition \ref{prop:upper-bound} is less than 1.
\begin{proposition}
\label{prop:cut-prb}
Suppose that $v \in \mbox{\rm range} \hspace{0.3mm} \mathcal{A}^*$ is given by $v = \sum_{i=1}^r \lambda_i c_i$ as in Proposition \ref{prop:spectral}.
If $v$ satisfies $\left \langle  e , \mathcal{P}_{\mathcal{K}} \left( -  \frac{1}{\lambda_i} v  \right) \right \rangle = \xi < 1$
for some $\xi < 1$ and for some $i \in \{ 1, \dots , r \}$ for which $\lambda_i \neq 0$ holds, then $\lambda_i$ has the  same sign as $\langle e ,v \rangle$.
\end{proposition}

\begin{proof}
First, we consider the case where  $\lambda_i > 0$.
Since the assumption implies that $\langle e , \mathcal{P}_{\mathcal{K}} (-v) \rangle = \lambda_i \xi$,
we have 
\begin{equation}
\notag
\langle e , v \rangle = \langle e , \mathcal{P}_{\mathcal{K}} (v)  \rangle - \langle e , \mathcal{P}_{\mathcal{K}} (-v)  \rangle \\
= \langle e , \mathcal{P}_{\mathcal{K}} (v)  \rangle -  \lambda_i \xi \\
\geq \lambda_i (1-\xi) > 0,
\end{equation}
where the first equality comes from Lemma \ref{lemma:-1}.

For the case where  $\lambda_i < 0$, 
since the assumption also implies that \modifyFirst{$\langle e , \mathcal{P}_{\mathcal{K}} (v) \rangle = - \lambda_i \xi$, }
we have 
\begin{equation}
\langle e , v \rangle = \langle e , \mathcal{P}_{\mathcal{K}} (v)  \rangle - \langle e , \mathcal{P}_{\mathcal{K}} (-v)  \rangle \\
= - \lambda_i \xi   - \langle e , \mathcal{P}_{\mathcal{K}} (-v)  \rangle \\
\leq - \lambda_i \xi   - ( -\lambda_i) \\
= (1-\xi) \lambda_i < 0.
\notag
\end{equation}
This completes the proof.
\end{proof}

The above two propositions imply that,
for any $v \in \mbox{\rm range} \hspace{0.3mm} \mathcal{A}^*$ with $v = \sum_{i=1}^r \lambda_i c_i$,
if we compute $\langle c_i, x \rangle$ according to Proposition \ref{prop:upper-bound} for $i \in \{ 1, \dots , r \}$ having the same sign as the one of $\langle e , v \rangle$, we obtain an upper bound for the sum of eigenvalues of $x$ over the set $F_{{\rm P}_{S_{\infty}}(\mathcal{A})}$.
The following proposition concerns the scaling method of problem ${\rm P}_{S_{\infty}}(\mathcal{A})$ when we find such a $v \in \mbox{\rm range} \hspace{0.3mm} \mathcal{A}^*$.
\begin{proposition}
\label{prop:scaling}
\modifyThird{Let $H \subseteq \{1 , \dots r \}$ be a nonempty set, $c_1 , \dots , c_r $ be a Jordan frame, and $\xi$ be a real number satisfying $0 < \xi < 1$.}
Let us define $g \in \mbox{\rm int} \hspace{0.75mm} \mathcal{K}$ as
\begin{equation}
\label{eq:d}
g := \sqrt{\xi} \sum_{h \in H} c_h + \sum_{h \notin H} c_h \ \ \ \ \mbox{i.e.,} \ \ \ g^{-1} = \frac{1}{\sqrt{\xi}} \sum_{h \in H} c_h + \sum_{h \notin H} c_h.
\end{equation}

For the two sets $SR^{\rm Cut} = \{x \in \mathbb{E} : x \in \mbox{{\rm int}} \hspace{0.75mm} \mathcal{K}, \ \|x\|_\infty \leq 1, \ \langle c_i, x \rangle \leq \xi \ (i \in H), \ \langle c_i, x \rangle \leq 1 \  (i \notin H) \}$ and, 
$SR^{\rm Scaled} = \{ \bar{x} \in \mathbb{E} : \bar{x} \in \mbox{{\rm int}} \hspace{0.75mm} \mathcal{K}, \  \| \bar{x} \|_\infty \leq 1 \}$, $Q_g( SR^{\rm Scaled} ) \subseteq SR^{\rm Cut}$ holds.

\end{proposition}
\begin{proof}
\modifySecond{
Let $\bar{x}$ be an arbitrary point of $SR^{\rm Scaled}$. It suffices to show that (i) $Q_g(\bar{x}) \in \mbox{{\rm int}} \hspace{0.75mm} \mathcal{K}$, (ii) $\| Q_g(\bar{x}) \|_\infty \leq 1$, (iii) $\langle c_i, Q_g(\bar{x}) \rangle \leq \xi \ (i \in H)$ and (iv) $\langle c_i, Q_g(\bar{x}) \rangle \leq 1 \ (i \notin H)$ hold. \\
}
\medskip \\
\noindent (i):
Let us show that $Q_g(\bar{x}) \in \mbox{{\rm int}} \hspace{0.75mm} \mathcal{K}$.
Since $g$ and $\bar{x}$ lie in the set $\mbox{{\rm int}} \hspace{0.75mm} \mathcal{K}$, from Propositions \ref{ptop:quadratic} and \ref{q-det-relation}, we see that 
\begin{equation}
Q_g(\bar{x}) \in \mathcal{K}, \ \ \det Q_g(\bar{x}) = \det (g)^2 \det (\bar{x}) > 0, \notag 
\end{equation}
which implies that $Q_g(\bar{x}) \in \mbox{{\rm int}} \hspace{0.75mm} \mathcal{K}$. \\
\medskip \\
\noindent (ii):
Next let us show that $\| Q_g(\bar{x}) \|_\infty \leq 1$.
Since $\bar{x} \in SR^{\rm Scaled}$, we see that $\bar{x} \in \mbox{{\rm int}} \hspace{0.75mm} \mathcal{K}$, $\| \bar{x} \|_\infty \leq 1$ and hence  $e-\bar{x} \in \mathcal{K}$.
Since $g \in \mbox{{\rm int}} \hspace{0.75mm} \mathcal{K}$, Proposition \ref{ptop:quadratic} guarantees that 
\begin{equation}
\label{eq:Q-g}
Q_g(e-\bar{x}) \in \mathcal{K}.
\end{equation}
By the definition (\ref{eq:d}) of $g$,  the following equations hold for $c_1 , \dots , c_r$:
\begin{align*}
\mbox{For any $i \in H$,} \ \ \ \
Q_g (c_i) &= 2 g \circ ( g \circ c_i ) - (g \circ g ) \circ c_i \\
&= 2 g \circ \sqrt{\xi} c_i - \left( \xi \sum_{h \in H} c_h + \sum_{h \notin H} c_h \right) \circ c_i \\
&= 2 \xi c_i - \xi c_i  = \xi c_i.
\end{align*}
\begin{align*}
\mbox{For any $ i \notin H$,} \ \ \ \
Q_g (c_i) &= 2 g \circ ( g \circ c_i ) - (g \circ g ) \circ c_i \\
&= 2 g \circ c_i - \left( \xi \sum_{h \in H} c_h + \sum_{h \notin H} c_h \right) \circ c_i \\
&= 2c_i - c_i = c_i.
\end{align*}
Thus, we obtain $Q_g(e) = \xi \sum_{i \in H} c_i + \sum_{i \notin H} c_i$.
Combining this with the facts $c_i \in \mathcal{K}$ and $(1-\xi)>0$  and (\ref{eq:Q-g}), we have
\begin{align*}
\mathcal{K} \ni  (1-\xi) \sum_{i \in H} c_i + Q_g(e - \bar{x}) 
&=  (1-\xi) \sum_{i \in H} c_i + Q_g(e) -  \modifyThird{Q_g(\bar{x})} \\
&= (1-\xi) \sum_{i \in H} c_i + \left( \xi \sum_{i \in H} c_i + \sum_{i \notin H} c_i \right) -  \modifyThird{Q_g(\bar{x})} \\
&=  e-Q_g(\bar{x}).
\end{align*}
Since we have shown that  $Q_g(\bar{x}) \in \mbox{{\rm int}} \hspace{0.75mm} \mathcal{K}$, we can conclude that $\|Q_g(\bar{x})\|_\infty \leq 1$. \\
\medskip \\
\noindent \modifySecond{(iii) and (iv)}: 
Finally, we compute an upper bound for the value $\langle Q_g(\bar{x}) , c_i \rangle$ over the set $SR^{\rm Scaled}$. It follows from $c_i \in \mathcal{K}$ and (\ref{eq:Q-g}) that $\langle Q_g(e - \bar{x}) , c_i \rangle \geq 0$, i.e., $\langle Q_g(e) , c_i \rangle \geq \langle Q_g(\bar{x}) , c_i \rangle$ holds. Since we have shown that $Q_g(e) = \xi \sum_{i \in H} c_i + \sum_{i \notin H} c_i$, this implies $\langle Q_g(\bar{x}) , c_i \rangle \leq \xi$ holds if $i \in H$ and $\langle Q_g(\bar{x}) , c_i \rangle \leq 1$ holds if $i \notin H$.
\end{proof}

\modifyThird{
Note that Proposition \ref{prop:scaling} implies that if a cut is obtained for ${\rm P}_{S_{\infty}}(\mathcal{A})$ based on Proposition \ref{prop:upper-bound},  we can expect a more efficient search for solutions to problem ${\rm P}_{S_{\infty}}(\mathcal{A}Q_g)$ rather than trying to solve problem ${\rm P}_{S_{\infty}}(\mathcal{A})$.
}
\subsection{Non-simple symmetric cone case}
\label{subsec:non-simple}

In this section, we consider the case where the symmetric cone is not simple. 
Propositions \ref{prop:upper-bound-direct} and  \ref{prop:cut-prb-direct} are extensions of Proposition \ref{prop:upper-bound} and \ref{prop:cut-prb}, respectively.

\begin{proposition}
\label{prop:upper-bound-direct}
Suppose that, for any $v \in \mbox{\rm range} \hspace{0.3mm} \mathcal{A}^*$, the $\ell$-th block element $v_\ell$ of $v \in \mathbb{E}$ is decomposed into $v_\ell = \sum_{i=1}^{r_\ell} {\lambda(v_\ell)}_i {c(v_\ell)}_i$ as in Proposition \ref{prop:spectral}.
For each $\ell \in \{1,\ldots,p \}$ and $i \in \{1,\ldots,r_p\}$, define
\begin{equation}
\label{eq:q_ell,i(alpha)}
q_{\ell,i}(\alpha) := \left[ 1 - \alpha {\lambda(v_\ell)}_i \right]^+ + \sum_{k \neq i}^{r_\ell} \left[ -\alpha {\lambda(v_\ell)}_k \right]^+ + \sum_{j \neq \ell}^p \sum_{k=1}^{r_j} \left[  -\alpha {\lambda(v_j)}_k \right]^+ .
\end{equation}
Then, the following relations hold for any feasible solution $x$ of ${\rm P}_{S_{\infty}}(\mathcal{A})$, $\ell \in \{1,\ldots,p \}$ and $i \in \{1,\ldots,r_p\}$.
\begin{equation}
\label{eq:upper-bound-direct1}
\langle {c(v_\ell)}_i , x_\ell \rangle \leq  \min_{\alpha \in \mathbb{R} } q_{\ell,i}(\alpha) =
\begin{cases}
\min \left\{  1 ,  \left \langle  e , \mathcal{P}_{\mathcal{K}} \left( -  \frac{1}{{\lambda(v_\ell)}_i} v  \right) \right \rangle \right\} & \mbox{if ${\lambda(v_\ell)}_i \neq 0$}, \\
1 & \mbox{if ${\lambda(v_\ell)}_i = 0$} 
\end{cases}
.
\end{equation}
\end{proposition}

\begin{proof}
Let $c \in \mathbb {E}$ be an element whose $\ell$-th block element is $c_\ell = c (v_\ell)_i$ and other block elements take $0$.
For any real number $\alpha \in \mathbb{R}$, Proposition \ref {prop:p-d-relation} ensures that
\begin{align}
\label{eq:upper-bound-direct2}
\langle {c(v_\ell)}_i , x_\ell \rangle = \langle c,x \rangle &\leq \left \langle \mathcal{P}_\mathcal{K} \left( c -  \alpha v \right) , e \right \rangle \notag \\
&= \left \langle \mathcal{P}_{\mathcal{K}_\ell} \left( {c(v_\ell)}_i -  \alpha v_\ell \right) , e_\ell \right \rangle + \sum_{j \neq \ell}^p \left \langle \mathcal{P}_{\mathcal{K}_j} \left( -  \alpha v_j \right) , e_j \right \rangle \notag \\
&= \left[ 1 - \alpha {\lambda(v_\ell)}_i \right]^+ + \sum_{k \neq i}^{r_\ell} \left[ -\alpha {\lambda(v_\ell)}_k \right]^+ + \sum_{j \neq \ell}^p \sum_{k=1}^{r_j} \left[  -\alpha {\lambda(v_j)}_k \right]^+  = q_{\ell,i}(\alpha). 
\end{align}
We obtain (\ref{eq:upper-bound-direct1}) by following a similar argument to the one used in the proof of Proposition \ref{prop:upper-bound}.
\end{proof}

The next proposition follows similarly to Proposition \ref{prop:cut-prb}, by noting that  $\langle e,  \mathcal{P}_\mathcal{K}(-v) \rangle = {\lambda(v_\ell)}_i \xi$ holds if ${\lambda(v_\ell)}_i >0$ and that $\langle e,  \mathcal{P}_\mathcal{K}(v) \rangle = - {\lambda(v_\ell)}_i \xi$ if  ${\lambda(v_\ell)}_i <0$.

\begin{proposition}
\label{prop:cut-prb-direct}
Suppose that, for any $v \in \mbox{\rm range} \hspace{0.3mm} \mathcal{A}^*$, each $\ell$-th block element $v_\ell$ of $v$ is decomposed into $v_\ell = \sum_{i=1}^{r_\ell} {\lambda(v_\ell)}_i {c(v_\ell)}_i$ as in Proposition \ref{prop:spectral}.
If $v$ satisfies
\begin{equation}
\begin{array}{lll}
\label{eq:if}
{\lambda(v_\ell)}_i \neq 0 &\mbox{and}& \left \langle  e , \mathcal{P}_{\mathcal{K}} \left( -  \frac{1}{{\lambda(v_\ell)}_i} v  \right) \right \rangle = \xi_\ell < 1 
\end{array}
\end{equation}
for some $\xi < 1$, $\ell \in \{1 , \dots, p \} $ and $ i \in \{1 , \dots , r_\ell \} $, 
then ${\lambda(v_\ell)}_i$ has the same sign as $\langle e , v \rangle$.
\end{proposition}

From Proposition \ref{prop:upper-bound-direct}, if we obtain $v \in \mbox{range} \hspace{0.3mm} \mathcal{A}^*$ satisfying  (\ref{eq:if})  for a block $\ell \in \{1 , \dots, p \} $ with an index $i \in \{1 , \dots r_\ell \} $, then the upper bound for the sum of the eigenvalues of any feasible solution $x$ of  ${\rm P}_{S_{\infty}}(\mathcal{A})$ is reduced by $\langle e , x \rangle \leq r-1 + \xi_\ell < r$.
In this case, as described below, we can find a scaling such that the sum of eigenvalues of any feasible solution of  ${\rm P}_{S_\infty}(\mathcal{A})$ is bounded by $r$.
Let $H_\ell$ be the set of indices $i$ satisfying (\ref{eq:if}) for each block $\ell$.
According to Proposition \ref{prop:scaling}, set $g_\ell = \sqrt{\xi_\ell} \sum_{h \in H_\ell} {c(v_\ell)}_h + \sum_{h \notin H_\ell} {c(v_\ell)}_h$ and define the linear operator $Q$ as follows:
\begin{equation}
Q_\ell := 
\begin{cases}
Q_{g_\ell} & \mbox{if } | H_\ell | \neq 0, \\
I_\ell &  \mbox{otherwise},
\end{cases}
\notag
\end{equation}
\begin{equation}
Q(\mathbb{E}_1 , \dots , \mathbb{E}_p) := \left( Q_1(\mathbb{E}_1) , \dots , Q_p(\mathbb{E}_p) \right), \notag
\end{equation}
where $I_\ell$ is the identity operator of the Euclidean Jordan algebra $\mathbb{E}_\ell$ associated with the symmetric cone $\mathcal{K}_\ell$.
From Proposition \ref{prop:scaling} and its proof, we can easily see that
\begin{equation}
\label{eq:prop3.4-1}
Q_{g_\ell^{-1}}(c_i) = \frac{1}{\xi}  c_i \ (i\in H_\ell), \ \ \ Q_{g_\ell^{-1}}(c_i) =c_i \ (i \not\in H_\ell),
\end{equation}
and the sum of eigenvalues of any feasible solution of  the scaled problem ${\rm P}_{S_\infty}(\mathcal{A}Q)$ is bounded by $\langle e,e \rangle = r = \sum_{\ell=1}^{p} r_\ell$.

\section{Basic procedure of the extended method}
\label{sec: basic procedure}
\subsection{Outline of the basic procedure}
In this section, we describe the details of our basic procedure. 
First, we introduce our stopping criteria and explain how to update $y^k$ when the the stopping criteria is not satisfied. Next, we show that the stopping criteria is satisfied within a finite number of iterations. Our stopping criteria is new and different from the ones used in \cite{Bruno2019,Pena2017}, while the method of updating $y^k$ is similar to the one used in \cite{Bruno2019} or in the von Neumann scheme of \cite{Pena2017}.
Algorithm \ref{basic procedure} is a full description of our basic procedure. 

\subsection{Termination conditions of the basic procedure}
For $z^k = P_\mathcal{A}(y^k)$, $v^k = y^k - z^k$ and a given $\xi \in (0,1)$, 
our basic procedure terminates when any of the following four cases occurs:
\begin{enumerate}
\item $z^k \in \mbox{int} \hspace{0.75mm} \mathcal{K} $ meaning that $z^k$ is a solution of ${\rm P}(\mathcal{A})$,
\item $z^k = \bm{0}$ meaning that $y^k$ is feasible for ${\rm D}(\mathcal{A})$,
\item $y^k - z^k \in \mathcal{K}$ and $y^k - z^k \neq \bm{0}$ meaning that $y^k - z^k$ is feasible for ${\rm D}(\mathcal{A})$, or
\item there exist $\ell \in \{ 1 , \dots , p\}$ and $i \in \{1 , \dots , r_\ell\}$ for which 
\begin{equation}
\begin{array}{lll}
{\lambda(v^k_\ell)}_i \neq 0 &\mbox{and}& \left \langle  e , \mathcal{P}_{\mathcal{K}} \left( -  \frac{1}{{\lambda(v^k_\ell)}_i} v^k  \right) \right \rangle = \xi_\ell \leq \xi < 1, \label{eq:terminate}
\end{array}
\end{equation}
meaning that $\langle e,x \rangle < r$ holds for any feasible solution $x$ of ${\rm P}_{S_\infty}(\mathcal{A})$ (see Proposition \ref{prop:upper-bound-direct}).
\end{enumerate}

Cases 1 and 2 are direct extensions of the cases in \cite{Chubanov2015}, while case 3 was proposed in \cite{Kitahara2018,Bruno2019}.
Case 3 helps us to determine the feasibility of ${\rm P}(\mathcal{A})$ efficiently,  while we have to decompose $y^k-z^k$ for checking it.
If the basic procedure ends with case 1, 2, or 3, the basic procedure returns a solution of ${\rm P}(\mathcal{A})$ or ${\rm D}(\mathcal{A})$ to the main algorithm. 
If the basic procedure ends with case 4, the basic procedure returns to the main algorithm $p$ index sets $H_1 , \dots , H_p$ each of which consists of indices $i$ satisfying (\ref{eq:terminate}) and the set of primitive idempotents $C_\ell = \{ {c(v^k_\ell)}_1, \dots ,{c(v^k_\ell)}_{r_\ell} \}$ of $v^k_\ell$ for each $\ell$.

\subsection{Update of the basic procedure}

The basic procedure updates $y^k \in \mbox{int} \hspace{0.75mm} \mathcal{K}$ with $\langle y^k  , e \rangle = 1$ so as to reduce the value of $\| z^k \|_J$. 
The following proposition is essentially the same as Proposition 13 in \cite{Bruno2019}, so we will omit its proof.

\begin{proposition}[cf. Proposition 13, \cite{Bruno2019}] 
\label{prop:update}
For $y^k \in \mbox{\rm int} \hspace{0.75mm} \mathcal{K}$ satisfying $\langle y^k  , e \rangle = 1$, let $z^k = P_\mathcal{A}(y^k)$.
If  $z^k \notin \mbox{\rm int} \hspace{0.75mm} \mathcal{K} $ and $z^k \neq \bm{0}$, then the following hold.
\begin{enumerate}
\item 
There exists $ c \in \mathcal{K}$ such that $\langle c , z^k \rangle =  \lambda_{\min} (z^k) \leq 0$, $\langle e , c \rangle=1$ and $\ c\in \mathcal{K}$.
\item 
For the above $c$, suppose that $P_\mathcal{A}(c) \neq \bm{0}$ and define
\begin{equation} 
\label{eq:bp2}
\alpha= \langle P_\mathcal{A}(c) , P_\mathcal{A}(c) - z^k \rangle  \|z^k-P_\mathcal{A}(c) \|^{-2}_J. 
\end{equation}
Then, $y^{k+1} := \alpha y^k + (1-\alpha) c$ satisfies $y^{k+1} \in \mbox{\rm int} \hspace{0.75mm} \mathcal{K}$, $\|y^{k+1}\|_{1,\infty} \geq 1/p$, $\langle y^{k+1} , e \rangle = 1$, and $z^{k+1} := P_\mathcal{A}(y^{k+1})$ satisfies $\| z^{k+1} \|^{-2}_J \geq \|z^k\|^{-2}_J + 1$.
\end{enumerate}
\end{proposition}
A method of accelerating the update of $y^k$ is provided in \cite{Roos2018}. 
For $\ell \in \{1,2,\ldots,p\}$, let $I_\ell := \{ i \in \{1,2,\ldots,r_\ell\} \mid \lambda_i (z^k_\ell) \leq 0\} $ and set $N = \sum_{\ell=1}^p | I_\ell | $. 
Define the $\ell$-th block element of $c \in \mathcal{K}$ as $c_\ell = \frac{1}{N} \sum_{i \in I_\ell} {c(z^k_\ell)}_i$.
Using $P_\mathcal{A} \left(  c \right) $, the acceleration method computes $\alpha$ by (\ref{eq:bp2}) so as to minimize the norm of $z^{k+1}$ and update $y$ by $y^{k+1} = \alpha y^k +   (1-\alpha)  c$.
We incorporate this method in the basic procedure of our computational experiment. 
As described in~\cite{Pena2017}, we can also use the smooth perceptron scheme~\cite{Soheili2012,Soheili2013} to update $y^k$ in the basic procedure.
As explained in the next section, using the smooth perceptron scheme significantly reduces the maximum number of iterations of the basic procedure.
A detailed description of our basic procedure is given in Appendix \ref{app:modified basic procedure}.

\subsection{Finite termination of the basic procedure}
\label{sec: finite termination of BP}
In this section, we show that the basic procedure terminates in a finite number of iterations.
To do so, we need to prove Lemma \ref{lemma:1} and Proposition \ref{prop:bp1}.
\begin{lemma}
\label{lemma:1}
Let $(\mathbb{E}, \circ)$ be a Euclidean Jordan algebra with the corresponding symmetric cone $\mathcal{K}$
given by the Cartesian product of $p$ simple symmetric cones.
For any $x \in \mathbb{E}$ and $y \in \mathcal{K}$, $[ \langle x , y \rangle ]^+ \leq \langle \mathcal{P}_\mathcal{K} (x)  , y \rangle$ holds.
\end{lemma}
\begin{proof}
Let $x \in \mathbb{E}$ and suppose that each $\ell$-th block element $x_\ell$ of $x$ is given by $x_\ell = \sum_{i=1}^{r_\ell} {\lambda(x_\ell)}_i {c(x_\ell)}_i$ as in Proposition \ref{prop:spectral}.
Then, we can see that 
\begin{align}
\left[  \langle x , y \rangle  \right]^+ &=  \left[   \sum_{\ell=1}^p \left \langle  \sum_{i=1}^{r_\ell} {\lambda(x_\ell)}_i {c(x_\ell)}_i , y_\ell \right \rangle  \right]^+ \notag \\
&=  \left[  \sum_{\ell=1}^p \left( \sum_{i=1}^{r_\ell} {\lambda(x_\ell)}_i \left \langle {c(x_\ell)}_i , y_\ell \right \rangle  \right)  \right]^+ \notag \\
&\leq \sum_{\ell=1}^p \sum_{i=1}^{r_\ell} \left[   {\lambda(x_\ell)}_i \left \langle {c(x_\ell)}_i , y_\ell \right \rangle  \right]^+ \notag \\
&= \sum_{\ell=1}^p \sum_{i=1}^{r_\ell} \left[  { \lambda(x_\ell)}_i \right]^+ \langle {c(x_\ell)}_i , y_\ell \rangle   
= \sum_{\ell=1}^p \left \langle \sum_{i=1}^{r_\ell} \left[  { \lambda(x_\ell)}_i \right]^+  {c(x_\ell)}_i , y_\ell \right \rangle 
= \left \langle \mathcal{P}_\mathcal{K} (x)  , y \right \rangle . \notag
\end{align}
where the inequality follows from the fact that ${c(x_\ell)}_1 , \dots , {c(x_\ell)}_{r_\ell}$, and $y_\ell$ lie in $\mathcal{K}_\ell$.
\end{proof}
\begin{proposition}
\label{prop:bp1}
For a given $y \in \mathcal{K}$, define $z = P_\mathcal{A} (y)$ and $v = y - z$.
Suppose that $v \neq 0$ and each $\ell$-th element $v_\ell$ is given by $v_\ell = \sum_{i=1}^{r_\ell} {\lambda(v_\ell)}_i {c(v_\ell)}_i$, as in Proposition \ref{prop:spectral}.
Then, for any $x \in F_{{\rm P}_{S_{\infty}}(\mathcal{A})}$, $\ell \in \{ 1, \dots , p \}$ and $ i \in \{ 1, \dots , r_\ell \}$, 
\begin{equation}
\label{eq:bp1-1}
\langle {c(v_\ell)}_i , x_\ell \rangle \leq \min_\alpha q_{\ell,i}(\alpha) \leq \frac{1}{ \left \langle y_\ell , {c(v_\ell)}_i \right \rangle} \|z\|_J 
\end{equation}
hold where $q_{\ell,i}(\alpha)$ is defined in (\ref{eq:q_ell,i(alpha)}).
\end{proposition}
\begin{proof}
The first inequality of (\ref{eq:bp1-1}) follows from (\ref{eq:upper-bound-direct2}) in the proof of Proposition \ref{prop:upper-bound-direct}. 
The second inequality is obtained by evaluating  $q_{\ell, i}(\alpha)$ at $\alpha = \frac{1}{ \langle y_\ell , {c(v_\ell)}_i \rangle}$, as follows:
\begin{align*}
q_{\ell,i} \left(  \frac{1}{ \langle y_\ell , {c(v_\ell)}_i \rangle} \right) 
&= \left[ 1 - \frac{1}{ \langle y_\ell , {c(v_\ell)}_i \rangle} {\lambda(v_\ell)}_i \right]^+ + \sum_{k \neq i}^{r_\ell} \left[ - \frac{1}{ \langle y_\ell , {c(v_\ell)}_i \rangle} {\lambda(v_\ell)}_k \right]^+ + \sum_{j \neq \ell}^p \sum_{k=1}^{r_j} \left[  -\frac{1}{ \langle y_\ell , {c(v_\ell)}_i \rangle} {\lambda(v_j)}_k \right]^+\\
&= \left[ 1 - \frac{\langle y_\ell -z_\ell , {c(v_\ell)}_i \rangle}{ \langle y_\ell , {c(v_\ell)}_i \rangle}    \right]^+ + \sum_{k \neq i}^{r_\ell} \left[ - \frac{\langle y_\ell -z_\ell , c(v_\ell)_k \rangle}{ \langle y_\ell , {c(v_\ell)}_i \rangle}   \right]^+ + \sum_{j \neq \ell}^p \sum_{k=1}^{r_j} \left[  -\frac{\langle y_j - z_j , {c(v_j)}_k \rangle}{ \langle y_\ell , {c(v_\ell)}_i \rangle} \right]^+ \\
& \ \ \ \  (\mbox{since} \ \lambda(v_\ell)_i = \langle v_\ell, c(v_\ell)_i \rangle \ \mbox{and} \ v_\ell = y_\ell -z_\ell ) \\
&= \left[ \frac{\langle z_\ell , {c(v_\ell)}_i \rangle }{ \langle y_\ell , {c(v_\ell)}_i \rangle} \right]^+ + \sum_{k \neq i}^{r_\ell}  \left[  \frac{ \langle z_\ell , {c(v_\ell)}_k \rangle - \langle y_\ell , {c(v_\ell)}_k \rangle}{ \langle y_\ell , {c(v_\ell)}_i \rangle}  \right]^+ + \sum_{j \neq \ell}^p \sum_{k=1}^{r_j} \left[  \frac{\langle z_j , {c(v_j)}_k \rangle - \langle y_j , {c(v_j)}_k \rangle}{ \langle y_\ell , {c(v_\ell)}_i \rangle} \right]^+ \\
&\leq \left[ \frac{\langle z_\ell , {c(v_\ell)}_i \rangle }{ \langle y_\ell , {c(v_\ell)}_i \rangle} \right]^+  + \sum_{k \neq i}^{r_\ell} \left[  \frac{\langle z_\ell , {c(v_\ell)}_k \rangle}{ \langle y_\ell , {c(v_\ell)}_i \rangle}  \right]^+ + \sum_{j \neq \ell}^p \sum_{k=1}^{r_j} \left[ \frac{\langle z_j , {c(v_j)}_k \rangle}{ \langle y_\ell , {c(v_\ell)}_i \rangle} \right]^+ \\
& \ \ \ \  (\mbox{since} \ y_\ell, c(v_\ell)_i \in \mathcal{K}_\ell \ \mbox{and then} \ \langle y_\ell, c(v_\ell)_i \rangle \geq 0) \\
&= \frac{1}{ \langle y_\ell , {c(v_\ell)}_i \rangle} \left( \sum_{k=1}^{r_\ell} \left[  \langle z_\ell , {c(v_\ell)}_k \rangle  \right]^+ + \sum_{j \neq \ell}^p \sum_{k=1}^{r_j} \left[ \langle z_j , {c(v_j)}_k \rangle \right]^+ \right) \\
&\leq \frac{1}{ \langle y_\ell , {c(v_\ell)}_i \rangle} \left(  \sum_{k=1}^{r_\ell}   \langle  \mathcal{P}_{{\mathcal{K}_\ell}}\left( z_\ell \right), {c(v_\ell)}_k \rangle + \sum_{j \neq \ell}^p \sum_{k=1}^{r_j} \langle  \mathcal{P}_{{\mathcal{K}_j}}\left( z_j \right), {c(v_j)}_k \rangle \right) \ \ \mbox{(by Lemma \ref{lemma:1})}  \\
&= \frac{1}{ \langle y_\ell , {c(v_\ell)}_i \rangle}  \left(  \langle  \mathcal{P}_{\mathcal{K}_\ell}\left( z_\ell \right), e_\ell \rangle  + \sum_{j \neq \ell}^p \langle  \mathcal{P}_{\mathcal{K}_j}\left( z_j \right), e_j \rangle \right) \\
&= \frac{\langle  \mathcal{P}_\mathcal{K}\left( z \right), e \rangle  }{ \langle y_\ell , {c(v_\ell)}_i \rangle}
= \frac{\| \mathcal{P}_{\mathcal{K}}\left( z \right) \|_1 }{ \langle y_\ell , {c(v_\ell)}_i \rangle}
\leq \frac{\| \mathcal{P}_{\mathcal{K}}\left( z \right) \|_J }{ \langle y_\ell , {c(v_\ell)}_i \rangle}
\leq \frac{\|  z  \|_J }{ \langle y_\ell , {c(v_\ell)}_i \rangle}.
\end{align*}
\end{proof}

\begin{proposition}
\label{prop:bp2}
Let $r_{\max} = \max \{ r_1 , \dots , r_p\}$.
The basic procedure (Algorithm \ref{basic procedure}) terminates in at most $\frac{p^2r_{\max}^2}{\xi^2} $ iterations. 
\end{proposition}
\begin{proof}
Suppose that $y^k$ is obtained at the $k$-th iteration of Algorithm \ref{basic procedure}. Proposition \ref{prop:update} implies that $\|y^k\|_{1,\infty} \geq \frac{1}{p}$ and an $\ell$-th block element exists for which $\langle y_\ell , e_\ell \rangle \geq \frac{1}{p}$ holds. 
Thus, by letting $v^k = y^k - z^k$ and the $\ell$-th block element $v_\ell^k$ of $v^k$ be $v_\ell^k = \sum_{i=1}^{r_\ell} {\lambda(v^k_\ell)}_i {c(v^k_\ell)}_i$ as in Proposition \ref{prop:spectral}, we have
\begin{equation}
\max_{i=1,\dots,r_\ell} \left \langle y_\ell^k , {c(v^k_\ell)}_i \right \rangle \geq (pr_\ell)^{-1}. \label{eq:bp2-1}
\end{equation}
Since Proposition \ref{prop:update} ensures that $\frac{1}{\|z^k\|^2_J}  \geq k$ holds at the $k$-th iteration,
by setting $k=\frac{p^2r_{\max}^2}{\xi^2}$, we see that $\xi \geq p r_{\max}  \|z^k\|_J$, and combining this with (\ref{eq:bp2-1}), we have
\begin{equation}
\notag
\xi \geq p r_{\max}  \|z^k\|_J \geq p r_\ell  \|z^k\|_J \geq \frac{1}{\max_{i=1,\dots,r_\ell} \left \langle y_\ell^k , {c(v_\ell)}_i \right \rangle} \| z^k \|_J . 
\end{equation}
The above inequality and Proposition \ref{prop:bp1} imply that for any $\ell \in \{1,\ldots,p \}$ and $i \in \{1,\ldots,r_p\}$,
\begin{equation}
\notag
\langle c(v^k_\ell)_i , x_\ell \rangle \leq  \min_\alpha q_{\ell,i} (\alpha)
\leq \frac{1}{ \langle y^k_\ell , {c(v^k_\ell)}_i \rangle} \|z^k\|_J \leq \xi. 
\end{equation}
From the equivalence in (\ref{eq:upper-bound-direct1}) and the setting $\xi \in (0,1)$, we conclude that Algorithm \ref{basic procedure} terminates in at most $\frac{p^2r_{\max}^2}{\xi^2} $ iterations by satisfying  (\ref{eq:terminate}) in the fourth termination condition at an $\ell$-th block and an index $i$.
\end{proof}

An upper bound for the number of iterations of Algorithm \ref{bp-alg-sp} using smooth perceptoron scheme can be found as follows.
\begin{proposition}
Let $r_{\max} = \max \{ r_1 , \dots , r_p\}$.
The basic procedure (Algorithm \ref{bp-alg-sp}) terminates in at most $\frac{2 \sqrt{2} p r_{\max}}{\xi}$ iterations.
\end{proposition}
\begin{proof}
From Proposition 6 in~\cite{Pena2017}, after $k \geq 1$ iterations,  we obtain the inequality $\|z^k\|_J^2 \leq \frac{8}{(k+1)^2}$.
Similarly to the previous proof of Proposition \ref{prop:bp2}, if $\xi \geq p r_{max} \|z^k\|_J$ holds, then Algorithm \ref{bp-alg-sp} terminates. 
Thus, $k \leq \frac{2 \sqrt{2} p r_{\max}}{\xi}$ holds for a given $k$ satisfying $\left( \frac{\xi}{p r_{max}} \right)^2 \leq \frac{8}{(k+1)^2}$.
\end{proof}

Here, we discuss the computational cost per iteration of Algorithm  \ref{basic procedure}.
At each iteration, the two most expensive operations are computing the spectral decomposition on line 5 and computing $P_{\mathcal{A}}(\cdot)$ on lines  24 and  26.
Let $C_{\ell}^{\rm sd}$ be the computational cost of the spectral decomposition of an element of $\mathcal{K}_\ell$.
For example, $C_{\ell}^{\rm sd}=\mathcal{O}(r_\ell^3)$ if $\mathcal{K}_\ell = \mathbb{S}^{r_\ell}_+$ and $C_{\ell}^{\rm sd}=\mathcal{O}(r_\ell)$ if $\mathcal{K}_\ell = \mathbb{L}_{r_\ell}$, where $\mathbb{L}_{r_\ell}$ denotes the $r_\ell$-dimensional second-order cone. 
Then, the cost $C^{\rm sd}$ of computing the spectral decomposition of an element of $\mathcal{K}$ is $C^{\rm sd} = \sum_{\ell=1}^p C_{\ell}^{\rm sd}$. Next, let us consider the computational cost of $P_{\mathcal{A}}(\cdot)$. 
Recall that $d$ is the dimension of the Euclidean space $\mathbb{E}$ corresponding to $\mathcal{K}$. 
As discussed in \cite{Bruno2019}, we can compute $P_{\mathcal{A}} = I - \mathcal{A}^*(\mathcal{A}\mathcal{A}^*)^{-1}\mathcal{A}$ by using the Cholesky decomposition of  $(\mathcal{A}\mathcal{A}^*)^{-1}$.
Suppose that $(\mathcal{A}\mathcal{A}^*)^{-1} = LL^*$, where $L$ is an $m \times m$ matrix and we store $L^*\mathcal{A}$ in the main algorithm.
Then, we can compute $P_{\mathcal{A}}(\cdot)$ on lines 24 and 26, which costs $\mathcal{O}(md)$.
The operation $u_\mu (\cdot) : \mathbb{E} \rightarrow \{ u \in \mathcal{K} \mid \langle u,e \rangle = 1\}$ in Algorithm \ref{bp-alg-sp} can be performed within the cost $C^{\rm sd}$~\cite{Soheili2013,Pena2017}.
\modifySecond{
From the above discussion and Proposition \ref{prop:bp2}, the total costs of Algorithm \ref{basic procedure} and Algorithm \ref{bp-alg-sp} are given by $\mathcal{O} \left( \frac{p^2 r_{\max}^2}{\xi^2}  \max ( C^{\rm sd} , md ) \right)$ and $\mathcal{O} \left( \frac{p r_{\max}}{\xi}  \max ( C^{\rm sd} , md ) \right)$, respectively.
}

\begin{algorithm}[H]
 \caption{Basic procedure (von Neumann scheme)}
 \label{basic procedure}
 \begin{algorithmic}[1]
 \renewcommand{\algorithmicrequire}{\textbf{Input: }}
 \renewcommand{\algorithmicensure}{\textbf{Output: }}
 \renewcommand{\stop}{\textbf{stop }}
 \renewcommand{\return}{\textbf{return }}

 \STATE \algorithmicrequire $P_\mathcal{A}$, $y^1 \in \mbox{int} \hspace{0.75mm} \mathcal{K}$ such that $\langle y^1 , e \rangle = 1$ and a constant $\xi$ such that $0 < \xi < 1$
 \STATE \algorithmicensure (i) a solution to ${\rm P}(\mathcal{A})$ or (ii) ${\rm D}(\mathcal{A})$ or (iii) a certificate that, for any feasible solution $x$ to ${\rm P}_{S_\infty}(\mathcal{A})$, $\langle e,x \rangle < r$
 \STATE initialization: $k \leftarrow 1, z^1 \leftarrow P_\mathcal{A}(y^1), v^1 \leftarrow y^1 - z^1, H_1 , \dots , H_p  = \emptyset$
 \WHILE{ $k \leq \frac{p^2 r_{\max}^2}{\xi^2}$}
 \STATE For every $\ell \in \{1 , \dots , p\}$, perform spectral decomposition: $z^k_\ell = \sum_{i=1}^{r_\ell} {\lambda(z_\ell^k)}_i {c(z_\ell^k)}_i$ and $v^k_\ell = \sum_{i=1}^{r_\ell} {\lambda(v_\ell^k)}_i {c(v^k_\ell)}_i$
 \IF {$z^k \in \mbox{int } \mathcal{K}$} 
 \STATE \stop basic procedure and \return $z^k$ (Output (i))
 \ELSIF { $z^k = 0$ or $v^k \in \mathcal{K} \setminus \{0\} $}
 \STATE \stop basic procedure and \return $y^k$ or $v^k$ (Output (ii))
 \ENDIF
 \IF {$\langle v^k , e \rangle  > 0$}
 \FOR {$\ell \in \{1 , \dots , p \}$}
 \STATE $I_\ell \leftarrow \left \{ i  \mid  {\lambda(v_\ell^k)}_i > 0 \right \}$ and then $H_\ell \leftarrow \left \{ i \in I_\ell | \left \langle e , \mathcal{P}_\mathcal{K} \left( -\frac{1}{{\lambda(v_\ell^k)}_i} \modifySecond{v^k} \right) \right \rangle \leq \xi \right \}$
 \ENDFOR
 \ELSE 
 \FOR {$\ell \in \{1 , \dots , p \}$}
 \STATE $I_\ell \leftarrow \left \{ i  \mid  {\lambda(v_\ell^k)}_i < 0 \right\}$ and then $H_\ell \leftarrow \left \{ i \in I_\ell | \left \langle e , \mathcal{P}_\mathcal{K} \left( -\frac{1}{{\lambda(v_\ell^k)}_i} \modifySecond{v^k} \right) \right \rangle \leq \xi \right \}$
 \ENDFOR 
 \ENDIF
 \IF { $ |H_1| + \dots + |H_p| > 0 $}
 \STATE For every $\ell \in \{1 , \dots , p\}$, let $ C_\ell $ be  $\{ {c(v_\ell^k)}_1 , \dots , {c(v_\ell^k)}_{r_\ell} \}$. 
 \STATE \stop basic procedure and \return $H_1 , \dots , H_p$ and $C_1 , \dots , C_p$  (Output (iii))
 \ENDIF
 \STATE Let $u$ be an idempotent such that $\langle e , u \rangle = 1$ and $\langle z^k , u \rangle = \lambda_{\min}(z^k)$
 \STATE $y^{k+1} \leftarrow \alpha y^k + (1-\alpha) u$, where $\alpha = \frac{ \langle P_\mathcal{A} (u) , P_\mathcal{A} (u) - z^k \rangle }{\|z^k-P_\mathcal{A} (u)\|^2_J} $
 \STATE $k \leftarrow k+1$ , $z^k \leftarrow P_\mathcal{A}(y^k)$ and $v^k \leftarrow y^k - z^k$
 \ENDWHILE
 \STATE \return basic procedure error
 \end{algorithmic} 
 \end{algorithm}

\section{Main algorithm of the extended method}
\label{sec: main algorithm}
\subsection{Outline of the main algorithm}
\label{sec:outline of main alg}

In what follows, for a given accuracy $\varepsilon > 0$, we call a feasible solution of ${{\rm P}_{S_{\infty}}}(\mathcal{A})$ whose minimum eigenvalue is $\varepsilon$ or more an {\em  $\varepsilon$-feasible solution of ${{\rm P}_{S_{\infty}}}(\mathcal{A})$}.
This section describes the main algorithm.
To set the upper bound for the minimum eigenvalue of any feasible solution $x$ of \modifySecond{${\rm P}_{S_{\infty}}(\mathcal{A})$}, Algorithm \ref{main algorithm} focuses on the product $\det (\bar{x})$ of the eigenvalues of the arbitrary feasible solution $\bar{x}$ of the scaled problem \modifyFirst{${\rm P}_{S_{\infty}}(\mathcal{A}^k Q^k)$}.
Algorithm \ref{main algorithm} works as follows.
First, we calculate the corresponding projection $P_{\mathcal{A}}$ onto $\mbox{ker} \hspace{0.3mm} \mathcal{A}$ and generate an initial point as input to the basic procedure. 
Next, we call the basic procedure and determine whether to end the algorithm with an $\varepsilon$-feasible solution or to perform problem scaling according to the returned result, as follows:
\begin{enumerate}
\item If a feasible solution of ${\rm P}(\mathcal{A})$ or ${\rm D}(\mathcal{A})$ is returned from the basic procedure, the feasibility of ${\rm P}(\mathcal {A})$ can be determined, and we stop the main algorithm.
\item If the basic procedure returns the sets of indices $H_1 , \dots , H_p$ and the sets of primitive idempotents $C_1 , \dots , C_p$ that construct the corresponding Jordan frames, then 
\modifyFirst{
for the total number of cuts obtaibed in the $\ell$-th block ${\rm num}_\ell$, 
}
\begin{enumerate}
\item if $\mbox{num}_\ell  \geq r_\ell \frac{\log \varepsilon}{\log \xi}$ holds for some $\ell \in \{ 1, \dots p \}$, we  determine that ${{\rm P}_{S_{\infty}}}(\mathcal{A})$ has no $\varepsilon$-feasible solution according to Proposition \ref{prop:lambda-min-upper} and stop the main algorithm, 
\item if $\mbox{num}_\ell  < r_\ell \frac{\log \varepsilon}{\log \xi}$ holds for any $\ell \in \{ 1, \dots p \}$, we rescale the problem and call the basic procedure again.
\end{enumerate}
\end{enumerate}

Note that our main algorithm is similar to Louren\c{c}o et al.'s method in the sense that it keeps information about the possible minimum eigenvalue of any feasible solution of the problem. 
In contrast, Pena and Soheili's method \cite{Pena2017} does not keep such information.
Algorithm \ref{main algorithm} terminates after no more than $-\frac{r}{\log \xi} \log \left( \frac{1}{\varepsilon}\right) - p + 1$ iterations, so our main algorithm can be said to be a polynomial-time algorithm.
We will give this proof in section \ref{sec: finite termination of MA}.
We should also mention that step 24 in Algorithm \ref{main algorithm} is not a reachable output theoretically. We have added this step in order to consider the influence of the numerical error in practice.

 \begin{algorithm}
 \caption{Main algorithm}
 \label{main algorithm}
 \begin{algorithmic}[1]
 \renewcommand{\algorithmicrequire}{\textbf{Input: }}
 \renewcommand{\algorithmicensure}{\textbf{Output: }}
  \renewcommand{\stop}{\textbf{stop }}
 \renewcommand{\return}{\textbf{return }}

 \STATE \algorithmicrequire $\mathcal{A}$, $\mathcal{K}$, $\varepsilon$ and a constant $\xi$ such that $0 < \xi < 1$
 \STATE \algorithmicensure a solution to ${\rm P}(\mathcal{A})$ or ${\rm D}(\mathcal{A})$ or a certificate that there is no $\varepsilon$ feasible solution.
 \STATE $k \leftarrow 1$ , $\mathcal{A}^1 \leftarrow \mathcal{A}$ , $\mbox{num}_\ell \leftarrow 0 $, $\bar{Q_\ell} \leftarrow I_\ell $, \modifyThird{$RP_\ell \leftarrow I_\ell$, $RD_\ell \leftarrow I_\ell$} for all $\ell \in \{1,\dots,p\}$
	 \STATE Compute \modifyFirst{$P_{\mathcal{A}^k}$} and call the basic procedure with \modifyFirst{$P_{\mathcal{A}^k}$}, $\frac{1}{r}e$, $\xi$
  \IF {basic procedure returns $z$ }
 \STATE \stop main algorithm and \return \modifyThird{$RPz$ ($RPz$ is a feasible solution of ${\rm P}(\mathcal{A})$)}
 \ELSIF { basic procedure returns $y$ or $v$ }
 \STATE \stop main algorithm and \return \modifyThird{$RDy$ or $RDv$ ( $RDy$ or $RDv$ is a feasible solution of $D(\mathcal{A})$)}
 \ELSIF  {basic procedure returns $H^k_1 , \dots , H^k_p$ and $C^k_1 , \dots , C^k_p$}
 \FOR {$\ell \in \{1,\dots,p\}$}
 \IF{$|H^k_\ell| > 0$}
 \STATE $g_\ell \leftarrow \sqrt{\xi} \sum_{h \in H^k_\ell} c^k(v_\ell)_h + \sum_{h \notin H_\ell^k}^{r_\ell} c^k(v_\ell)_h$ 
 \STATE $Q_\ell \leftarrow Q_{g_\ell}$, \modifyThird{$RP_\ell \leftarrow RP_\ell Q_{g_\ell}$, $RD_\ell \leftarrow RD_\ell Q_{g^{-1}_\ell}$}
 \STATE $\mbox{num}_\ell \leftarrow |H^k_\ell| + \mbox{num}_\ell $
 \IF { $ \mbox{num}_\ell  \geq r_\ell \frac{\log \varepsilon}{\log \xi}$ }
 \STATE { \stop main algorithm. There is no $\varepsilon$ feasible solution.}
 \ENDIF
 \STATE $\bar{Q_\ell} \leftarrow  Q_{g_\ell^{-1}} \bar{Q_\ell} $
 \ELSE
 \STATE $ Q_\ell \leftarrow I_\ell$
 \ENDIF
 \ENDFOR
 \ELSE
 \STATE \return basic procedure error
 \ENDIF
 \STATE Let $Q^k = (Q_1 , \dots , Q_p)$ 
 \STATE $\mathcal{A}^{k+1} \leftarrow \mathcal{A}^k Q^k$ , $k \leftarrow k+1$. Go back to line 4.
 \end{algorithmic} 
 \end{algorithm}
 
\subsection{Finite termination of the main algorithm}
\label{sec: finite termination of MA}

Here, we discuss how many iterations are required until we can determine that the minimum eigenvalue $\lambda_{\rm min} (x)$ is less than $\varepsilon$ for any $x \in F_{{\rm P}_{S_{\infty}}(\mathcal{A})}$.
Before going into the proof, we explain the Algorithm \ref{main algorithm} in more detail than in section \ref{sec:outline of main alg}.
At each iteration of Algorithm \ref{main algorithm}, it accumulates the number of cuts $|H_\ell^k|$ obtained in the $\ell$-th block and stores the value in ${\rm num}_\ell$. Using ${\rm num}_\ell$, we can compute an upper bound for $\lambda_{\rm min} (x)$ (Proposition \ref{prop:lambda-min-upper}). On line 18, $\bar{Q}_\ell$ is updated to $\bar{Q}_\ell \leftarrow Q_{g_\ell^{-1}} \bar{Q}_\ell $, where $\bar{Q}_\ell$ plays the role of an operator that gives the relation $\bar{x}_\ell = \bar{Q}_\ell (x_\ell)$ for the solution $x$ of the original problem and the solution $\bar{x}$ of the scaled problem. For example, if $|H^1_\ell| > 0$ for $k=1$ (suppose that the cut was obtained in the $\ell$-th block), then the proposed method scales $\mathcal{A}^1_\ell Q^1_\ell$ and the problem to yield $\bar{x}_\ell = Q_{g_\ell^{-1}}(x_\ell)$ for the feasible solution $x$ of the original problem. And if $|H^2_\ell| > 0$ even for $k=2$, then the proposed method scales $\bar{x}$ again, so that $\bar{\bar{x}}_\ell = Q_{g_\ell^{-1}}(\modifyFirst{\bar{x}_\ell}) = \bar{Q}_\ell (x_\ell)$ holds. Note that $\bar{Q}_\ell$ is used only for a concise proof of Proposition \ref{prop:lambda-min-upper}, so it is not essential.

Now, let us derive an upper bound for the minimum eigenvalue $\lambda_{\min}(x_\ell)$ of each $\ell$-th block of $x$ obtained after the $k$-th iteration of Algorithm \ref{main algorithm}.
Proposition \ref{prop:iteration-num-ma} gives an upper bound for the number of iterations of Algorithm \ref{main algorithm}.

\begin{proposition}
\label{prop:lambda-min-upper}
After $k$ iterations of Algorithm \ref{main algorithm}, for any feasible  solution $x$ of ${\rm P}_{S_{\infty}}(\mathcal{A})$ and $ \ell \in \{1 , \dots , p\}$, the $\ell$-th block element $x_\ell$ of $x$ satisfies
\begin{equation}
r_\ell \log {\left( \lambda_{\min}(x_\ell) \right)} \leq \mbox{\rm num}_\ell \log {\xi} . \label{prop:10-0}
\end{equation}
\end{proposition}

\begin{proof}
At the end of the $k$-th iteration, any feasible solution $\bar{x}$ of the scaled problem ${\rm P}_{S_{\infty}}(\mathcal{A}^{k+1}) = {\rm P}_{S_{\infty}}(\mathcal{A}^kQ^k)$ obviously satisfies
\begin{equation}
\det \bar{x}_\ell \leq \det e_\ell \ \ (\ell =1,2,\ldots,p).  \label{prop:10-1}
\end{equation}
Note that $\det \bar{x}_\ell$ can be expressed in terms of $\det x_\ell$. For example, if $|H_\ell^1|>0$ when $k=1$, then using Proposition \ref{q-det-relation}, for any feasible solution $\bar{x}$ of ${\rm P}_{S_{\infty}}(\mathcal{A}^{2})$, we find that
\begin{equation}
\det \bar{x}_\ell = \det Q_{g_\ell^{-1}} (x_\ell) = \det(g_\ell^{-1})^2 \det x_\ell =   {\left( \frac{1}{\sqrt{\xi}} \right)}^{2|H_\ell^1|} \det x_\ell =   {\left( \frac{1}{\xi} \right)}^{|H_\ell^1|} \det x_\ell  .\notag 
\end{equation}
This means that $\det \bar{x}_\ell$ can be determined from $\det x_\ell$ and the number of cuts obtained so far in the $\ell$-th block. In Algorithm \ref{main algorithm}, the value of $\mbox{num}_\ell$ is updated only when $|H_\ell^k|>0$. Since $\bar{x}$ satisfies $\bar{x}_\ell = \bar{Q}_\ell (x_\ell) \ (\ell=1,2,\ldots,p)$ for each feasible solution $x$ of ${\rm P}_{S_{\infty}}(\mathcal{A})$, we can see that 
\begin{equation}
\det \bar{x}_\ell = \det \bar{Q_\ell} (x_\ell) =  {\left( \frac{1}{\xi} \right)}^{|H_\ell^k|} \times  {\left( \frac{1}{\xi} \right)}^{|H_\ell^{k-1}|} \dots \times  {\left( \frac{1}{\xi} \right)}^{|H_\ell^1|} \times \det x_\ell =  {\left( \frac{1}{\xi} \right)}^{\mbox{num}_\ell} \det x_\ell . \notag
\end{equation}
Therefore,  (\ref{prop:10-1}) implies $\det x_\ell \leq \xi^{\mbox{num}_\ell} \det e_\ell = \xi^{\mbox{num}_\ell}$ and the fact ${\left( \lambda_{\min}(x_\ell) \right)}^{r_\ell} \leq \det x_\ell$ implies  ${\left( \lambda_{\min}(x_\ell) \right)}^{r_\ell} \leq {\xi}^{\mbox{num}_\ell}$.
By taking the logarithm of both sides of this inequality, we obtain (\ref{prop:10-0}). 
\end{proof}

\begin{proposition}
\label{prop:iteration-num-ma}
Algorithm \ref{main algorithm} terminates after no more than $-\frac{r}{\log \xi} \log \left( \frac{1}{\varepsilon}\right) - p + 1$ iterations.
\end{proposition}

\begin{proof}
Let us call iteration $k$ of Algorithm \ref{main algorithm} {\em good} if $|H_\ell^k| > 0$ for some $\ell \in \{1,2, \dots , p\}$ at that iteration. Suppose that at least $-\frac{r_\ell}{\log \xi} \log \left( \frac{1}{\varepsilon}\right)$ {\em good} iterations are observed for a cone $\mathcal{K}_\ell$. Then, by substituting $-\frac{r_\ell}{\log \xi} \log \left( \frac{1}{\varepsilon}\right)$ into $\mbox{num}_\ell$ of inequality (\ref{prop:10-0}) in Proposition \ref{prop:lambda-min-upper}, we have $\log {\left( \lambda_{\min}(x_\ell) \right)} \leq \log \varepsilon$
and hence, $\lambda_{\min}(x_\ell) \leq \varepsilon$.
This implies that Algorithm \ref{main algorithm} terminates after no more than 
\begin{equation}
\sum_{\ell=1}^p \left( -\frac{r_\ell}{\log \xi} \log \left( \frac{1}{\varepsilon}\right) - 1 \right) + 1 = -\frac{r}{\log \xi} \log \left( \frac{1}{\varepsilon}\right) - p + 1  \notag
\end{equation}
iterations.
\end{proof}

\section{Computational costs of the algorithms}
\label{sec: compare}

This section compares the computational costs of Algorithm \ref{main algorithm}, Louren\c{c}o et al.'s method~\cite{Bruno2019} and Pena and Soheili's method~\cite{Pena2017}.
Section \ref{sec: complexity of MA vs Lourenco} compares the computational costs of Algorithm \ref{main algorithm} and Louren\c{c}o et al.'s method, and
Section \ref{sec: complexity of MA vs Pena} compares those of Algorithm \ref{main algorithm} and Pena and Soheili's method under the assumption that  $\mbox{ker} \hspace{0.3mm} \mathcal{A} \cap  \mbox{int} \hspace{0.75mm} \mathcal{K} \neq  \emptyset$.

Both the proposed method and the method of Louren\c{c}o et al. guarantee finite termination of the main algorithm by termination criteria indicating the nonexistence of an $\varepsilon$-feasible solution, so that it is possible to compare the computational costs of the methods without making any special assumptions. This is because both methods proceed by making cuts to the feasible region using the results obtained from the basic procedure. On the other hand, Pena and Soheili's method cannot be simply compared because the upper bound of the number of iterations of their main algorithm includes an unknown value of $\delta (\mbox{ker} \hspace{0.3mm} \mathcal{A} \cap \mbox{int} \hspace{0.75mm} \mathcal{K}) := \max_x \left \{ {\rm det}(x) \mid x \in \mbox{ker} \hspace{0.3mm} \mathcal{A} \cap \mbox{int} \hspace{0.75mm} \mathcal{K} , \|x\|_J^2 = r \right\}$.
However, by making the assumption $\mbox{ker} \hspace{0.3mm} \mathcal{A} \cap \mbox{int} \hspace{0.75mm} \mathcal{K} \neq \emptyset$ and deriving a lower bound for $\delta (\mbox{ker} \hspace{0.3mm} \mathcal{A} \cap \mbox{int} \hspace{0.75mm} \mathcal{K})$, we make it possible to compare Algorithm \ref{main algorithm} with Pena and Soheili's method without knowing the specific values of $\delta (\mbox{ker} \hspace{0.3mm} \mathcal{A} \cap \mbox{int} \hspace{0.75mm} \mathcal{K})$.
\subsection{Comparison of Algorithm \ref{main algorithm} and Louren\c{c}o et al.'s method}
\label{sec: complexity of MA vs Lourenco}
Let us consider the computational cost of Algorithm \ref{main algorithm}.
At each iteration, the most expensive operation is computing $P_\mathcal{A}$ on line 4.
Recall that $d$ is the dimension of the Euclidean space $\mathbb{E}$ corresponding to $\mathcal{K}$.
As discussed in  \cite{Bruno2019}, by considering $P_\mathcal{A}$ to be an $m \times d$ matrix, 
we find that the computational cost of $P_{\mathcal{A}}$ is  $\mathcal{O}(m^3+m^2d)$.
Therefore,  by taking the computational cost of the basic procedure (Algorithm \ref{basic procedure}) and Proposition \ref{prop:iteration-num-ma} into consideration, the cost of Algorithm \ref{main algorithm} turns out to be
\begin{equation}
\label{eq: computational cost 1}
\mathcal{O} \left( -\frac{r}{\log \xi} \log \left( \frac{1}{\varepsilon}\right) \left( m^3+m^2d +  \frac{1}{\xi^2} p^2 r_{\max}^2  \left( \max \left( C^{\rm sd} , md \right) \right) \right) \right)
\end{equation}
where $C^{\rm sd}$ is the computational cost of the spectral decomposition of $x \in \mathbb{E}$.

Note that, in \cite{Bruno2019}, the authors showed that the cost of their algorithm is
\begin{equation}
\label{eq: computational cost 2}
\mathcal{O} \left( \left( \frac{r}{\varphi(\rho)} \log \left( \frac{1}{\varepsilon} \right) - \sum_{i=1}^p \frac{r_i \log(r_i)}{\varphi(\rho)}\right) \left( m^3 + m^2d +  \rho^2 p^3 r_{max}^2 \left(\max \left(C^{\rm min} , md \right)  \right) \right) \right) 
\end{equation}
where $C^{\rm min}$ is the cost of computing the minimum eigenvalue of $x \in \mathbb{E}$ with the corresponding idempotent, \modifyFirst{$\rho$ is an input parameter used in their basic procedure (like $\xi$ in the proposed algorithm) and $\varphi(\rho) = 2 - 1/\rho - \sqrt{3-2/\rho}$.}

When the symmetric cone is simple, by setting $\xi = 1/2$ and $\rho = 2$, the maximum number of iterations of the basic procedure is bounded by the same value in both algorithms. Accordingly, we will compare the two computational costs (\ref{eq: computational cost 1}) and (\ref{eq: computational cost 2}) by supposing $\xi = 1/2$ and $\rho = 2$ (hence, $- \log \xi \simeq 0.69$ and $\varphi (\rho) \simeq 0.09$). As we can see below, the cost (\ref{eq: computational cost 1}) of our method is smaller than (\ref{eq: computational cost 2}) in the cases of linear programming and second-order cone problems and is equivalent to (\ref{eq: computational cost 2}) in the case of semidefinite problems.
First, let us consider the case where $\mathcal{K}$ is the $n$-dimensional nonnegative orthant $\mathbb{R}^n_+$.
Here, we see that $r=p=d=n$, $r_1 = \cdots = r_p = r_{\max}=1$, and $\max \left( C^{\rm sd} , md \right) = \max \left( C^{\rm min} , md \right) = md$ hold.
By substituting these values, the bounds  (\ref{eq: computational cost 1}) and  (\ref{eq: computational cost 2}) turn out to be
\begin{equation}
\mathcal{O} \left( \frac{n}{0.69} \log \left( \frac{1}{\varepsilon} \right) \left( m^3 + m^2n + 4 mn^3 \right) \right)  \notag
\end{equation}
and
\begin{equation}
\mathcal{O} \left( \frac{n}{0.09} \log \left( \frac{1}{\varepsilon} \right) \left(  m^3 + m^2n + 4 mn^4 \right) \right).  \notag
\end{equation}
This implies that for the linear programming case, our method (which is equivalent to Roos's original method \cite{Roos2018}) is superior to Louren\c{c}o et al.'s method \cite{Bruno2019} in terms of bounds  (\ref{eq: computational cost 1}) and  (\ref{eq: computational cost 2}) .

Next, let us consider the case where $\mathcal{K}$ is composed of $p$ simple second-order cones $\mathbb{L}^{n_i} \ (i= 1, \ldots, p)$.
In this case, we see that $d = \sum_{i=1}^p n_i$, $r_1 = \cdots = r_p = r_{\max}=2$ and $\max \left( C^{\rm sd} , md \right) = \max \left( C^{\rm min} , md \right) = md$ hold.
By substituting these values, the bounds  (\ref{eq: computational cost 1}) and  (\ref{eq: computational cost 2}) turn out to be
\begin{equation}
\mathcal{O} \left( \frac{2p}{0.69} \log \left( \frac{1}{\varepsilon} \right) \left( m^3 + m^2d + 16p^2md \right) \right)  \notag
\end{equation}
and
\begin{equation}
\mathcal{O} \left(  \frac{2p}{0.09} \left( \log \left( \frac{1}{\varepsilon} \right) -  \log 2 \right) \left(  m^3 + m^2d + 16  p^3 md \right) \right).  \notag
\end{equation}
Note that $\varepsilon$ is expected to be very small ($10^{-6}$ or even $10^{-12}$ in practice) and $\frac{1}{0.69} \log \left( \frac{1}{\varepsilon} \right) \leq \frac{1}{0.09} \left( \log \left( \frac{1}{\varepsilon} \right)  -  \log 2 \right)$ if $\varepsilon \leq 0.451$.
Thus, even in this case, we may conclude that our method is superior to Louren\c{c}o et al.'s method in terms of the  bounds  (\ref{eq: computational cost 1}) and  (\ref{eq: computational cost 2}) .

Finally, let us consider the case where $\mathcal{K}$ is a simple $n \times n$ positive semidefinite cone. We see that $p=1$, $r = n$, and $d = \frac{n(n+1)}{2}$ hold, and upon substituting these values, the bounds (\ref{eq: computational cost 1}) and (\ref{eq: computational cost 2}) turn out to be
\begin{equation}
\mathcal{O} \left( \frac{n}{0.69} \log \left( \frac{1}{\varepsilon} \right) \left(m^3 + m^2n^2 + 4n^2  \max \left( C^{\rm sd} , mn^2 \right)  \right) \right)  \notag
\end{equation}
and 
\begin{equation}
\mathcal{O} \left( \frac{n}{0.09} \log \left( \frac{1}{\varepsilon} \right) \left(m^3 + m^2n^2 + 4n^2 \max ( C^{\rm min} , mn^2) \right) \right).  \notag
\end{equation}
From the discussion in Section \ref{sec: CSD and Cmin}, we can assume $\mathcal{O}\left( C^{\rm sd} \right) =  \mathcal{O}\left( C^{\rm min} \right)$, and  the computational bounds of two methods are equivalent.

\subsection{Comparison of Algorithm \ref{main algorithm} and Pena and Soheili's method}
\label{sec: complexity of MA vs Pena}

In this section, we assume that $\mathcal{K}$ is simple since \modifyFirst{~\cite{Pena2017} has presented an algorithm for the simple form.} We also assume that $\mbox{{\rm ker}} \hspace{0.3mm} \mathcal{A} \cap \mbox{{\rm int}} \hspace{0.75mm} \mathcal{K} \neq \emptyset$, because Pena and Soheili's method does not terminate if $\mbox{{\rm ker}} \hspace{0.3mm} \mathcal{A} \cap \mbox{{\rm int}} \hspace{0.75mm} \mathcal{K} = \emptyset$ and $\mbox{{\rm range}} \mathcal{A}^* \cap \mbox{{\rm int}} \hspace{0.75mm} \mathcal{K} = \emptyset$. Furthermore, for the sake of simplicity, we assume that the main algorithm of Pena and Soheili's method applies only to $\mbox{{\rm ker}} \hspace{0.3mm} \mathcal{A} \cap \mbox{{\rm int}} \hspace{0.75mm} \mathcal{K}$. (Their original method applies the main algorithm to $\mbox{{\rm range}} \mathcal{A}^* \cap \mbox{{\rm int}} \hspace{0.75mm} \mathcal{K}$ as well.)

First, we will briefly explain the idea of deriving an upper bound for the number of iterations required to find $x \in \mbox{{\rm ker}} \hspace{0.3mm} \mathcal{A} \cap \mbox{{\rm int}} \hspace{0.75mm} \mathcal{K}$ in Pena and Soheili's method. Pena and Soheili derive it by focusing on the indicator $\delta (\mbox{ker} \hspace{0.3mm} \mathcal{A} \cap \mbox{int} \hspace{0.75mm} \mathcal{K}) := \max_x \left \{ {\rm det}(x) \mid x \in \mbox{ker} \hspace{0.3mm} \mathcal{A} \cap \mbox{int} \hspace{0.75mm} \mathcal{K} , \|x\|_J^2 = r \right\}$. If $\mbox{{\rm ker}} \hspace{0.3mm} \mathcal{A} \cap \mbox{{\rm int}} \hspace{0.75mm} \mathcal{K} \neq \emptyset$, then $\delta (\mbox{ker} \hspace{0.3mm} \mathcal{A} \cap \mbox{int} \hspace{0.75mm} \mathcal{K}) \in (0,1]$ holds, and if $e \in \mbox{{\rm ker}} \hspace{0.3mm} \mathcal{A} \cap \mbox{{\rm int}} \hspace{0.75mm} \mathcal{K}$, then $\delta (\mbox{ker} \hspace{0.3mm} \mathcal{A} \cap \mbox{int} \hspace{0.75mm} \mathcal{K}) =1$ holds. If $e \in \mbox{{\rm ker}} \hspace{0.3mm} \mathcal{A} \cap \mbox{{\rm int}} \hspace{0.75mm} \mathcal{K}$, then the basic procedure terminates immediately and returns $\frac{1}{r}e$ as a feasible solution. Then, they prove that $\delta (Q_v \left( \mbox{ker} \hspace{0.3mm} \mathcal{A} \right) \cap \mbox{int} \hspace{0.75mm} \mathcal{K}) \geq 1.5 \cdot \delta (\mbox{ker} \hspace{0.3mm} \mathcal{A} \cap \mbox{int} \hspace{0.75mm} \mathcal{K})$ holds if the parameters are appropriately set, and derive an upper bound on the number of scaling steps, i.e., the number of iterations, required to obtain $\delta (Q_v \left( \mbox{ker} \hspace{0.3mm} \mathcal{A} \right) \cap \mbox{int} \hspace{0.75mm} \mathcal{K}) = 1$.

In the following, we obtain an upper bound for the number of iterations of Algorithm \ref{main algorithm} using the index $\delta^{{\rm supposed}} \left( \mbox{{\rm ker}} \hspace{0.3mm} \mathcal{A} \cap \mbox{int} \hspace{0.75mm} \mathcal{K} \right) := \max_x \left \{ {\rm det}(x) \mid x \in \mbox{ker} \hspace{0.3mm} \mathcal{A} \cap \mbox{int} \hspace{0.75mm} \mathcal{K} , \|x\|_J^2 = 1 \right \}$. Note that $\delta \left( \mbox{{\rm ker}} \hspace{0.3mm} \mathcal{A} \cap \mbox{int} \hspace{0.75mm} \mathcal{K} \right) = r^{\frac{r}{2}} \cdot \delta^{{\rm supposed}} \left( \mbox{{\rm ker}} \hspace{0.3mm} \mathcal{A} \cap \mbox{int} \hspace{0.75mm} \mathcal{K} \right)$. In fact, if $x^*$ is the point giving the maximum value of $\delta^{{\rm supposed}} \left( \mbox{{\rm ker}} \hspace{0.3mm} \mathcal{A} \cap \mbox{int} \hspace{0.75mm} \mathcal{K} \right)$, then the point giving the maximum value of $\delta \left( \mbox{{\rm ker}} \hspace{0.3mm} \mathcal{A} \cap \mbox{int} \hspace{0.75mm} \mathcal{K} \right)$ is $\sqrt{r} x^*$. Also, if $\mbox{{\rm ker}} \hspace{0.3mm} \mathcal{A} \cap \mbox{{\rm int}} \hspace{0.75mm} \mathcal{K} \neq \emptyset$, then $\delta^{{\rm supposed}} (\mbox{ker} \hspace{0.3mm} \mathcal{A} \cap \mbox{int} \hspace{0.75mm} \mathcal{K}) \in (0,1/ r^{\frac{r}{2}} ]$, and if $\frac{1}{\sqrt{r}} e \in \mbox{{\rm ker}} \hspace{0.3mm} \mathcal{A} \cap \mbox{{\rm int}} \hspace{0.75mm} \mathcal{K}$, then $\delta^{{\rm supposed}} (\mbox{ker} \hspace{0.3mm} \mathcal{A} \cap \mbox{int} \hspace{0.75mm} \mathcal{K}) =1/ r^{\frac{r}{2}}$.

The outline of this section is as follows: First, we show that a lower bound for $\delta^{{\rm supposed}} \left( \mbox{{\rm ker}} \hspace{0.3mm} \mathcal{A} \cap \mbox{int} \hspace{0.75mm} \mathcal{K} \right)$ can be derived using the index value $\delta^{{\rm supposed}} \left( Q_{g^{-1}} \left( \mbox{{\rm ker}} \hspace{0.3mm} \mathcal{A} \right) \cap \mbox{int} \hspace{0.75mm} \mathcal{K} \right)$ for the problem after scaling (Proposition \ref{prop:compare scaling}). Then, using this result, we derive an upper bound for the number of operations required to obtain $\delta^{{\rm supposed}} \left( Q_{g^{-1}} \left( \mbox{{\rm ker}} \hspace{0.3mm} \mathcal{A} \right) \cap \mbox{int} \hspace{0.75mm} \mathcal{K} \right) = 1/ r^{\frac{r}{2}}$ (Proposition \ref{pro:compare iteration}). Finally, we compare the proposed method with Pena and Soheili's method.
To prove Proposition \ref{prop:compare tr} used in the proof of Proposition \ref{prop:compare scaling}, we use the following propositions from \cite{JNW1934}.
\begin{proposition}[Theorem 3 of~\cite{JNW1934}]
\label{JNW_Th3}
Let $c \in \mathbb{E}$ be an idempotent and $N_\lambda (c)$ be the set such that $N_\lambda (c) = \{ x \in \mathbb{E} \mid c \circ x = \lambda x\}$.
Then $N_\lambda(c)$ is a linear maniforld, but if $\lambda \neq 0, \frac{1}{2}$, and $1$, then $N_\lambda (c)$ consists of zero alone. 
Each $x \in \mathbb{E}$ can be represented in the form
\begin{equation}
\notag
\begin{array}{lccc}
x = u + v + w,	&u \in N_0(c),	&v \in N_{\frac{1}{2}}(c),	&w \in N_1(c),
\end{array}
\end{equation}
in one and only one way.
\end{proposition}
\begin{proposition}[Theorem 11 of~\cite{JNW1934}]
\label{JNW_Th11}
$c \in \mathbb{E}$ is a primitive idempotent if and only if $N_1(c) =\{ x \in \mathbb{E} \mid c \circ x = x\} = \mathbb{R} c$.
\end{proposition}
\begin{proposition}
\label{prop:compare tr}
Let $c \in \mathbb{E}$ be a primitive idempotent.
Then, for any $x \in \mathbb{E}$, $\langle x, Q_c(x) \rangle =  \langle x,c \rangle^2$ holds.
\end{proposition}
\begin{proof}
From Propositions \ref{JNW_Th3} and \ref{JNW_Th11}, for any $x \in \mathbb{E}$, there exist a real number $\lambda \in \mathbb{R}$ and elements $u \in N_0(c)$ and $v \in N_{\frac{1}{2}}(c)$ such that $x =u + v + \lambda c $. \\
First, we show that $\langle x,c \rangle = \lambda$. For $v \in N_\frac{1}{2} (c)$, we see that $\langle v,c \rangle = \langle v, c \circ c \rangle = \langle v \circ c, c \rangle = \langle c \circ v, c \rangle = \frac{1}{2} \langle v, c \rangle$, which implies that $\langle v,c \rangle =0$. Thus, since $u \in N_0(c)$ and $u \circ c = 0$, $\langle x,c \rangle$ is given by
\begin{equation}
\notag
\langle x,c \rangle = \langle u + v + \lambda c , c  \rangle  = \langle u,c \rangle + \langle v,c \rangle = \lambda \langle c,c \rangle = 0 + 0 + \lambda.
\end{equation}
On the other hand, by using the facts $x = u + v + \lambda c$, $c^2 = c$, $c \circ u =0$ and $c \circ v =\frac{1}{2} v$ repeatedly, we have 
\begin{align*}
\langle x, Q_c(x) \rangle 
&= \langle x, 2 c \circ (c \circ x) - c^2 \circ x \rangle \\
&=\langle x, 2 c \circ (c \circ (u + v + \lambda c)) - c \circ (u + v + \lambda c)\rangle \\
&= \langle x, 2 c \circ (\frac{1}{2} v + \lambda c) - ( \frac{1}{2}v + \lambda c ) \rangle \\
&= \langle x, (\frac{1}{2} v + 2 \lambda c) - ( \frac{1}{2}v + \lambda c )  \rangle = \langle x,  \lambda c  \rangle =  \lambda^2.
\end{align*}
Thus, we have shown that $\langle x, Q_c(x) \rangle = \langle x,c \rangle ^2$ holds.
\end{proof}
\begin{remark}
\label{remark:pena}
It should be noted that the proof of Proposition 3 in \cite{Pena2017} is not correct since equation (14) does not necessarily hold.
The above Proposition \ref{prop:compare tr} also gives a correct proof of Proposition 3 in \cite{Pena2017}. See the computation $\langle y , Q_{g^{-2}} (y) \rangle$ in the proof of Proposition \ref{prop:compare scaling}.
\end{remark}
\begin{proposition}
\label{prop:compare scaling}
Suppose that $\mbox{{\rm ker}} \hspace{0.3mm} \mathcal{A} \cap  \mbox{{\rm int}} \hspace{0.75mm} \mathcal{K} \neq  \emptyset$ and that, for a given nonempty index set $H \subseteq \{1 , \dots r \}$, Jordan frame $c_1 , \dots , c_r $, and $0 < \xi < 1$,
\begin{equation}
\notag
\langle c_i , x \rangle  \leq  \xi \  ( i \in H), \ \  \langle c_i , x \rangle \leq 1 \ ( i \notin H)
\end{equation}
holds for any $x \in F_{{\rm P}_{S_{\infty}}(\mathcal{A})}$.
Define $g \in \mbox{\rm int} \hspace{0.75mm} \mathcal{K}$ as $g := \sqrt{\xi} \sum_{h \in H} c_h + \sum_{h \notin H} c_h$.
Then, the following inequality holds:
\begin{equation}
\notag
\delta^{{\rm supposed}} \left( \mbox{{\rm ker}} \hspace{0.3mm} \mathcal{A} \cap  {\rm int} \hspace{0.75mm} \mathcal{K} \right) > \xi \cdot \delta^{{\rm supposed}} \left( Q_{g^{-1}} \left(\mbox{{\rm ker}} \hspace{0.3mm} \mathcal{A} \right) \cap  {\rm int} \hspace{0.75mm} \mathcal{K} \right).
\end{equation}
\end{proposition}

\begin{proof}
For simplicity of discussion, let $|H|=1$, i.e., $H =\{i\}$.
Let us  define the points $x^*$, $y^*$, and $\bar{x}^*$ as follows:
\begin{equation}
\begin{array}{l}
x^*= \argmax \left\{ \det (x) \mid x \in \mbox{{\rm ker}} \hspace{0.3mm} \mathcal{A}  \cap {\rm int} \hspace{0.75mm} \mathcal{K}, \| x \|^2_J = 1 \right\}, \\
y^* = \argmax \left\{ \det (y) \mid y \in \mbox{{\rm ker}} \hspace{0.3mm} \mathcal{A} \cap {\rm int} \hspace{0.75mm} \mathcal{K}, \| Q_{g^{-1}} (y)\|^2_J = 1 \right\}, \\
\bar{x} ^*= \argmax \left\{ \det (\bar{x}) \mid \bar{x} \in Q_{g^{-1}} \left(\mbox{{\rm ker}} \hspace{0.3mm} \mathcal{A} \right) \cap {\rm int} \hspace{0.75mm} \mathcal{K}, \| \bar{x} \|^2_J = 1 \right\}.
\end{array}
\notag
\end{equation}
Note that the feasible region with respect to $y$ is the set of solutions whose norm is $1$ after scaling. 
First, we show that $\|y\|_J^2 < 1$, and then $\det (x^*) > \det (y^*)$.
Proposition \ref{q-det-relation} ensures that $\|Q_{g^{-1}} (y)\|_J^2 = \langle Q_{g^{-1}} (y) , Q_{g^{-1}} (y) \rangle = \langle y , Q_{g^{-2}} (y) \rangle$.
To expand $Q_{g^{-2}} (y)$, we expand the following equations by letting $a =  \frac{1}{\sqrt{\xi}} - 1 $:
\begin{align*}
g^{-2} &= e + (2a + a^2) c_i, \\
g^{-4} &=  e + \left( 2(2a + a^2) + (2a+a^2)^2 \right) c_i \\
g^{-2} \circ y &= y + (2a+a^2) c_i \circ y, \\
g^{-2} \circ ( g^{-2} \circ y) &= y + 2(2a+a^2) c_i \circ y + (2a+a^2)^2 c_i \circ (c_i \circ y), \\
g^{-4} \circ y &= y + \left( 2(2a + a^2) + (2a+a^2)^2 \right) c_i \circ y.
\end{align*}
Thus, $Q_{g^{-2}} (y)$ turns out to be
\begin{align*}
Q_{g^{-2}} (y)
&=  2 g^{-2} \circ ( g^{-2} \circ y) - g^{-4} \circ y \\
&=  y +  2(2a+a^2) c_i \circ y +  2(2a+a^2)^2 c_i \circ (c_i \circ y) - (2a+a^2)^2  c_i \circ y \\
&=  y +  2(2a+a^2) c_i \circ y +  (2a+a^2)^2 Q_{c_i} (y), \\
\end{align*}
and hence, we obtain $\|Q_{g^{-1}} (y)\|_J^2$ as 
\begin{align*}
\langle y , Q_{g^{-2}} (y) \rangle 
&= \|y\|_J^2 +  2(2a+a^2) \langle y, c_i \circ y \rangle +  (2a+a^2)^2  \langle y, Q_{c_i} (y) \rangle \\
&= \|y\|_J^2 +  2(2a+a^2) \langle y \circ y, c_i \rangle +  (2a+a^2)^2  \left( \langle y, c_i \rangle \right)^2 \\
\end{align*}
where the second equality follows from Proposition \ref{prop:compare tr}. 
Here, $y \in {\rm int} \hspace{0.75mm} \mathcal{K}$ and  $c_i \in \mathcal{K}$imply that $\langle y,c_i \rangle >0$, and $y \circ y = y^2 \in {\rm int} \hspace{0.75mm} \mathcal{K}$ implies  $\langle y \circ y,c_i \rangle > 0$.
Noting that $a > 0$ and $\|Q_{g^{-1}} (y)\|_J^2 = 1$, $\|y\|_J^2 < 1$ should hold and hence, $\frac{1}{\|y^*\|_J} >1$, which implies that $\det \left(\frac{1}{\|y^*\|_J} y^* \right) > \det(y^*)$.
Since $\left\| \frac{1}{\|y^*\|_J} y^* \right\|_J^2 = 1$ holds, we find that  $\det(x^*) > \det(y^*)$.
Next, we describe the lower bound for $\det (y^*)$ using $\det (\bar{x}^*)$.
Since the largest eigenvalue of $\bar{x}$ satisfying $\| \bar{x} \|_J^2 =1$ is less than $1$, by Proposition \ref{prop:scaling}, we have:
\begin{equation}
\left\{ Q_g(\bar{x}) \in \mathbb{E} \mid \bar{x} \in Q_{g^{-1}} \left(\mbox{{\rm ker}} \hspace{0.3mm} \mathcal{A} \right) \cap \mathcal{K}, \| \bar{x} \|^2_J = 1 \right\} \subseteq \mbox{{\rm ker}} \hspace{0.3mm} \mathcal{A} \cap \mathcal{K}. \notag
\end{equation}
This implies $\det (y^*) \geq \det \left( Q_g (\bar{x}^*) \right)$, and by Proposition \ref{q-det-relation}, we have $\det (y^*) \geq \det (g)^2 \det (\bar{x}^*) = \xi^{|H|} \det (\bar{x}^*) = \xi \det (\bar{x}^*)$.
Thus, $\det(x^*) > \det (y^*) \geq \xi \det (\bar{x}^*)$ holds, and we can conclude  that 
$\delta^{{\rm supposed}} \left( \mbox{{\rm ker}} \hspace{0.3mm} \mathcal{A} \cap  {\rm int} \hspace{0.75mm} \mathcal{K} \right) > \xi \cdot \delta^{{\rm supposed}} \left( Q_{g^{-1}} \left(\mbox{{\rm ker}} \hspace{0.3mm} \mathcal{A} \right) \cap  {\rm int} \hspace{0.75mm} \mathcal{K} \right)$. 
\end{proof}
Next, using Proposition \ref{prop:compare scaling}, we derive the maximum number of iterations until the proposed method finds $x \in \mbox{{\rm ker}} \hspace{0.3mm} \mathcal{A} \cap {\rm int} \hspace{0.75mm} \mathcal{K}$ by using $\delta \left( \mbox{{\rm ker}} \hspace{0.3mm} \mathcal{A} \cap {\rm int} \hspace{0.75mm} \mathcal{K} \right)$ as in Pena and Soheili's method.
\begin{proposition}
\label{pro:compare iteration}
Suppose that  $\mbox{{\rm ker}} \hspace{0.3mm} \mathcal{A} \cap  \mbox{{\rm int}} \hspace{0.75mm} \mathcal{K} \neq  \emptyset$ holds.
Algorithm \ref{main algorithm} returns $x \in \mbox{{\rm ker}} \hspace{0.3mm} \mathcal{A} \cap  {\rm int} \hspace{0.75mm} \mathcal{K}$ after at most $\log_\xi \delta \left( \mbox{{\rm ker}} \hspace{0.3mm} \mathcal{A} \cap  {\rm int} \hspace{0.75mm} \mathcal{K} \right)$ iterations.
\end{proposition}
\begin{proof}
Let $\mbox{{\rm ker}} \hspace{0.3mm} \bar{\mathcal{A}}$ be the linear subspace at the start of $k$ iterations of Algorithm \ref{main algorithm} and suppose that  $\delta^{{\rm supposed}} \left( \mbox{{\rm ker}} \hspace{0.3mm} \bar{\mathcal{A}} \cap  {\rm int} \hspace{0.75mm} \mathcal{K} \right) = 1/r^{\frac{r}{2}}$ holds.
Then, from Proposition \ref{prop:compare scaling}, we find that $\delta^{{\rm supposed}} \left( \mbox{{\rm ker}} \hspace{0.3mm} \mathcal{A} \cap  {\rm int} \hspace{0.75mm} \mathcal{K} \right) > \xi^k / r^{\frac{r}{2}}$.
This implies that $\delta \left( \mbox{{\rm ker}} \hspace{0.3mm} \mathcal{A} \cap  \mbox{int} \hspace{0.75mm} \mathcal{K} \right) > \xi^k$ since $\delta \left( \mbox{{\rm ker}} \hspace{0.3mm} \mathcal{A} \cap  \mbox{int} \hspace{0.75mm} \mathcal{K} \right) = r^{\frac{r}{2}} \cdot \delta^{{\rm supposed}} \left( \mbox{{\rm ker}} \hspace{0.3mm} \mathcal{A} \cap  \mbox{int} \hspace{0.75mm} \mathcal{K} \right)$ holds.
By taking the logarithm base $\xi$,  we obtain $\log_\xi \delta \left( \mbox{{\rm ker}} \hspace{0.3mm} \mathcal{A} \cap \mbox{int} \hspace{0.75mm} \mathcal{K} \right ) > k$.
\end{proof}

From here on, using the above results, we will compare the computational complexities of the methods in the case that $\mathcal{K}$ is simple and $\mbox{{\rm ker}} \hspace{0.3mm} \mathcal{A} \cap \mbox{{\rm int}} \hspace{0.75mm} \mathcal{K} \neq \emptyset$ holds. Table \ref{compare-PS-MA} summarizes the upper bounds on the number of iterations of the main algorithm (UB\#iter) of the two methods and the computational costs required per iteration (CC/iter). As in the previous section, the main algortihm requires $\mathcal{O} (m^3+m^2d)$ to compute the projection $\mathcal{P}_\mathcal{A}$. Here, BP shows the computational cost of the basic procedure in each method.
\begin{table}[H]
\caption{Comparison of our method and Pena and Soheili's method in the main algorithm}
\label{compare-PS-MA}
\begin{center}
\begin{tabular}{c|cc}\toprule
Method	& UB\#iter		& CC/iter  \\ \midrule
Proposed method 	&$\log_\xi \delta \left( \mbox{{\rm ker}} \hspace{0.3mm} \mathcal{A} \cap  {\rm int} \hspace{0.75mm} \mathcal{K} \right)$	&$m^3+m^2d+$ BP \\
Pena and Soheili's method &$- \log_{1.5} \delta \left( \mbox{{\rm ker}} \hspace{0.3mm} \mathcal{A} \cap  {\rm int} \hspace{0.75mm} \mathcal{K} \right)$	&$m^3+m^2d+$ BP	 \\ \bottomrule
\end{tabular}
\end{center}
\end{table}
The upper bound on the number of iterations of Algorithm \ref{main algorithm} is given by
$\log_\xi \delta \left( \mbox{{\rm ker}} \hspace{0.3mm} \mathcal{A} \cap  {\rm int} \hspace{0.75mm} \mathcal{K} \right) = \log_{1.5} \delta \left( \mbox{{\rm ker}} \hspace{0.3mm} \mathcal{A} \cap  {\rm int} \hspace{0.75mm} \mathcal{K} \right)/\log_{1.5} \xi$,
where we should note that $0 < \xi < 1$.
Since $0 < \frac{1}{-\log_{1.5} \xi} \leq 1$ when $\xi \leq 2/3$,  if $\xi \leq 2/3$, then the upper bound on the number of iterations of Algorithm \ref{main algorithm} is smaller than that of the main algorithm of Pena and Soheili's method.

Next, Table \ref{compare-PS-BP} summarizes upper bounds on the number of iterations of basic procedures in the proposed method (UB\#iter) and Pena and Soheili's method and the computational cost required per iteration (CC/iter). It shows cases of using the von Neumann scheme and the smooth perceptron in each method (corresponding to Algorithm \ref{basic procedure} and Algorithm \ref{bp-alg-sp} in the proposed method). As in the previous section, $ C^{\rm sd}$ denotes the computational cost required for spectral decomposition, and $C^{\min}$ denotes the computational cost required to compute only the minimum eigenvalue and the corresponding primitive idempotent.
\begin{table}[H]
\caption{Comparison of our method and Pena and Soheili's method in the basic procedure}
\label{compare-PS-BP}
\begin{center}
\begin{tabular}{c|cc|cc}\toprule
		&\multicolumn{2}{c|}{von Neumann scheme}	&\multicolumn{2}{c}{smooth perceptron} \\
Method	& UB\#iter		& CC/iter   & UB\#iter		& CC/iter   \\ \midrule
Proposed method	&$\frac{r^2}{\xi^2}$	&$\max ( C^{\rm sd} , md)$	&$\frac{2 \sqrt{2}r}{\xi} -1$	&$\max ( C^{\rm sd} , md)$ \\
Pena and Soheili's method &$16r^4$	&$\max (C^{\min} , md)$	&$8 \sqrt{2} r^2-1$		&$\max ( C^{\rm sd} , md)$ \\ \bottomrule
\end{tabular}
\end{center}
\end{table}
Note that by setting $\xi = (4r)^{-1}$, the upper bounds on the number of iterations of the basic procedure of the two methods are the same. If $\xi =  (4r)^{-1}$, then $\frac{1}{-\log_{1.5} \xi} = \frac{1}{\log_{1.5} 4r} \leq \frac{1}{\log_{1.5} 4} = 0.292$, and the upper bound of the number of iterations of Algorithm \ref{main algorithm} is less than $0.3$ times the upper bound of the number of iterations of the main algorithm of Pena and Soheili's method, which implies that the larger the value of $r$ is, the smaller the ratio of those bounds becomes.
From the discussion in Section \ref{sec: CSD and Cmin}, we can assume $\mathcal{O}(C^{\rm sd} = \mathcal{C^{\min}})$, and Table \ref{compare-PS-BP} shows that the proposed method is superior for finding a point $x \in \mbox{{\rm ker}} \hspace{0.3mm} \mathcal{A} \cap {\rm int} \hspace{0.75mm} \mathcal{K}$.

\subsection{Computational costs of $ C^{\rm sd}$ and $C^{\min}$}
\label{sec: CSD and Cmin}

This section discusses the computational cost required for spectral decomposition $ C^{\rm sd}$ and the computational cost required to compute only the minimum eigenvalue and the corresponding primitive idempotent $C^{\min}$.

There are so-called direct and iterative methods for eigenvalue calculation algorithms, briefly described on pp.139-140 of \cite{Demmel1997}.
(Note that it is also written that there is no direct method in the strict sense of an eigenvalue calculation since finding eigenvalues is mathematically equivalent to finding zeros of polynomials).

In general, when using the direct method of $\mathcal{O}(n^3)$, we see that $ C^{\rm sd}=\mathcal{O}(n^3)$ and $C^{\min}=\mathcal{O}(n^3)$. The Lanczos algorithm is a typical iterative algorithm used for sparse matrices. Its cost per iteration of computing the product of a matrix and a vector once is $\mathcal{O}(n^2)$. Suppose the number of iterations at which we obtain a sufficiently accurate solution is constant with respect to the matrix size. In that case, the overall computational cost of the algorithm is $\mathcal{O}(n^2)$. Corollary 10.1.3 in \cite{Golub2013} discusses the number of iterations that yields sufficient accuracy. It shows that we can expect fewer iterations if the value of "the difference between the smallest and second smallest eigenvalues / the difference between the second smallest and largest eigenvalue" is larger. However, it is generally difficult to assume that the above value does not depend on the matrix size and is sufficiently large. Thus, even in this case, we cannot take advantage of the condition that we only need the minimum eigenvalue, and we conclude that it is reasonable to consider that $\mathcal{O}(C^{\rm sd})=\mathcal{O}(C^{\min})$.

\section{Numerical experiments}
\label{sec: numerical experiments}
\subsection{Outline of numerical implementation}
\label{sec: outline of numerical implementation}

Numerical experiments were performed
using the authors' implementations of the algorithms
on a positive semidefinite optimization problem with one positive semidefinite cone $\mathcal{K} = \mathbb{S}^n_{+}$ of the form 
\begin{equation}
\begin{array}{llllll}
{\rm P} (\mathcal{A})& \mbox{find}   &X \in \mathbb{S}^n_{++}	&\mbox{s.t.}	&\mathcal{A} (X) = \bm{0} \in \mathbb{R}^m  \notag
\end{array}
\end{equation}
where $\mathbb{S}^n_{++}$ denotes the interior of $\mathcal{K} = \mathbb{S}^n_{+}$.
We created strongly feasible ill-conditioned instances, i.e.,   $\mbox{ker} \hspace{0.3mm} \mathcal{A} \cap  \mathbb{S}^n_{++} \neq  \emptyset$ and $X \in \mbox{ker} \hspace{0.3mm} \mathcal{A} \cap  \mathbb{S}^n_{++}$ has positive but small eigenvalues.
We will explain how to make a such instance in section \ref{sec: instances}.
In what follows, we refer to Louren\c{c}o et al.'s method~\cite{Bruno2019} as Louren\c{c}o (2019), and Pena and Soheili's method~\cite{Pena2017} as Pena (2017).
We set the termination parameter as $\xi = 1/4$ in our basic procedure.
The reason for setting $\xi=1/4$ is to prevent the square root of $\xi$ from becoming an infinite decimal, and to prevent the upper bound on the number of iterations of the basic procedure from becoming too large.
We also set the accuracy parameter as $\varepsilon$ = 1e-12, both in our main algorithm and in Louren\c{c}o (2019) and  determined whether ${{\rm P}_{S_\infty} (\mathcal{A})}$ or  ${{\rm P}_{S_1} (\mathcal {A})}$ has a solution whose minimum eigenvalue is greater than or equal to $\varepsilon$.
Note that \cite{Pena2017} proposed various update methods for the basic procedure.
In our numerical experiments, all methods employed the modified von Neumann scheme (Algorithm \ref{bp-alg-mvn}) with the identity matrix as the initial point and the smooth perceptron scheme (Algorithm \ref{bp-alg-sp}).
This implies that the basic procedures used in the three methods differ only in the termination conditions for moving to the main algorithm and that all other steps are the same.
All executions were performed using MATLAB R2022a on an Intel (R) Core (TM) i7-6700 CPU @ 3.40GHz machine with 16GB of RAM. Note that we computed the projection $\mathcal{P}_\mathcal{A}$ using the MATLAB function for the singular value decomposition. The projection $\mathcal{P}_\mathcal{A}$ was given by $\mathcal{P}_\mathcal{A} = I - A^\top (AA^\top)^{-1} A$ using the matrix $A \in \mathbb{R}^{m \times d}$ which represents the linear operator $\mathcal{A}(\cdot)$ and the identity matrix $I$. Here, suppose that the singular value decomposition of a matrix $A$ is given by 
$
A = U \Sigma V^\top
=
U
(\Sigma_m \ O)
V^\top
$ where $U \in \mathbb{R}^{m \times m}$ and $V \in \mathbb{R}^{d \times d}$ are orthogonal matrices, and $\Sigma_m \in \mathbb{R}^{m \times m}$ is a diagonal matrix with $m$ singular values on the diagonal. 
Substituting this decomposition into $A^\top (AA^\top)^{-1} A$, we have 
\begin{align*}
A^\top (AA^\top)^{-1} A
&= A^\top (U \Sigma \Sigma^\top U^\top)^{-1} A \\
&= A^\top U^{-\top} (\Sigma_m^2)^{-1} U^{-1} A \\
&= V \Sigma^\top \Sigma_m^{-2} \Sigma V^\top 
= V
\begin{pmatrix}
I_m & O \\
O & O
\end{pmatrix}
V^\top = V_{:, 1:m} V_{:,1:m}^\top,
\end{align*}
where $V_{:, 1:m}$ represents the submatrix from column $1$ to column $m$ of $V$.
Thus, for any $x \in \mathbb{E}$, we can compute $\mathcal{P}_\mathcal{A}(x) = x - V_{:, 1:m} V_{:,1:m}^\top x$.

In what follows, $\bar{X} \in \mathbb{S}^n$ denotes the output obtained from the main algorithm and $X^*$ the result scaled as the solution of the original problem ${\rm P}(\mathcal{A})$ multiplied by a real number such that $\lambda_{max} (X^*) = 1$. When $X^*$ was obtained, we defined the residual of the constraints as the value of $\| \mathcal{A}(X^*) \|_2$.



We also solved the following problem with a commercial code, Mosek~\cite{Mosek}, and compared it with the output of Chubanov's methods:
\begin{equation}
\begin{array}{cccccc}
{\rm (P)}	&\min		&0	&{\rm s.t}	&\mathcal{A} (X) = \bm{0},	&X \in \mathbb{S}^n_{+}\modifySecond{,}	\\
{\rm (D)}	&\max	&\bm{0}^\top y	&{\rm s.t}	&-\mathcal{A}^* y \in \mathbb{S}^n_{+}.	
\end{array}
\notag
\end{equation}
Here, Mosek solves the self-dual embedding model by using a path-following interior-point method, so if we obtain a solution $(X^*, y^*)$, then $X^*$ and $-\mathcal {A}^* y^*$ lie in the (approximate) relative interior of the primal feasible region and the dual feasible region, respectively~\cite{Handbook}. That is, $X^*$ obtained by solving a strongly feasible problem with Mosek is in $\mathbb{S}^n_{++}$, $X^*$ obtained by solving a weakly feasible problem is in $\mathbb{S}^n_+ \setminus \mathbb{S}^n_{++}$, and $X^*$ obtained by solving an infeasible problem is $X^*=O$ (i.e., $-\mathcal{A}^* y^* \in \mathbb{S}^n_{++}$). As well as for Chubanov's methods, we computed $\| \mathcal{A}(X^*) \|_2$ for the solution obtained by Mosek after scaling so that $\lambda_{max} (X^*)$ would be 1.
Note that (P) and (D) do not simultaneously have feasible interior points.
In general, it is difficult to solve such problems stably by using interior point methods, but since strong complementarity exists between (P) and (D), they can be expected to be stably solved. By applying Lemma 3.4 of~\cite{Lourenco2021}, we can generate a problem in which both the primal and dual problems have feasible interior points in which it can be determined whether (P) has a feasible interior point. However, since there was no big difference between the solution obtained by solving the problem generated by applying Lemma 3.4 of~\cite{Lourenco2021} and the solution obtained by solving the above (P) and (D), we showed only the results of solving (P) and (D) above.


\subsection{How to generate instances}
\label{sec: instances}
Here, we describe how the strongly feasible ill-conditioned instances were generated.
In what follows, for any natural numbers $m,n$, $\mbox{rand}(n)$ is a function that returns $n$-dimensional real vectors whose elements are uniformly distributed in the open segment $(0,1)$, and $\mbox{rand}(m,n)$ is a function that returns an $m \times n$ real matrix whose elements are uniformly distributed in the open segment $(0,1)$. Furthermore, for any $x \in \mathbb{R}^n$ and $X \in \mathbb{R}^{m\times n}$, $\mbox{diag}(x) \in \mathbb{R}^{n \times n}$ is a function that returns a diagonal matrix whose diagonal elements are the elements of $x$, and $\mbox{vec}(X) \in \mathbb{R}^{mn}$ is a function that returns a vector obtained by stacking the $n$ column vectors of $X$.
The strongly feasible ill-conditioned instances were generated by extending the method of generating ill-conditioned strongly feasible instances proposed in~\cite{Pena2019} to the symmetric cone case.
%

\begin{proposition}
\label{prop: strongly_feasible}
Suppose that $\bar{x} \in \rm{int} \mathcal{K}$, $\| \bar{x} \|_\infty \leq 1$ and $\bar{u} \in \mathcal{K} , \| \bar{u} \|_1 = r$ satisfy $\langle \bar{x} , \bar{u} \rangle = r$.
Define the linear operator $\mathcal{A} : \mathbb{E} \rightarrow \mathbb{R}^m$ as $\mathcal{A}(x) = ( \langle a_1, x \rangle , \langle a_2, x \rangle , \dots ,\langle a_m, x \rangle)^T$ for which $a_1 = \bar{u} - \bar{x}^{-1}$ and $\langle a_j , \bar{x} \rangle = 0$ hold for any $j = 2,\dots,m$.
Then, 
\begin{equation}
\bar{x} = \argmax_x \left \{ \det(x) : x \in \mathcal{K} \cap \mbox{\rm ker}\mathcal{A}, \| x \|_\infty = 1 \right \}. 
\label{eq:0}
\end{equation}
\end{proposition}

\begin{proof}
First, note that the assertion (\ref{eq:0}) is equivalent to
\begin{equation}
\bar{x} = \argmax_{x \in \mathcal{F}}  \left \{ \log \det(x) \right \}  
\ \mbox{where} \ 
\mathcal{F} := \left \{ x \in \mathcal{K} \cap \mbox{ker}\mathcal{A} : \| x \|_\infty \leq 1 \right \}.
\label{eq:1}
\end{equation}
From the assumptions, we see that $\bar{x} \in \mathcal{K}$, $\| \bar{x} \|_\infty \leq 1$ and $\langle a_1 , \bar{x} \rangle = \langle \bar{u} - \bar{x}^{-1} , \bar{x} \rangle = r - r = 0$; thus, $\mathcal{A}(\bar{x}) = 0$ and $\bar{x} \in \mathcal{F}$. Since $\nabla \log \det (x) = x^{-1}$, if $\bar{x}$ satisfies 
\begin{equation}
\langle x - \bar{x} , \bar{x}^{-1} \rangle \leq 0 \ \mbox{for any} \ x \in \mathcal{F} \label{eq:2}
\end{equation}
we can conclude that (\ref{eq:1}) holds. In what follows, we show that (\ref{eq:2}) holds.

For any $x \in \mathcal{F}$, $x \in \mbox{ker}\mathcal{A}$ and hence, $\langle a_1 , x \rangle = \langle \bar{u} - \bar{x}^{-1} , x \rangle = \langle \bar{u} - \bar{x}^{-1} , x \rangle = 0$, i.e., $\langle \bar{u} , x \rangle = \langle \bar{x}^{-1} , x \rangle$.
Thus, we obtain
\begin{align*}
\langle x - \bar{x} , \bar{x}^{-1} \rangle &= \langle \bar{u} , x \rangle - r \\
&\leq \langle \bar{u} , x \rangle - \| \bar{u} \|_1 \| x \|_\infty \hspace{1cm}(\mbox{by} \ \| \bar{u} \|_1 = r \ \mbox{and} \  \| x \|_\infty \leq 1) \\
&\leq 0 \hspace{1cm}(\mbox{by} \ \langle \bar{u} , x \rangle \leq \| \bar{u} \|_1 \| x \|_\infty)
\end{align*}
which completes the proof.
\end{proof}

Proposition \ref{prop: strongly_feasible} guarantees that we can generate a linear operator $\mathcal{A}$ satisfying $\mbox{ker} \hspace{0.3mm} \mathcal{A} \cap  \mathbb{S}^n_{++} \neq  \emptyset$ by determining an appropriate value $\mu = \displaystyle \max_{ X \in \mathcal{F}} \det(X)$, where $\mathcal{F} = \{ X \in \mathbb{S}^n : X \in \mathbb{S}^n_{++} \cap \mbox{ker} \hspace{0.3mm} \mathcal{A} , \|X\|_\infty = 1\}$.
The details on how to generate the strongly feasible instances are in Algorithm \ref{strongly feasible problem}. The input consists of the rank of the semidefinite cone $n$, the number of constraints $m$, an arbitrary orthogonal matrix $P$, and the parameter $\tau \in \mathbb{R}_{++}$ which determines the value of $\mu$. We made instances for which the value of $\mu$ satisfies $1e-\tau \leq \mu \leq 1e-(\tau-1)$. In the experiments, we set $\tau \in \{50,100,150, 200, 250 \}$ so that $\mu$ would vary around 1e-50, 1e-100, 1e-150, 1e-200, and 1e-250.

Note that Algorithm \ref{strongly feasible problem} generates instances using $\bar{x}$ that has a natural eigenvalue distribution. For example, let $n-1=3$ and consider two $X$s where one has 3 eigenvalues of about 1e-2, and the others have 1 each of 1e-1, 1e-2, and 1e-3. $\det(X) \simeq $1e-6 is obtained for both $X$s, but the latter is more natural for the distribution of eigenvalues. In our experiment, we generated ill-conditioned instances by using $X$ having a natural eigenvalue distribution as follows:
\begin{enumerate}
\item Find an integer $s$ that satisfies 1e-$s$ $\leq l^{\frac{1}{n-1}} \leq u^{\frac{1}{n-1}} \leq$ 1e-$(s-1)$.
\item Generate $t = 2s-1$ eigenvalue classes.
\item Decide how many eigenvalues to generate for each class.
\end{enumerate}

For example, when $n=13$ and $\tau = 30$, Algorithm \ref{strongly feasible problem} yields $s=3$, $t=5$, $a=2$ and $b=2$, and since $b$ is even, we have $num = (2,3,2,3,2)^\top$. The classes of $t=5$ eigenvalues are shown in Table \ref{example1} below. \modifyFirst{Note that $( l^{\frac{1}{n-1}} \cdot 10^{s-i} ) \cdot ( l^{\frac{1}{n-1}} \cdot 10^{s-(t-i+1)} )= l^{\frac{2}{n-1}}$ and $( u^{\frac{1}{n-1}} \cdot 10^{s-i} ) \cdot ( u^{\frac{1}{n-1}} \cdot 10^{s-(t-i+1)} )= u^{\frac{2}{n-1}}$ hold for the $i$-th and $t-i+1$-th classes.} This implies that we obtain $1e-\tau \leq \mu = \det(X) \leq 1e-(\tau-1)$ both when generating $n-1$ eigenvalues in the $s$th class and when generating $n-1$ eigenvalues of $X$ according to $num$. When $n=14, \tau = 30$, Algorithm \ref{strongly feasible problem} gives $s=3$, $t=5$, $a=2$, and $b=3$, and since $b$ is an odd number, we have $num = (2,3,3,3,2)^\top$. Thus, Algorithm \ref{strongly feasible problem} generates the instances by controlling the frequency so that the geometric mean of the $n-1$ eigenvalues of $X$ falls within the $s$-th class width. 
\begin{table}[H]
\caption{Frequency distribution table of eigenvalues of $X$ generated by Algorithm \ref{strongly feasible problem} when $n=13$ or $n=14$, $\tau=30$}
\label{example1}
\begin{center}
\begin{tabular}{ccccc} \toprule
Class	&\multicolumn{2}{c}{Class width of eigenvalues of $\bar{x}$}					&\multicolumn{2}{c}{Frequency($num$)}\\ \midrule
	& Lower bound 		& Upper bound						&$n=13$		&$n=14$	\\ \midrule
1	&$ l^{\frac{1}{n-1}} \cdot 10^{2}$	&$u^{\frac{1}{n-1}} \cdot 10^{2}$	&2			&2\\
2	&$ l^{\frac{1}{n-1}} \cdot 10^{1}$	&$u^{\frac{1}{n-1}} \cdot 10^{1}$	&3			&3\\
3	&$ l^{\frac{1}{n-1}}$			&$u^{\frac{1}{n-1}}$			&2			&3\\
4	&$ l^{\frac{1}{n-1}} \cdot 10^{-1}$&$u^{\frac{1}{n-1}} \cdot 10^{-1}$	&3			&3\\
5	&$ l^{\frac{1}{n-1}} \cdot 10^{-2}$&$u^{\frac{1}{n-1}} \cdot 10^{-2}$	&2			&2\\ \bottomrule
\end{tabular}
\end{center}
\end{table}

\begin{algorithm}
\caption{Strongly feasible instance}
\label{strongly feasible problem}
\begin{algorithmic}[1]
 \renewcommand{\algorithmicrequire}{\textbf{Input: }}
 \renewcommand{\algorithmicensure}{\textbf{Output: }}

 \STATE \algorithmicrequire $n, m, \tau, P$
 \STATE \algorithmicensure $A$ 
 \STATE $l \leftarrow 1{\rm e}-\tau$, $u \leftarrow 1{\rm e}-(\tau-1)$, $s \leftarrow \lceil \frac{\tau}{n-1} \rceil$ and $t \leftarrow 2s-1$
 \STATE $b \leftarrow (n-1) \bmod t$, $a \leftarrow \frac{(n-1) - b}{t}$ and $num \leftarrow a \cdot \bm{1} \in \mathbb{R}^t$
 \IF {$b$ is odd}
 \STATE $\bar{b} \leftarrow \frac{b-1}{2}$ and $num_i \leftarrow num_i + 1$ such that $s - \bar{b} \leq i \leq s + \bar{b}$
 \ELSE 
 \STATE $\bar{b} \leftarrow \frac{b}{2}$ and $num_i \leftarrow num_i + 1$ such that $s - \bar{b} \leq i < s$ or $s < i \leq  s + \bar{b}$
 \ENDIF
 \STATE $d_1 \leftarrow 1$ and $k \leftarrow 2$
 \FOR {$i = 1$ \mbox{to} $t$}
 \FOR {$j=1$ \mbox{to} $num_i$}
 \STATE $dl \leftarrow l^{\frac{1}{n-1}} \cdot 10^{s-i}$and $du \leftarrow u^{\frac{1}{n-1}} \cdot 10^{s-i}$
 \STATE $d_k \leftarrow  dl + \left( du-dl \right) \mbox{rand} \left( 1 \right)$ 
 \STATE $k \leftarrow k+1$
 \ENDFOR
 \ENDFOR
 \STATE $D' \leftarrow \mbox{diag}(d)$ and then compute $C \leftarrow PD'P^T $ and $c \leftarrow \mbox{vec}(C)  $
 \STATE $u \leftarrow (n , 0_{n-1}^T )^T$ where $0_{n-1}$ denotes the $n-1$-dimensional vector of zeros   
 \STATE $U \leftarrow P(\mbox{diag}(u) - {D'}^{-1})P^T$, $A' \leftarrow \mbox{vec}(U)$ and $R \leftarrow I - \frac{1}{\|c\|_2^2} cc^T $
 \FOR {$ i = 1 \mbox{ to  }  m-1 $}
 \STATE $A_i' \leftarrow \mbox{rand}(n,n)$ and $A_i \leftarrow \left( A_i' + (A_i')^T \right)/2 $
 \STATE $A' \leftarrow 
\begin{pmatrix}
A' \\
\mbox{vec}(A_i)^T \\
\end{pmatrix}$
 \ENDFOR
 \STATE $\bar{A} \leftarrow A'R$
 \STATE $A \leftarrow \begin{pmatrix}
\mbox{vec}(U)^T \\
\bar{A}
\end{pmatrix}$
\end{algorithmic} 
\end{algorithm}

%
%
\subsection{Numerical results and observations}
\label{sec: numerical results and observations}

We set the size of the positive semidefinite matrix to $n = 50$, so that the computational experiments could be performed in a reasonable period of time. To eliminate bias in the experimental results, we generated instances in which the number of constraints $m$ was controlled using the parameter $\nu$ for the number $\frac{n(n+1)}{2}$ of variables in the symmetric matrix of order $n$. Specifically, the number of constraints $m$ on an integer was obtained by rounding the value of $\frac{n(n+1)}{2} \nu$, where $\nu \in \{0.1, 0.3, 0.5, 0.7, 0.9\}$.
For each $\nu \in \{0.1, 0.3, 0.5, 0.7, 0.9\}$, we generated five instances, i.e., 25 instances for each of five strongly feasible cases (corresponding to five patterns of $\mu \simeq \mbox{1e-50}, \dots, \mu \simeq \mbox{1e-250}$, see section \ref{sec: instances} for details).
Thus, we generated 125 strongly feasible instances. We set the upper limit of the execution time to 2 hours and compared the performance of our method with those of Louren\c{c}o (2019), Pena(2017) and Mosek.
When using Mosek, we set the primal feasibility tolerance to 1e-12.

Tables \ref{strongly-feasible-MVN-SP1} and \ref{strongly-feasible-MVN-SP2} list the results for the (ill-conditioned) strongly feasible case. The ``CO-ratio'' column shows the ratio of \modifySecond{$|N_2|/|N_1|$ where $N_1$ is the set of problems for which the algorithm terminated within 2 hours, the upper limit of the execution time, and $N_2 \subseteq N_1$ is the set of problems for which a correct output is obtained}, the ``times(s)'' column shows the average CPU time of the method, the ``M-iter'' column shows the average iteration number of each main algorithm, the $\|\mathcal{A}(X^*)\|_2$ column shows the residual of the constraints, \modifyFirst{and the $\lambda_{\min}(X^*)$ column shows the minimum eigenvalue of $X^*$.} The ``BP'' column shows which scheme  (the modified von Neumann (MVN) or the smooth perceptron (SP)) was used in the basic procedure.
The values in parentheses () in row $\mu \approx \mbox{1e-100}$ are the average values excluding instances for which the method ended up running out of time. 

First, we compare the results when using MVN or SP as the basic procedure in each method.
From Table \ref{strongly-feasible-MVN-SP1}, we can see that for strongly-feasible problems, using SP as the basic procedure has a shorter average execution time than using MVN.
Next, we compare the results of each method.
For $\mu \simeq \mbox{1e-50}$, there was no significant difference in performance among the three methods.
For $\mu \leq \mbox{1e-100}$, the results in the rows BP=MVN show that our method and Louren\c{c}o (2019) obtained interior feasible solutions for all problems, while Pena (2017) ended up running out of time for 99 instances.
This is because Pena (2017) needs to call its basic procedure to find a solution of $\mbox{range} \hspace{0.3mm} \mathcal{A}^* \cap \mathbb{S}^n_{++}$. Comparing our method  with Louren\c{c}o (2019), we see that \modifySecond{Algorithm \ref{main algorithm}} is superior in terms of CPU time.
Finally, we compare the results for each value of $\mu$. As $\mu$ becomes smaller, i.e., as the problem becomes more ill-conditioned, the number of scaling times and the execution time increase, and the accuracy of the obtained solution gets worse.

\begin{table}
\caption{Results for strongly feasible instances (Correct output (CO-) ratio, CPU time and \modifyFirst{and M-iter})}
\label{strongly-feasible-MVN-SP1}
\begin{center}
\begin{tabular}{ll|rrr|rrr|rrr} \toprule
			&	&\multicolumn{3}{c|}{Algorithm \ref{main algorithm}} & \multicolumn{3}{c|}{Louren\c{c}o (2019)} & \multicolumn{3}{c}{Pena (2017)} \\ 
Instance		&BP	&CO-ratio&time(s)&M-iter&CO-ratio&time(s)&M-iter&CO-ratio&time(s)&M-iter	\\ \midrule
\multirow{2}{*}{$\mu \simeq $ 1e-50}		&MVN	&25/25	&7.81		&3.28	&25/25	&25.94	&14.48&25/25	&3.60 &1.00	\\
							&SP		&25/25	&0.75		&1.00	&25/25	&10.12	&14.08&25/25	&0.80 &1.00	\\ 
\multirow{2}{*}{$\mu \simeq $ 1e-100}	&MVN	&25/25	&51.62	&53.12&25/25	&448.05	&329.04&1/1	&(4513.59)&(2.00)	 \\
							&SP		&25/25	&32.11	&36.04&25/25	&256.24	&365.76&25/25	&123.65 &23.32	\\ 
\multirow{2}{*}{$\mu \simeq $ 1e-150}	&MVN	&25/25	&99.39	&118.12&25/25	&888.25	&728.68&-	&- &-	\\
							&SP		&25/25	&76.98	&91.96&25/25	&520.73	&756.36&25/25	&781.88 &117.32	\\ 
\multirow{2}{*}{$\mu \simeq $ 1e-200}	&MVN	&25/25	&144.48	&185.40&25/25	&1328.68	&1145.40&-	&- &-	\\
							&SP		&25/25	&118.06	&151.44&25/25	&789.29	&1150.20&25/25	&1874.20 &236.44	\\ 
\multirow{2}{*}{$\mu \simeq $ 1e-250}	&MVN	&25/25	&188.11	&251.44&25/25	&1827.24	&1601.20&-	&- &-	\\
							&SP		&25/25	&162.67	&215.12&25/25	&1074.07	&1564.80&25/25	&3308.35 &376.24	\\\bottomrule
\end{tabular}
\end{center}
\end{table}
\begin{table}
\caption{\modifyFirst{Results for ill-conditioned strongly feasible instances ($\|\mathcal{A}(X^*)\|_2$ and $\lambda_{\min} (X^*)$)}}
\label{strongly-feasible-MVN-SP2}
\begin{center}
\modifyFirst{
\begin{tabular}{ll|rr|rr|rr} \toprule
			&	&\multicolumn{2}{c|}{Algorithm \ref{main algorithm}} & \multicolumn{2}{c|}{Louren\c{c}o (2019)} & \multicolumn{2}{c}{Pena (2017)} \\ 
Instance		&BP	&$\| \mathcal{A} (X^*)\|_2$&$\lambda_{\min} (X^*)$&$\| \mathcal{A} (X^*)\|_2$&$\lambda_{\min} (X^*)$&$\| \mathcal{A} (X^*)\|_2$&$\lambda_{\min} (X^*)$ \\ \midrule
\multirow{2}{*}{$\mu \simeq $ 1e-50}		&MVN		&1.24e-11	&4.42e-4		&7.64e-12	&3.60e-4	&1.27e-11 &3.95e-4	\\
							&SP			&1.23e-11	&8.48e-4		&8.22e-12	&8.10e-4	&1.23e-11 &8.48e-4	\\ 
\multirow{2}{*}{$\mu \simeq $ 1e-100}	&MVN		&9.98e-12	&2.73e-6		&1.26e-11	&3.01e-6	&(1.07e-8) &(3.34e-6)\\
							&SP			&4.18e-11	&3.19e-5		&1.10e-11	&3.75e-5	&5.38e-9 &3.39e-6	\\ 
\multirow{2}{*}{$\mu \simeq $ 1e-150}	&MVN		&1.96e-10	&5.24e-8		&4.29e-10	&4.10e-8	&- &-\\
							&SP			&2.21e-10	&3.98e-7		&5.60e-10	&5.54e-7	&6.31e-9 &3.78e-7	\\ 
\multirow{2}{*}{$\mu \simeq $ 1e-200}	&MVN		&1.51e-8	&7.86e-10		&4.76e-8	&1.18e-9	&- &-\\
							&SP			&1.09e-8	&4.87e-9		&3.81e-8	&5.93e-9	&1.72e-8 &5.06e-9	\\ 
\multirow{2}{*}{$\mu \simeq $ 1e-250}	&MVN		&9.51e-7	&8.43e-12		&2.58e-6	&2.52e-11	&- &-\\
							&SP			&1.72e-6	&5.14e-11		&3.35e-6	&7.05e-11	&1.73e-6 &5.40e-11	\\ \bottomrule
\end{tabular}
}
\end{center}
\end{table}

Table \ref{strongly-feasible-Mosek} summarizes the results of our experiments using Mosek to solve strongly feasible ill-conditioned instances. Mosek sometimes returned the error message ``rescode = 10006'' for the $\mu \leq 1e-200$ instances. This error message means that "the optimizer is terminated due to slow progress." In this case, the obtained solution is not guaranteed to be optimal, but it may have sufficient accuracy as a feasible solution. Therefore, we took the CO-ratio when the residual $\|\mathcal{A}(X^*)\|_2$ is less than or equal to 1e-5 to be the correct output. The reason why we set the threshold to 1e-5 is that the maximum value of $\|\mathcal{A}(X^*)\|_2$ was less than 1e-5 among the $X^*$ values obtained for the strongly feasible ill-conditioned instances by the three methods, Algorithm \ref{main algorithm}, Louren\c{c}o (2019) and Pena (2017). On the other hand, for the $\mu \leq 1e-200$ instances, the Chubanov methods had higher CO-ratios. That is, when the problem was quite ill-conditioned, the solution obtained by each of the Chubanov methods had a smaller value of $\| \mathcal{A}(X^*)\|_2$ compared with the solution obtained by Mosek, which implies that the accuracy of the solution obtained by each of the Chubanov methods was higher than that of Mosek.

\begin{table}
\caption{Results for ill-conditioned strongly feasible instances with Mosek}
\label{strongly-feasible-Mosek}
\begin{center}
\begin{tabular}{l|rrrr} \toprule
Instance			&CO-ratio	&time(s)	&$\| \mathcal{A} (X^*)\|_2$	&\modifyFirst{$\lambda_{\min} (X^*)$}	\\ \midrule
$\mu \simeq $ 1e-50	&25/25 	&1.96		&8.73e-13 	&7.99e-3	\\
$\mu \simeq $ 1e-100	&25/25 	&3.18		&1.87e-12 	&4.51e-5	\\
$\mu \simeq $ 1e-150	&25/25 	&3.72		&2.48e-10 	&4.45e-7	\\
$\mu \simeq $ 1e-200	&21/25 	&6.56 		&2.58e-7 	&4.35e-9	\\
$\mu \simeq $ 1e-250	&1/25 	&6.88 		&2.57e-7 	&5.37e-11	\\	\bottomrule
\end{tabular}
\end{center}
\end{table}

\section{Concluding remarks}
\label{sec: concluding remarks}

In this study, we proposed a new version of Chubanov's method for solving the feasibility problem over the symmetric cone by extending Roos's method \cite{Roos2018} for the feasible problem over the nonnegative orthant. Our method has the following features:
\begin{itemize}
\item Using the norm $\| \cdot \|_\infty$ in problem ${\rm P}_{S_\infty} (\mathcal{A})$ makes it possible to (i) calculate the upper bound for the minimum eigenvalue of any feasible solution of ${\rm P}_{S_\infty} (\mathcal{A})$, (ii) quantify the feasible region of ${\rm P} (\mathcal{A})$, and hence (iii) determine whether there exists a feasible solution of ${\rm P} (\mathcal{A})$ whose minimum eigenvalue is greater than $\epsilon$ as in \cite{Bruno2019}.
\item In terms of the computational bound, our method is (i) equivalent to Roos's original method \cite{Roos2018} and superior to Louren\c{c}o et al.'s method \cite{Bruno2019} when the symmetric cone is the nonnegative orthant, (ii) superior to Louren\c{c}o et al.'s when the symmetric cone is a Cartesian product of second-order cones, (iii) equivalent to Louren\c{c}o et al.'s when the symmetric cone is the simple positive semidefinite cone, under the assumption that the costs of computing the spectral decomposition and the minimum eigenvalue are of the same order for any given symmetric matrix, and (iv) superior to Pena and Soheili's method~\cite{Pena2017} for any simple symmetric cones under the assumption that ${\rm P} (\mathcal{A})$ is feasible.
\end{itemize}

We also conducted comprehensive numerical experiments comaring our method with the existing mtehods of Chubanov \cite{Bruno2019,Pena2017} and Mosek.
Our numerical results showed that
\begin{itemize}
\item It is considerably faster than the existing methods on ill-conditioned strongly feasible instances.
\item Mosek was the better than Chubanov’s methods in terms of execution time. On the other hand, in terms of the accuracy of the solution (the value of $\|\mathcal{A}(X^*)\|_2$), we found that all of Chubanov’s methods are better than Mosek. In particular, we have seen such results for strongly-feasible (terribly) ill-conditioned ($\mu \simeq 1e-250$) instances.
\end{itemize}

\modifyFirst{
In this paper, we performed computer experiments by setting $\xi=1/4$ in the basic procedure to avoid inducements for calculation errors, but there is room for further study on how to choose the value of $\xi$.
For example, if the problem size is large, the computation of the projection $\mathcal{P}_\mathcal{A}$ is expected to take much more time.
In this case, rather than setting $\xi=1/4$, running the algorithm as $\xi < 1/4$ may reduce the number of scaling steps to be performed before completion. As a result, the algorithm's run time may be shorter than when we set $\xi=1/4$.
More desirable approach may be to choose an appropriate value of $\xi$ at each iteration along to the algorithm's progress.
}

\section*{Acknowledgments}
We would like to express our deep gratitude to the reviewers and editors for their many valuable comments. 
Their comments significantly enriched the content of this paper, especially sections \ref{sec: extension}, \ref{sec: main algorithm}, \ref{sec: compare}, and  \ref{sec: numerical experiments}.
We also would like to sincerely thank Daisuke Sagaki for essential ideas on the proof of Proposition \ref{prop:compare tr}, and Yasunori Futamura for helpful information about the computational cost of the eigenvalue calculation in Section \ref{sec: CSD and Cmin}.
We could not complete this paper without their support.
This work was supported by JSPS KAKENHI Grant Numbers (B)19H02373 and JP 21J20875.

\appendix

\section{Basic procedure}
\label{app:modified basic procedure}
 \begin{algorithm}[H]
 \caption{Basic procedure (Modified von Neumann scheme)}
 \label{bp-alg-mvn}
 \begin{algorithmic}[1]
 \renewcommand{\algorithmicrequire}{\textbf{Input: }}
 \renewcommand{\algorithmicensure}{\textbf{Output: }}
 \renewcommand{\stop}{\textbf{stop }}
 \renewcommand{\return}{\textbf{return }}

 \STATE Same as lines 1-3 of Algorithm \ref{basic procedure}
 \WHILE{ $k \leq \frac{p^2 r_{\max}^2}{\xi^2}$}
 \STATE Same as lines 5-23 of Algorithm \ref{basic procedure}
 \FOR {$\ell \in \{1 , \dots ,p \}$}
 \STATE $S_\ell \leftarrow \{ i  \mid  {\lambda(z_\ell^k)}_i \leq 0\}$ and then $u_\ell \leftarrow \sum_{i \in S_\ell} {c(z_\ell^k)}_i$
 \ENDFOR
 \STATE $u \leftarrow \frac{1}{\sum_{\ell=1}^p |S_\ell|} u$ and $y^{k+1} \leftarrow \alpha y^k + (1-\alpha) u$, where $\alpha = \frac{ \langle P_\mathcal{A} (u) , P_\mathcal{A} (u) - z^k \rangle }{\|z^k-P_\mathcal{A} (u)\|^2_J} $
 \STATE $k \leftarrow k+1$ , $z^k \leftarrow P_\mathcal{A}(y^k)$ and $v^k \leftarrow y^k - z^k$
 \ENDWHILE
 \end{algorithmic} 
 \end{algorithm}
Below, we describe the results of updating $y^k$ with the smooth perceptron scheme as described in~\cite{Pena2017}.
Given $\mu > 0$, we define operator $u_\mu (\cdot) : \mathbb{E} \rightarrow \{ u \in \mathcal{K} \mid \langle u,e \rangle = 1\}$ as $u_\mu (v) := \argmin_{u \in \mathcal{K}, \langle u,e \rangle =1} \left\{ \langle u,v \rangle + \frac{\mu}{2} \| u - \bar{u} \|^2_J \right\}$.
 \begin{algorithm}
 \caption{Basic procedure (Smooth perceptron scheme)}
 \label{bp-alg-sp}
 \begin{algorithmic}[1]
 \renewcommand{\algorithmicrequire}{\textbf{Input: }}
 \renewcommand{\algorithmicensure}{\textbf{Output: }}
 \renewcommand{\stop}{\textbf{stop }}
 \renewcommand{\return}{\textbf{return }}

 \STATE \algorithmicrequire $P_\mathcal{A}$ and $\xi$ such that a constant $0 < \xi < 1$
 \STATE \algorithmicensure (i) a solution to ${\rm P}(\mathcal{A})$ or (ii) ${\rm D}(\mathcal{A})$ or (iii) a certificate that, for any feasible solution $x$ to ${\rm P}_{S_\infty}(\mathcal{A})$, $\langle e,x \rangle < r$
 \STATE initialization :  $\bar{u} \leftarrow \frac{1}{r}e$, $\mu^0 \leftarrow 2$, $u^0 \leftarrow \bar{u}$, $k \leftarrow 0, H_1 , \dots , H_p  = \emptyset$.
 \STATE compute $y^0 \leftarrow u_{\mu_0} \left( P_\mathcal{A}(u^0) \right), z^0 \leftarrow  P_\mathcal{A}(y^0),  v^0 \leftarrow y^0 - z^0$.
 \WHILE{$k \leq \frac{2 \sqrt{2} p r_{\max}}{\xi} -1$}
 \STATE Same as lines 5-23 of Algorithm \ref{basic procedure}
 \STATE $\theta^k \leftarrow \frac{2}{k+3}$ and $u^{k+1} \leftarrow (1-\theta^k) (u^k + \theta^k y^k) + (\theta^k)^2 u_{\mu^k} \left( P_\mathcal{A}(u^k) \right)$
 \STATE $\mu^{k+1} \leftarrow  (1-\theta^k) \mu^k$ and $y^{k+1} \leftarrow  (1-\theta^k) y^k +  \theta^k  u_{\mu^{k+1}} \left( P_\mathcal{A}(u^{k+1}) \right)$
 \STATE $k \leftarrow k+1$ , $z^k \leftarrow P_\mathcal{A}(y^k)$ and $v^k \leftarrow y^k - z^k$
 \ENDWHILE
 \end{algorithmic} 
 \end{algorithm}

\newpage
\mbox{ }
\newpage

\section{Another new main algorithm}

Here, we introduce another new main algorithm, Algorithm \ref{main algorithm 2}, whose computational cost may not be given in polynomials but might better determine $\epsilon$-feasible solutions.

\subsection{Outline of Algorithm \ref{main algorithm 2}}
\label{app:another-ma}

	The procedures of Algorithm \ref{main algorithm 2} are almost identical to Algorithm \ref{main algorithm}, except for one of the termination criteria (the criterion indicating the non-existence of $\varepsilon$-feasible solutions). Specifically, to set the upper bound for the minimum eigenvalue of any feasible solution $x$ of \modifyFirst{${\rm P}_{S_{\infty}}(\mathcal{A})$}, Algorithm \ref{main algorithm} focuses on the product $\det (\bar{x})$ of the eigenvalues of the arbitrary feasible solution $\bar{x}$ of the scaled problem ${\rm P}_{S_{\infty}}(\mathcal{A}^k Q^k)$, while Algorithm \ref{main algorithm 2} focuses on the sum $\langle \bar{x} , e \rangle$ of the eigenvalues. Algorithm \ref{main algorithm 2} works as follows.

\begin{enumerate}
\item If a feasible solution of ${\rm P}(\mathcal{A})$ or ${\rm D}(\mathcal{A})$ is returned from the basic procedure, the feasibility of ${\rm P}(\mathcal {A})$ can be determined, and we stop the main algorithm.
\item If the basic procedure returns the sets of indices $H_1 , \dots , H_p$ and the sets of primitive idempotents $C_1 , \dots , C_p$ that construct the corresponding Jordan frames, then
\begin{description}
\item in Algorithm \ref{main algorithm 2}:
\begin{enumerate}
\item if $\frac{r_\ell}{\left( r_\ell + \left( \frac{1}{\xi} -1 \right)m_\ell \right)} < \varepsilon$ holds for some $\ell \in \{ 1, \dots p \}$, we  determine that ${{\rm P}_{S_{\infty}}}(\mathcal{A})$ has no $\varepsilon$-feasible solution according to Proposition \ref{prop:lambda-min-upper-2} and stop the main algorithm, 
\item if $\frac{r_\ell}{\left( r_\ell + \left( \frac{1}{\xi} -1 \right)m_\ell \right)} \geq \varepsilon$ holds for any $\ell \in \{ 1, \dots p \}$, we rescale the problem and call the basic procedure again.
\end{enumerate}
\end{description}
\end{enumerate}

\begin{table}[H]
\caption{Upper bounds on the number of iterations of the main algorithms (cf. proposition \ref{prop:iteration-num-ma} and proposition \ref{prop:iteration-num-ma-2})}
\begin{center}
\label{compare_max_itr_MA}
\begin{tabular}{lc} \toprule
Main Algorithm				&Upper bound on \# of iterations	\\ \midrule
Algorithm \ref{main algorithm}		&$-\frac{r}{\log \xi} \log \left( \frac{1}{\varepsilon}\right) - p + 1$	\\
Algorithm \ref{main algorithm 2}	&$\frac{\xi}{1-\xi} \left( \frac{1}{\varepsilon} -1  \right) r - p + 1$	\\\bottomrule
\end{tabular}
\end{center}
\end{table}

Table \ref{compare_max_itr_MA} lists upper bounds on the numbers of iterations required by Algorithms \ref{main algorithm} and \ref{main algorithm 2}.
As shown in the table, Algorithm \ref{main algorithm} can be said to be a polynomial-time algorithm, but Algorithm \ref{main algorithm 2} is not. On the other hand, the results of the numerical experiments in Appendix \ref{app: numerical results} show that Algorithm \ref{main algorithm 2} is superior to Algorithm \ref{main algorithm} at detecting $\varepsilon$-feasibility for the generated instances.

 \begin{algorithm}[H]
 \caption{Main algorithm using another criteria for $\varepsilon$-feasibility}
 \label{main algorithm 2}
 \begin{algorithmic}[1]
 \renewcommand{\algorithmicrequire}{\textbf{Input: }}
 \renewcommand{\algorithmicensure}{\textbf{Output: }}
  \renewcommand{\stop}{\textbf{stop }}
 \renewcommand{\return}{\textbf{return }}

 \STATE \algorithmicrequire $\mathcal{A}$, $\mathcal{K}$, $\varepsilon$ and a constant $\xi$ such that $0 < \xi < 1$
 \STATE \algorithmicensure a solution to ${\rm P}(\mathcal{A})$ or ${\rm D}(\mathcal{A})$ or a certificate that there is no $\varepsilon$ feasible solution.
 \STATE $k \leftarrow 1$ , $\mathcal{A}^1 \leftarrow \mathcal{A}$ , $m_\ell \leftarrow 0 $ , $\bar{Q_\ell} \leftarrow I_\ell$, \modifyThird{$RP_\ell \leftarrow I_\ell$, $RD_\ell \leftarrow I_\ell$} for all $\ell \in \{1,\dots,p\}$
	  \STATE Compute \modifyFirst{$P_{\mathcal{A}^k}$} and call the basic procedure with \modifyFirst{$P_{\mathcal{A}^k}$}, $\frac{1}{r}e$, $\xi$
 \IF { basic procedure returns $z$ }
 \STATE \stop main algorithm and \return \modifyThird{$RPz$ ($RPz$ is a feasible solution of ${\rm P}(\mathcal{A})$)}
 \ELSIF { basic procedure returns $y$ or $v$ }
 \STATE \stop main algorithm and \return \modifyThird{$RDy$ or $RDv$ ( $RDy$ or $RDv$ is a feasible solution of $D(\mathcal{A})$)}
 \ELSIF  {basic procedure returns $H^k_1 , \dots , H^k_p$ and $C^k_1 , \dots , C^k_p$}
 \FOR {$\ell \in \{1,\dots,p\}$}
 \IF{$|H^k_\ell| > 0$}
 \STATE $g_\ell \leftarrow \sqrt{\xi} \sum_{h \in H^k_\ell} c^k(v_\ell)_h + \sum_{h \notin H_\ell^k}^{r_\ell} c^k(v_\ell)_h$ 
 \STATE $Q_\ell \leftarrow Q_{g_\ell}$, \modifyThird{$RP_\ell \leftarrow RP_\ell Q_{g_\ell}$, $RD_\ell \leftarrow RD_\ell Q_{g^{-1}_\ell}$}
 \STATE $m_\ell \leftarrow \left \langle \bar{Q_\ell} \left( \sum_{h \in H^k_\ell} c^k(v_\ell)_h \right) , e_\ell \right \rangle + m_\ell$
 \IF {$ \frac{r_\ell}{\left( r_\ell + \left( \frac{1}{\xi} - 1 \right) m_\ell \right)} \leq \varepsilon$}
 \STATE { \stop main algorithm. There is no $\varepsilon$ feasible solution.}
 \ENDIF
 \STATE $\bar{Q_\ell} \leftarrow  \bar{Q_\ell} Q_{g_\ell^{-1}}$
 \ELSE
 \STATE $ Q_\ell \leftarrow I_\ell$
 \ENDIF
 \ENDFOR
 \ELSE
 \STATE \return basic procedure error
 \ENDIF
 \STATE Let $Q^k = (Q_1 , \dots , Q_p)$
 \STATE $\mathcal{A}^{k+1} \leftarrow \mathcal{A}^k Q^k$ , $k \leftarrow k+1$. Go back to line 4.
 \end{algorithmic} 
 \end{algorithm}


\subsection{Finite termination of Algorithm \ref{main algorithm 2}}
Here, we prove finite termination of Algorithm \ref{main algorithm 2}.
Before going into the proof, we explain Algorithm \ref{main algorithm 2} in more detail than in Appendix \ref{app:another-ma}.
In Algorithm \ref{main algorithm 2}, when a cut is obtained in the $\ell$-th block, it computes the value of $\left \langle \bar{Q_\ell} \left( \sum_{h \in H^k_\ell} c^k(v_\ell)_h \right) , e_\ell \right \rangle$ and stores its cumulative value in $m_\ell$. In fact, using this $m_\ell$, we can compute an upper bound for $\lambda_{\rm min} (x)$ (Proposition \ref{prop:lambda-min-upper-2}). On line 18, $\bar{Q}_\ell$ is updated as $\bar{Q}_\ell \leftarrow \bar{Q}_\ell Q_{g_\ell^{-1}}$, and $\bar{Q}_\ell$ of Algorithm \ref{main algorithm 2} plays the role of an operator that gives the relation $\langle \bar{x}_\ell , a_\ell \rangle = \langle x_\ell, \bar{Q}_\ell (a_\ell) \rangle$ for the solution $x$ of the original problem, the solution $\bar{x}$ of the scaled problem, and any $a \in \mathbb{E}$. For example, as before, if $|H^1_\ell| > 0$ for $k=1$, then $\langle \bar{x}_\ell , a_\ell \rangle = \langle Q_{g_\ell^{-1}}(x_\ell) , a_\ell \rangle =\langle x_\ell , Q_{g_\ell^{-1}}(a_\ell) \rangle$ is valid. And if $|H^2_\ell| > 0$ even for $k=2$, then the proposed method scales $\bar{x}$ again, so that $\langle \bar{\bar{x}}_\ell ,a_\ell \rangle = \langle \bar{x}_\ell , Q_{g_\ell^{-1}}(a_\ell) \rangle = \langle x_\ell , \bar{Q}_\ell (a_\ell) \rangle $ holds.

Proposition \ref{prop:lambda-min-upper} guarantees that $\varepsilon$-feasibility of the problem ${\rm P}(\mathcal{A})$ can be detected by computing $\det(\bar{x})$ of any feasible solution of $ {\rm P}_{S_{\infty}}(\mathcal{A}^kQ^k)$. The following proposition ensures that we may use the value $\langle \bar{x} , e \rangle$ of any feasible solution of $ {\rm P}_{S_{\infty}}(\mathcal{A}^kQ^k)$ to detect the $\varepsilon$-feasibility of problem ${\rm P}(\mathcal{A})$, instead of $\det(\bar{x})$. While the analysis of the computational complexity in section \ref{sec: finite termination of MA} does not hold for it, the new criteria is better able to detect $\varepsilon$-feasibility in the numerical experiments presented in Appendix \ref{app: numerical results} .

\begin{proposition}
\label{prop:lambda-min-upper-2}
After $k$ iterations of Algorithm \ref{main algorithm 2}, for any feasible  solution $x$ of ${\rm P}_{S_{\infty}}(\mathcal{A})$ and $ \ell \in \{1 , \dots , p\}$, the $\ell$-th block element $x_\ell$ of $x$ satisfies
\begin{equation}
\lambda_{\min}(x_\ell) \leq \cfrac{r_\ell}{\left( r_\ell + \left( \frac{1}{\xi} - 1 \right) m_\ell \right)} . \label{prop:11-0}
\end{equation}
\end{proposition}

\begin{proof}
In Algorithm \ref{main algorithm 2}, $m_\ell$ is updated only when $|H_\ell^k|>0$. 
Suppose that, at the end of  the $k$-th iteration of Algorithm \ref{main algorithm 2}, the last update of $m_\ell$ had been at the $k'(\leq k)$-th iteration.
Then, the stopping criteria of the basic procedure guarantees that at the beginning of the $k'$-th iteration,  $\bar{Q_\ell}$
satisfies
\begin{equation}
\langle x, \bar{Q_\ell}(c^{k'}(v_\ell)_i) \rangle \leq 
\begin{cases}
\xi 	&i \in  H_\ell^{k'} \\
1	&i \notin H_\ell^{k'}
\end{cases}
.
\label{prop:11-0.5}
\end{equation}

This gives a lower bound for $|H_\ell^{k'}|$:
\begin{equation}
\frac{1}{\xi} \left \langle x, \bar{Q_\ell} \left( \sum_{i \in H_\ell^{k'} }c^{k'}(v_\ell)_i \right) \right \rangle \leq |H_\ell^{k'}|. \label{prop:11-0.75}
\end{equation}

Using the fact that $x_\ell - \lambda_{\min} (x_\ell) e_\ell \in \mathcal{K}_\ell$, we obtain
\begin{align*}
\lambda_{\min} (x_\ell) \langle e_\ell , \bar{Q_\ell} (e_\ell) \rangle  &\leq \langle x_\ell , \bar{Q_\ell} (e_\ell) \rangle \\
&= \left \langle x_\ell , \bar{Q_\ell} \left( \sum_{j \notin H_\ell^{k'}} {c^{k'}(v_\ell)}_j \right) \right \rangle + \left \langle x_\ell , \bar{Q_\ell} \left( \sum_{j \in H_\ell^{k'}} {c^{k'}(v_\ell)}_j \right) \right \rangle \\
&\leq r_\ell - |H_\ell^{k'}| + \left \langle x_\ell , \bar{Q_\ell} \left( \sum_{j \in H_\ell^{k'}} {c^{k'}(v_\ell)}_j \right) \right \rangle    \ \ \ \mbox{(by  (\ref{prop:11-0.5}))}  \\
&\leq r_\ell  - \left( \cfrac{1}{\xi} - 1 \right) \left \langle x_\ell  , \bar{Q_\ell} \left( \sum_{j \in H_\ell^{k'}} {c^{k'}(v_\ell)}_j \right) \right \rangle  \ \ \ \mbox{(by  (\ref{prop:11-0.75}))} \\
&\leq r_\ell  - \left( \cfrac{1}{\xi} - 1 \right) \lambda_{\min} (x_\ell) \left \langle e_\ell , \bar{Q_\ell} \left( \sum_{j \in H_\ell^{k'}} {c^{k'}(v_\ell)}_j \right) \right \rangle,
\end{align*}
and hence,  
\begin{equation}
\lambda_{\min} (x_\ell) \left( \left \langle e_\ell , \bar{Q_\ell} (e_\ell) \right \rangle + \left( \cfrac{1}{\xi} - 1 \right)  \left \langle e_\ell , \bar{Q_\ell} \left( \sum_{j \in H_\ell^{k'}} {c^{k'}(v_\ell)}_j \right) \right \rangle \right) \leq r_\ell. \label{prop:11-1}
\end{equation}

Next, suppose that, at the beginning of the $k'$-th iteration of Algorithm \ref{main algorithm 2}, the last update of $m_\ell$ had been performed at the $i(<k')$-th iteration.

Let $\bar{Q_\ell}^{\rm pre}$ be $\bar{Q_\ell}$ obtained at the beginning of the $i$-th iteration of Algorithm \ref{main algorithm 2}, and let $Q_{g_\ell}^{\rm pre}$ and $m_\ell^{\rm pre}$ be $Q_{\ell}$ and $m_\ell$ obtained after the update at the $i$-th iteration. 
Note that  $\bar{Q_\ell}$ at the beginning of the $k'$-th iteration of Algorithm \ref{main algorithm 2} can be represented by $\bar{Q_\ell} = \bar{Q_\ell}^{\rm pre} Q_{g_\ell^{-1}}^{\rm pre}$.
Thus, from (\ref{eq:prop3.4-1}), we see that 
\begin{align*}
 Q_{g_\ell^{-1}}^{\rm pre}(e_\ell) 
 & =  Q_{g_\ell^{-1}}^{\rm pre} \left( \sum_{j=1}^{r_\ell} {c^i(v_\ell)}_j \right) \\
 & =  Q_{g_\ell^{-1}}^{\rm pre} \left( \sum_{j \in H^{k'}_\ell}  {c^i(v_\ell)}_j \right) + Q_{g_\ell^{-1}}^{\rm pre} \left( \sum_{j \not\in H^{k'}_\ell}  {c^i(v_\ell)}_j \right) \\
 & = \frac{1}{\xi} \sum_{j \in H^{k'}_\ell}  {c^i(v_\ell)}_j  + \sum_{j \not\in H^{k'}_\ell}  {c^i(v_\ell)}_j \\
 & = e_\ell + \left( \frac{1}{\xi} - 1 \right) \sum_{j \in H_\ell^i} {c^i (v_\ell)}_j 
\end{align*}
and hence,  
\begin{align}
\bar{Q_\ell}(e_\ell) = \bar{Q_\ell}^{\rm pre} Q_{g_\ell^{-1}}^{\rm pre}(e_\ell) 
&= \bar{Q_\ell}^{\rm pre} \left( e_\ell + \left( \frac{1}{\xi} - 1 \right) \sum_{j \in H_\ell^i} {c^i (v_\ell)}_j \right) \notag \\
&= \bar{Q_\ell}^{\rm pre} \left( e_\ell \right) + \left( \frac{1}{\xi} - 1 \right) \bar{Q_\ell}^{\rm pre} \left(  \sum_{j \in H_\ell^i} {c^i (v_\ell)}_j \right). \label{prop:11-2}
\end{align}

By recursively applying (\ref{prop:11-2}) to $\bar{Q_\ell}^{\rm pre} \left( e_\ell \right)$, we finally obtain
\begin{equation}
\left \langle e_\ell , \bar{Q_\ell} (e_\ell) \right \rangle = r_\ell +\left( \frac{1}{\xi} - 1 \right) m_\ell^{\rm pre} . \notag
\end{equation}

Let $m_\ell^{k'}$ be the value of $m_\ell$ obtained after the update at the $k'$-th iteration.
Then, 
\begin{equation}
\label{update-m}
m_\ell^{k'} = m_\ell^{\rm pre} + \left \langle e_\ell , \bar{Q_\ell} \left( \sum_{j \in H_\ell^{k'}} {c^{k'}(v_\ell)}_j \right) \right \rangle
\end{equation}
and, by (\ref{prop:11-1}), we obtain
\begin{align}
\lambda_{\min} (x_\ell)  &\leq \cfrac{r_\ell}{\left( r_\ell +\left( \frac{1}{\xi} - 1 \right) m_\ell^{\rm pre}  + \left( \frac{1}{\xi} - 1 \right)  \left \langle e_\ell , \bar{Q_\ell} \left( \sum_{j \in H_\ell^{k'}} {c^{k'}(v_\ell)}_j \right) \right \rangle \right)} \notag \\
&= \frac{r_\ell}{\left( r_\ell  + \left( \cfrac{1}{\xi} - 1 \right) m_\ell^{k'} \right)} . \notag
\end{align}

Since, at the end of the $k$-th iteration of  Algorithm \ref{main algorithm 2}, the last update of $m_\ell$ was at the $k'$-th iteration, we see that $m_\ell = m_\ell^{k'}$, and hence (\ref{prop:11-0}) holds after  $k$ iterations of Algorithm \ref{main algorithm 2}.
\end{proof}

Using Proposition \ref{prop:lambda-min-upper-2}, we find an upper bound on the number of iterations of Algorithm \ref{main algorithm 2}.
\begin{proposition}
\label{prop:iteration-num-ma-2}
Algorithm \ref{main algorithm 2} terminates after no more than $\frac{\xi}{1-\xi} \left( \frac{1}{\varepsilon} - 1\right) r - p + 1$ iterations.
\end{proposition}
\begin{proof}
When $|H_\ell^k| > 0$ for $ \ell \in \{1 , \dots , p\}$ at the $k$-th iteration of Algorithm \ref{main algorithm 2}, we say that the iteration is ``good'' for the $\ell$-th block. From Proposition \ref{prop:lambda-min-upper-2}, since the (meaningful) upper bound of the minimum eigenvalue $\lambda_{\min} (x_\ell)$ of $x_\ell$ of the $\ell$-th block of any feasible solution $x$ of ${\rm P}_{S_{\infty}}(\mathcal{A})$ depends on the value of $m_\ell$, we first calculate a lower bound for the increment of $m_\ell$ per good iteration in the $\ell$-th block.
Similar to the proof of Proposition \ref{prop:lambda-min-upper-2}, suppose that the $k'$-th iteration is a good iteration for the $\ell$-th block. As shown in equation (\ref{update-m}), the value of $m_\ell$ is increased at this time by $\left \langle e_\ell , \bar{Q_\ell} \left( \sum_{j \in H_\ell^{k'}} {c^{k'}(v_\ell)}_j \right) \right \rangle$ using $\bar{Q}_\ell$ at the beginning of the $k'$-th iteration. Let us express $Q_{g_\ell^{-1}}$ using $g_\ell$ obtained at the $k$-th iteration as $Q^k_{g_\ell^{-1}}$, i.e., $\bar{Q_\ell} = Q^1_{g_\ell^{-1}} Q^2_{g_\ell^{-1}} \dots Q^{k'-1}_{g_\ell^{-1}}$. Then, the increment of $m_\ell$ at the $k'$-th iteration is as follows:
\begin{equation}
\label{eq-m-value}
\left \langle e_\ell , \bar{Q_\ell} \left( \sum_{j \in H_\ell^{k'}} {c^{k'}(v_\ell)}_j \right) \right \rangle = \left \langle Q^{k'-1}_{g_\ell^{-1}} \dots Q^1_{g_\ell^{-1}} \left( e_\ell \right) , \sum_{j \in H_\ell^{k'}} {c^{k'}(v_\ell)}_j \right \rangle.
\end{equation}
Note that $Q^{k'-1}_{g_\ell^{-1}} \dots Q^1_{g_\ell^{-1}} \left( e_\ell \right) - e_\ell \in \mathcal{K}_\ell$ holds, as we will prove below using induction.
First, if the first iteration is a good one for the $\ell$ block, then $Q^1_{g_\ell^{-1}} \left( e_\ell \right) = \frac{1}{\xi} \sum_{i \in H_\ell^1} c^1(v_\ell)_i+ \sum_{j \notin H_\ell^1} c^1(v_\ell)_j = e_\ell + \left( \frac{1}{\xi} -1 \right)\sum_{i \in H_\ell^1} c^1(v_\ell)_i$, and if it is not a good iteration, then $Q^1_{g_\ell^{-1}} \left( e_\ell \right) =e_\ell$.
Thus, $Q^1_{g_\ell^{-1}} \left( e_\ell \right) - e_\ell \in \mathcal{K}_\ell$ holds.
Next, when  $Q^{i}_{g_\ell^{-1}} \dots Q^1_{g_\ell^{-1}} \left( e_\ell \right) - e_\ell \in \mathcal{K}$ holds, by Proposition \ref{ptop:quadratic}, $Q^{i+1}_{g_\ell^{-1}} \left( Q^{i}_{g_\ell^{-1}} \dots Q^1_{g_\ell^{-1}} \left( e_\ell \right) - e_\ell \right) \in \mathcal{K}_\ell$ holds.
Furthermore, the same calculation as in the first iteration yields $Q^{i+1}_{g_\ell^{-1}} \left( e_\ell \right) - e_\ell \in \mathcal{K}_\ell$, and we see that 
\begin{align*}
Q^{i+1}_{g_\ell^{-1}} \left( Q^{i}_{g_\ell^{-1}} \dots Q^1_{g_\ell^{-1}} \left( e_\ell \right) - e_\ell \right) \in \mathcal{K}_\ell
&\Leftrightarrow Q^{i+1}_{g_\ell^{-1}} Q^{i}_{g_\ell^{-1}} \dots Q^1_{g_\ell^{-1}} \left( e_\ell \right) - Q^{i+1}_{g_\ell^{-1}} \left( e_\ell \right)  \in \mathcal{K}_\ell \\
&\Rightarrow Q^{i+1}_{g_\ell^{-1}} Q^{i}_{g_\ell^{-1}} \dots Q^1_{g_\ell^{-1}} \left( e_\ell \right) - e_\ell  \in \mathcal{K}_\ell.
\end{align*}
Thus, from (\ref{eq-m-value}), we obtain a lower bound for the increment of $m_\ell$ as
\begin{align*}
\left \langle Q^{k'-1}_{g_\ell^{-1}} \dots Q^1_{g_\ell^{-1}} \left( e_\ell \right) , \sum_{j \in H_\ell^{k'}} {c^{k'}(v_\ell)}_j \right \rangle
&\geq \left \langle e_\ell,  \sum_{j \in H_\ell^{k'}} {c^{k'}(v_\ell)}_j \right \rangle \\
&= |H^{k'}_\ell| \geq 1,
\end{align*}
which means that the value of $m_\ell$ increases by at least $1$ per good iteration.
Therefore, if the number of good iterations for the $\ell$-th block is $\left( \frac{r_\ell}{\varepsilon} - r_\ell \right) \left( \frac{\xi}{1-\xi} \right)$ or more, then from Proposition \ref{prop:lambda-min-upper-2}, we can conclude that $\lambda_{min}(x_\ell) \leq \varepsilon$ holds; i.e., we obtain an upper bound for the number of iterations of Algorithm \ref{main algorithm 2} as follows:
\begin{equation}
\sum_{\ell=1}^p \left(  \left( \frac{r_\ell}{\varepsilon} - r_\ell \right) \left( \frac{\xi}{1-\xi} \right) -1 \right) + 1 = \frac{\xi}{1-\xi} \left( \frac{1}{\varepsilon} - 1\right) r - p + 1. \notag
\end{equation}
\end{proof}

It should be noted that the number of iterations required by Algorithm \ref{main algorithm 2} to detect the non-existence of $\varepsilon$-feasible solutions is actually likely to be much smaller than the value given in Proposition \ref{prop:iteration-num-ma-2}.
This is because Proposition \ref{prop:iteration-num-ma-2} calculates the lower bound for the increment of $m_\ell$ for one good iteration as $1$.
The increment of $m_\ell$ can be calculated using $\bar{Q}_\ell$, but it is difficult to calculate the exact increment of $m_\ell$ because $\bar{Q}_\ell$ depends on the results returned by the previous basic procedure. 
Suppose that both the first and second iterations are good for the $\ell$-th block. Then, the increment of $m_\ell$ at the second iteration is 
\begin{equation}
\notag
\left \langle Q^1_{g_\ell^{-1}} \left( e_\ell \right) , \sum_{j \in H_\ell^{2}} {c^{2}(v_\ell)}_j \right \rangle = \left \langle e_\ell + \left( \frac{1}{\xi} -1 \right)\sum_{i \in H_\ell^1} c^1(v_\ell)_i , \sum_{j \in H_\ell^{2}} {c^{2}(v_\ell)}_j \right \rangle,
\end{equation}
but it is difficult to find a lower bound greater than 0 for $\left \langle \sum_{i \in H_\ell^1} c^1(v_\ell)_i , \sum_{j \in H_\ell^{2}} {c^{2}(v_\ell)}_j \right \rangle$.

\section{Additional numerical experiments}

In addition to the strongly feasible instances described in section \ref{sec: instances}, we generated the following two types of instances and conducted numerical experiments.
\begin{itemize}
\item Weakly feasible instances, i.e., $\mbox{ker} \hspace{0.3mm} \mathcal{A} \cap  \mathbb{S}^n_{++} = \emptyset$, but $\mbox{ker} \hspace{0.3mm} \mathcal{A} \cap  \mathbb{S}^n_{+} \setminus \{O\} \neq  \emptyset$.
\item Infeasible instances, i.e., $\mbox{ker} \hspace{0.3mm} \mathcal{A} \cap  \mathbb{S}^n_{+} = O $.
\end{itemize}

\subsection{How to generate instances used in the additional experiments}

Here, we describe how the weakly feasible instances and infeasible instances were generated.
Note that, due to the rounding error of the numerical computation, the weakly (ill-conditioned strongly) feasible instances generated in this experiment may not have been weakly (ill-conditioned strongly) feasible, \modifySecond{and could be infeasible or interior feasible. Thus, the term ``fragilely'' may not be appropriate, but it would be better to call them ``fragilely weakly feasible (fragilely ill-conditioned strongly).''}
\subsubsection{Weakly feasible instances}
\label{sec: weakly feasible  instances}

The weakly feasible instances were generated by Algorithm \ref{weakly feasible instance}.

\begin{algorithm}
\caption{Weakly feasible instance}
\label{weakly feasible instance}
\begin{algorithmic}[1]
\renewcommand{\algorithmicrequire}{\textbf{Input: }}
\renewcommand{\algorithmicensure}{\textbf{Output: }}

 \STATE \algorithmicrequire $n , m$, $A' = [ \ \ ]$
 \STATE \algorithmicensure $A$ 
 \STATE $B \leftarrow \mbox{rand} (n,n)$        \hspace{2.4cm}// $B$ must not be $O$
 \STATE $C \leftarrow \frac{B+B^T}{2}$                     \hspace{2.8cm}// $C \neq O$ must not be $C \succeq O$ or  $C \preceq O$
 \STATE $C_+ \leftarrow \mathcal{P}_{\mathbb{S}^n_+} (C)$  \hspace{2.5cm}// $C_+ \neq O$ since $C \neq O$ is not negative semidefinite.
 \STATE $C_- \leftarrow -\mathcal{P}_{\mathbb{S}^n_+} (-C)$   \hspace{2cm}// $C_- \neq O$ since $C \neq O$ is not  positive semidefinite.
 \STATE $c_+ \leftarrow \mbox{vec} (C_+) $ and $R \leftarrow I - \frac{1}{\|c_+\|_2^2} c_+c_+^T $
 \FOR {$ i = 1 \mbox{ to  }  m-1 $}
 \STATE $A_i' \leftarrow \mbox{rand}(n,n)$ and $A_i \leftarrow \left( A_i' + (A_i')^T \right)/2 $
 \STATE $A' \leftarrow  
\begin{pmatrix}
A' \\
\mbox{vec}(A_i)^T \\
\end{pmatrix}$ 
 \ENDFOR
 \STATE $A \leftarrow 
\begin{pmatrix}
\mbox{vec}(C_-)^T \\
A'R
\end{pmatrix}$
\end{algorithmic} 
\end{algorithm}

\begin{proposition}
For any $A \in \mathbb{R}^{m \times n^2}$ returned by Algorithm \ref{weakly feasible instance}, no $X \in \mathbb{S}^n_{++}$ exists that satisfies $A\left( \mbox{\rm vec} (X) \right)= \bm{0}$, but an $X \in \mathbb{S}^n_{+} \setminus \{ O \}$ exists that satisfies $A\left( \mbox{\rm vec} (X) \right)= \bm{0}$.
\end{proposition}
\begin{proof}
First, we show that an $X \in \mathbb{S}^n_{+} \setminus \{ O \}$ exists that satisfies $A\left( \mbox{vec} (X) \right)= \bm{0}$. For the matrix $C_+ \in \mathbb{S}^n_{+}$ computed on line 5 of Algorithm \ref{weakly feasible instance}, we see that $C_+ \neq O$ and the following holds:
\begin{equation}
A\left( \mbox{vec} (C_+) \right)=  Ac_+ = 
\begin{pmatrix}
\mbox{vec}(C_-)^T \\
A'R
\end{pmatrix} c_+ 
=
\begin{pmatrix}
\mbox{vec}(C_-)^T c_+ \\
A'R c_+
\end{pmatrix}
=
\begin{pmatrix}
\bm{0} \\
A'( c_+ - c_+)
\end{pmatrix}
= \bm{0}. \notag
\end{equation}
Next, we show by contradiction that no $X \in \mathbb{S}^n_{++}$ exists that satisfies $A\left( \mbox{vec} (X) \right)= \bm{0}$. Suppose that an $X \in \mathbb{S}^n_{++}$ satisfies $A\left( \mbox{vec} (X) \right)= \bm{0}$. Since the first row of $A$ is $\mbox{vec}(C_-)^T$, if $A\left( \mbox{vec} (X) \right)= \bm{0}$ holds, then $\mbox{vec}(C_-)^T \mbox{vec}(X) = 0$, i.e., 
\begin{align*}
\mbox{vec}(C_-)^T \mbox{vec}(X) &= \langle C_- , X \rangle = \langle PDP^T, QEQ^T \rangle  \\
&= \langle D ,P^TQEQ^TP \rangle = \sum_{i=1}^n D_{ii} \left( P^TQEQ^TP \right)_{ii} = 0
\end{align*}
where $C_-=PDP^T$, $X = QEQ^T$, $P$ are $Q$ orthogonal matrices, and $D$ and $E$ are diagonal matrices. Here, $X \in \mathbb{S}^n_{++}$ implies $\left(P^TQEQ^TP\right)_{ii} > 0$ for any $ i \in \{ 1 , \dots, n \}$ and hence, $D$ should be $O$, but this contradicts $C_- \neq O$. Thus, no $X \in \mathbb{S}^n_{++}$ exists satisfying $A\left( \mbox{vec} (X) \right)= \bm{0}$.
\end{proof}

\subsubsection{Infeasible instances}
\label{sec: infeasible  instances}
The infeasible instances were generated by Algorithm \ref{infeasible instance}.
If we define the linear operator $\mathcal{A} : \mathbb{S}^n \rightarrow \mathbb{R}^m$ as $\mathcal{A}(X) = ( \langle A_1, X \rangle , \dots ,\langle A_m, X \rangle)^T$, then by choosing $A_1 \in \mathbb{S}^n_{++}$, we obtain $\mathcal{A}$ such that $\mbox{ker} \hspace{0.3mm} \mathcal{A} \cap \mathbb{S}^n_{+} = \{ O \} $. On the basis of this observation, by introducing a parameter $\alpha > 0$, we generated a positive definite matrix $A_1$ whose minimum eigenvalue is a uniformly distributed random number in $(0,\alpha)$. We chose $\alpha \in \{1e-1, 1e-2, 1e-3, 1e-4, 1e-5 \}$. The input of Algorithm \ref{infeasible instance} consisted of the rank of the semidefinite cone $n$, the number of constraints $m$, an arbitrary orthogonal matrix $P$, and the parameter $\alpha > 0$.

\begin{algorithm}
\caption{Infeasible instance}
\label{infeasible instance}
\begin{algorithmic}[1]
\renewcommand{\algorithmicrequire}{\textbf{Input: }}
\renewcommand{\algorithmicensure}{\textbf{Output: }}

 \STATE \algorithmicrequire $n, m, \alpha, P$, $A' = [ \ \ ]$
 \STATE \algorithmicensure $A$ 
 \STATE $B \leftarrow \mbox{rand} (n,n)$
 \STATE $B' \leftarrow \frac{B+B^T}{2}$ and then compute an orthogonal matrix $Q$ and diagonal matrix $E$ such that $B' = QDQ^T$
 \STATE $E_+ = \mbox{rand} \left( 1 \right) \times \alpha I + \mathcal{P}_{\mathbb{S}^n_+} (E)$ 
 \STATE $d \leftarrow \mbox{rand} (n)$ and $D \leftarrow \mbox{diag}(d)$
 \STATE $B_+ \leftarrow QE_+Q^T $ and $C \leftarrow PDP^T $ 
 \STATE $c = \mbox{vec}(C)  $ and $R \leftarrow I - \frac{1}{\|c\|_2^2} cc^T $
 \FOR {$ i = 1 \mbox{ to  }  m-1 $}
 \STATE $A_i' \leftarrow \mbox{rand} (n,n)$ and $A_i \leftarrow \left( A_i' + (A_i')^T \right)/2 $
 \STATE $A' \leftarrow  
\begin{pmatrix}
A' \\
\mbox{vec}(A_i)^T \\
\end{pmatrix}$
 \ENDFOR
 \STATE $A \leftarrow 
\begin{pmatrix}
\mbox{vec} (B_+)^T \\
A'R
\end{pmatrix}$
\end{algorithmic} 
\end{algorithm}

Note that the first row of the matrix $A$ returned by Algorithm \ref{infeasible instance} is ${\mbox{vec}(B_+)}^T$. Since $B_+ \in \mathbb{S}^n_{++}$, we see that ${\mbox{vec}(B_+)}^T \mbox{vec}(X) > 0$ for any positive definite matrix $X \in \mathbb{S}^n_{++}$. Thus, there is no $X \in \mathbb{S}^n_{++}$ satisfying $A \left( \mbox{vec}(X) \right)=\bm{0}$, which implies that the generated instance is infeasible.

\subsection{Additional numerical results and observations}
\label{app: numerical results}

As in section \ref{sec: numerical experiments}, we set the size of the semidefinite matrix to $n=50$ and the number of constraints $m$ using $\nu \in \{0.1, 0.3, 0.5, 0.7, 0.9\}$. 
For each $\nu \in \{0.1, 0.3, 0.5, 0.7, 0.9\}$, we generated five instances, i.e., 25 instances for a weakly feasible case and 25 instances for each of five infeasible cases (corresponding to five patterns of $\alpha = \mbox{1e-1}, \dots, \alpha = \mbox{1e-5}$, see section \ref{sec: infeasible instances} for details). Thus, we generated 25 weakly feasible instances and 125 infeasible instances. We set the upper limit of the execution time to 2 hours and compared the performance of our method (Algorithm \ref{main algorithm}, \ref{main algorithm 2}) with those of Louren\c{c}o (2019) and Pena(2017).

We classified the output-results into five types: A: an interior feasible solution is found; B: no interior feasible solution is found (ver.1); C: no $\varepsilon$-feasible solution is found (only for Loren\c{c}o (2019) and our method); D: no interior feasible solution is found (ver.2; only for Pena (2017)); E: Out-of-time. In what follows, we briefly explain how output-result type D for Pena (2017) differs from output-result type B.

\cite{Pena2017} pointed out that if ${\rm P}(\mathcal{A})$ has no interior feasible solution, meaning that if the main algorithm of Pena (2017) is applied to only ${\rm P}(\mathcal{A})$, it does not stop within a finite number of iterations. To overcome this problem, Pena et al. constructed the main algorithm in a way that it applies not only to ${\rm P}(\mathcal{A})$ but also to problem ${\rm Q}(\mathcal{A})$:

\begin{equation}
\begin{array}{llllll}
{\rm Q}(\mathcal{A})	&\mbox{find}	&X \in \mathbb{S}^n_{++}	&\mbox{s.t.}	&X \in \mbox{range} \hspace{0.3mm} \mathcal{A}^*. \notag
\end{array}
\end{equation}

Accordingly, we defined output-result type B as the case where a feasible solution of ${\rm D} (\mathcal{A})$ is obtained by applying the main algorithm to ${\rm P}(\mathcal{A})$ and defined output-result type D as the case where a feasible solution of ${\rm Q}(\mathcal{A})$ is obtained by applying the main algorithm to ${\rm Q}(\mathcal{A})$.

Table \ref{infeasible} summarizes the results for infeasible instances.
Similarly to Table \ref{strongly-feasible-MVN-SP1}, the ``CO-ratio and ``times(s)'' columns respectively show the ratio of correct outputs and the average CPU time of each method (the values in parentheses () in rows $\alpha = \mbox{1e-4}$ and $\alpha = \mbox{1e-5}$ are the average CPU times of each method excluding the instances for which the method ended up running out of time). 
When using MVN as the basic procedure, whereas our method and Louren\c{c}o (2019) found an element of $\mbox{range} \hspace{0.3mm} \mathcal{A}^* \cap \mathbb{S}^n_+$ for all instances, Pena (2017) ended up running out of time for one instance for $\alpha=\mbox{1e-4}$ and $\alpha=\mbox{1e-5}$. 

From the results for infeasible instances, we can observe the following three points. First, our method obtained correct outputs for every instance in a short execution time. This would be because it employed an efficient scaling and found an element of $\mbox{range} \hspace{0.3mm} \mathcal{A}^* \cap \mathbb{S}^n_+$. Second, the method of Pena (2017) obtained better results when SP was used as the basic procedure. As shown in Table \ref{infeasible}, the method of Pena (2017) using SP as the basic procedure solved all problems and had shorter execution times than the method using MVN. Since Pena's (2017) method calls the basic procedure not only to find points in $\mbox{ker} \hspace{0.3mm} \mathcal{A} \cap \mathbb{S}^n_{++}$ but also to find points in $\mbox{range} \hspace{0.3mm} \mathcal{A}^* \cap \mathbb{S}^n_{++}$, using SP, which can update basic procedures efficiently, is better than using MVN in terms of execution time. Third, it is not always possible to detect infeasibility (i.e., to find a point in $\mbox{range} \hspace{0.3mm} \mathcal{A}^* \cap \mathbb{S}^n_{+}$) in a shorter execution time when using SP than when using MVN. In fact, according to Louren\c{c}o (2019), the execution time is shorter when using MVN as the basic procedure than when using SP. SP is a more efficient update method than MVN in terms of satisfying a termination criterion (the criterion for moving to scaling) of the basic procedure. On the other hand, from the point of view of finding points in $\mbox{range} \hspace{0.3mm} \mathcal{A}^* \cap \mathbb{S}^n_{+}$, it is not possible to determine whether SP or MVN is more suitable. Pena (2017) used SP to significantly reduce the execution time, which is the result of updating the basic procedure for finding points in $\mbox{range} \hspace{0.3mm} \mathcal{A}^* \cap \mathbb{S}^n_{++}$ more efficiently than MVN.
Mosek obtained a point in $\mbox{range} \hspace{0.3mm} \mathcal{A}^* \cap \mathbb{S}^n_{++}$ as a feasible solution to the dual problem for all instances. From the viewpoint of execution time, Mosek was superior to the other methods.

\begin{table}
\caption{Results for infeasible instances}
\label{infeasible}
\begin{center}
\begin{tabular}{ll|rr|rr|rr|rr} \toprule
			&	&\multicolumn{2}{c|}{Algorithm \ref{main algorithm}} & \multicolumn{2}{c|}{Louren\c{c}o (2019)} &\multicolumn{2}{c|}{Pena (2017)} &\multicolumn{2}{c}{Mosek} \\ 
Instance		&BP	&CO-ratio&time(s)&CO-ratio&time(s)&CO-ratio&time(s)	&CO-ratio&time(s)\\ \midrule
\multirow{2}{*}{$\alpha= $ 1e-1}	&MVN	&25/25	&1.23	&25/25	&2.37		&25/25	&0.79	&\multirow{2}{*}{25/25}	&\multirow{2}{*}{1.22} \\
						&SP		&25/25	&1.01	&25/25	&21.46	&25/25	&0.61 	\\ 
\multirow{2}{*}{$\alpha= $ 1e-2}	&MVN	&25/25	&4.39	&25/25	&37.93	&25/25	&25.99 &\multirow{2}{*}{25/25}	&\multirow{2}{*}{1.25}\\
						&SP		&25/25	&3.87	&25/25	&62.92	&25/25	&1.05 \\ 
\multirow{2}{*}{$\alpha= $ 1e-3}	&MVN	&25/25	&5.38	&25/25	&61.61	&25/25	&61.55 &\multirow{2}{*}{25/25}	&\multirow{2}{*}{1.25}\\
						&SP		&25/25	&5.34	&25/25	&84.08	&25/25	&2.08 \\ 
\multirow{2}{*}{$\alpha= $ 1e-4}	&MVN	&25/25	&7.81	&25/25	&88.32	&24/24	&(20.80) &\multirow{2}{*}{25/25}&\multirow{2}{*}{1.24}\\
						&SP		&25/25	&7.40	&25/25	&98.79	&25/25	&33.48 \\ 
\multirow{2}{*}{$\alpha= $ 1e-5}	&MVN	&25/25	&9.08	&25/25	&76.17	&24/24	&(9.47) &\multirow{2}{*}{25/25}	&\multirow{2}{*}{1.24}\\
						&SP		&25/25	&8.00	&25/25	&91.88	&25/25	&55.42 \\ \bottomrule
\end{tabular}
\end{center}
\end{table}

For the weakly feasible instances, we compared our method (Algorithm \ref{main algorithm}), a modified version with another criteria for $\varepsilon$-feasibility (Algorithm \ref{main algorithm 2}), Louren\c{c}o (2019), and Pena (2017). The results are summarized in Table \ref{weakly-feasible-output}.
As described above, we classified the output-results into type A: an interior feasible solution is found; type B: no interior feasible solution is found  (ver.1); type C: no $\varepsilon$-feasible solution is found (only for Loren\c{c}o (2019) and our methods); type D: no interior feasible solution is found (ver.2; only for Pena (2017)); type E: Out-of-time. 
Note that $\mbox{B}^*$ indicates that the output was B, but when we converted the obtained solution to a solution of ${\rm D} (\mathcal{A})$, it contained a negative eigenvalue and violated the SDP constraint.
\modifySecond{Note that, due to rounding errors, the true state of each generated weakly feasible instance is unknown, and it is impossible to determine whether the results obtained by the algorithms are correct or incorrect. Thus, Table \ref{weakly-feasible-output} lists the output types and average execution time without noting which are correct.}

From Table \ref{weakly-feasible-output}, we can observe the following:
\begin{itemize}
\item For all the methods, the average execution time was shorter when SP was used as the basic procedure than when MVN was used.

\item All methods except Algorithm \ref{main algorithm 2} sometimes obtained output type A (an interior feasible solution is found), and Pena(2017) returned output-result D, while the obtained solution had $0 \sim 5$ negative eigenvalues (about -1e-16) and more than 20 positive eigenvalues (less than 1e-12) when we converted it into a solution of ${\rm P} (\mathcal{A})$.
\item Louren\c{c}o (2019) obtained output type $\mbox{B}^*$ (no interior feasible solution is found) but when we converted the obtained solution into a solution of ${\rm D} (\mathcal{A})$, it contained a negative eigenvalue and violated the SDP constraint). The obtained solution had $1 \sim 3$ negative eigenvalues (about -1e-6) and violated the SDP constraint when we converted it into a solution of ${\rm D} (\mathcal{A})$.
\item Our modified method (Algorithm \ref{main algorithm 2}) was able to determine the existence of an $\varepsilon$-feasible solution for all instances. This implies that, at least for this specific set of weakly feasible instances, the criteria focusing on the total value of the eigenvalues used in Algorithm \ref{main algorithm 2} is more suitable than the criteria focusing on the product of all the eigenvalues.
\end{itemize}

\begin{table}
\caption{Output types for weakly feasible instances}
\label{weakly-feasible-output}
\begin{center}
\begin{tabular}{cc|ccccc|r}\toprule
Method									&BP		&$\nu=0.1$&$\nu=0.3$&$\nu=0.5$&$\nu=0.7$&$\nu=0.9$	&time(s)\\ \midrule
\multirow{2}{*}{Algorithm \ref{main algorithm}}	&MVN	&AAAAA	&AAAAA	&AAAAA &AAAAA &BBBBB	&414.42\\
										&SP		&AAAAA	&AAAAA	&AAAAA &AAAAA &ABABB	&226.25 \\ 
\multirow{2}{*}{Algorithm \ref{main algorithm 2}} &MVN	&CCCCC	&CCCCC	&CCCCC &CCCCC  &CCCCC	&301.97\\
										&SP		&CCCCC	&CCCCC	&CCCCC &CCCCC  &CCCCC	&179.72\\ 
\multirow{2}{*}{Louren\c{c}o (2019)}			&MVN	&AAAAA & B$\mbox{B}^*$$\mbox{B}^*$$\mbox{B}^*$B & ABAAA & ABA$\mbox{B}^*$$\mbox{B}^*$ &BBBBB	&3512.78\\
										&SP		&AAAAA &AAAAA	&AAAAA	&AAAAA	&BBBBB	&1550.76 \\ 
\multirow{2}{*}{Pena (2017)}					&MVN	&EEEEE	&EEEEE	&EEEEE	&EEEEE	&EEEEE	&\\ 
										&SP		&AAAAA	&DAAAD	&AAAAA	&AAAAA	&DDDDD	&3239.12 \\ \bottomrule
\end{tabular}
\end{center}
\end{table}

Table \ref{weakly-feasible-output-Mosek} summarizes the results obtained by Mosek.
The error message ``rescode = 10006'' was obtained for 22 instances, similar to the results for the strongly feasible ill-conditioned instances.
Note that we assumed that feasible solutions were obtained for all instances since the constraint residual $\|\mathcal{A}(X^*)\|_2$ was as small as 1.1e-7 or less for all obtained solutions.

Note that for all problems with $0.1 \leq \nu \leq 0.7$, Algorithm \ref{main algorithm} and Louren\c{c}o (2019) using SP for the basic procedure returned output A, i.e., a feasible solution to the original problem. Table \ref{weakly feasible compare} summarizes the accuracies of the solutions obtained with Algorithm \ref{main algorithm}, Louren\c{c}o (2019), and Mosek for all instances with $0.1 \leq \nu \leq 0.7$. 
Chubanov’s methods sometimes returned output-result type A for weakly feasible instances, but Table \ref{weakly feasible compare} shows that the average accuracy of feasible solutions obtained by Chubanov’s methods was better than that of Mosek.

\begin{table}
\modifySecond{
\caption{Results for weakly feasible instances with Mosek}
\label{weakly-feasible-output-Mosek}
\begin{center}
\begin{tabular}{l|rr} \toprule
Instance		&time(s)	&$\| \mathcal{A} (X^*)\|_2$	\\ \midrule 
weakly feasible	&5.42 		&6.86e-9				\\ \bottomrule 
\end{tabular}
\end{center}
}
\end{table}

\begin{table}
\caption{Average of the constraint residuals $\|\mathcal{A}(X^*)\|_2$ of the solution $X^*$ obtained for the weakly feasible instances}
\label{weakly feasible compare}
\begin{center}
\begin{tabular}{l|ccc} \toprule
Value of $\nu$	&Algorithm \ref{main algorithm}	&Louren\c{c}o (2019)	&Mosek \\ \midrule
$\nu=0.1$ &1.28e-13	&5.51e-14	&1.45e-12	\\
$\nu=0.3$ &1.56e-13	&7.04e-14	&2.53e-10	\\
$\nu=0.5$ &1.40e-13	&1.05e-13	&1.29e-9	\\
$\nu=0.7$ &3.44e-13	&1.09e-13	&3.75e-9	\\\bottomrule
\end{tabular}
\end{center}
\end{table}


\section{More comparisons of the basic procedures}
\label{sec: comparisons}

In section \ref{sec: complexity of MA vs Lourenco}, we showed that the bound of the computational cost of our method is lower than that of Louren\c{c}o et al. when $\mathcal{K}$ is the $n$-dimensional nonnegative orthant $\mathbb{R}^n_+$ or a Cartesian product of simple second-order cones, and that their bounds on their costs are equivalent when $\mathcal{K}$ is a simple positive semidefinite cone under the assumption that the costs of computing the spectral decomposition and the minimum eigenvalue are the same for an $n \times n$ symmetric matrix. In this section, we make more detailed comparisons of these algorithms in terms of the performance of the cut obtained from the basic procedure and the detectability of an $\varepsilon$-feasible solution. Similarly to section \ref{sec: numerical experiments}, we will refer to Louren\c{c}o et al.'s method \cite{Bruno2019} as Louren\c{c}o (2019) throughout this section.

\subsection{Performance comparison of the two basic procedures for the simple case}

Here, for the sake of simplicity, we will focus on the case where the symmetric cone is simple, i.e., $p=1$. Let $\mathbb{E}$ be the Euclidean space corresponding to the symmetric cone $\mathcal{K}$. For any given $w, v \in \mathbb{E}$, Louren\c{c}o et al. \cite{Bruno2019} defined $\mbox{\rm vol} (w,v)$ as the volume of the intersection $H(w,v) \cap \mathcal{K}$, where $H(w,v)$ is the half space given by $H(w,v) = \{ x \in \mathbb{E} \ \ | \ \ \langle w , x \rangle \leq \langle w,v \rangle \}.$

In this section, we first identify the half-space $H(w,v)$ that will be transferred to the half-space $H(e, e/r)$ after scaling and then find the constant $\mbox{\textbf{rate}} \in \mathbb{R}$ that satisfies $\mbox{vol} (w,v) \leq \mbox{\textbf{rate}} \times \mbox{vol}(e,e/r)$, so that we can compare the proposed method and Louren\c{c}o (2019). The proposed method and Louren\c{c}o (2019) use the basic procedure results to narrow down the original problem's feasible region. It can be interpreted that the algorithm becomes more efficient as the constant $\mbox{\textbf{rate}} \in \mathbb{R}$ (indicating how much $\mbox{vol}(w,v)$ is reduced compared with $\mbox{vol}(e,e/r)$) gets smaller. In what follows, we call the constant $\mbox{\textbf{rate}} \in \mathbb{R}$ the reduction rate.

Section \ref{sec:reduction 1} derives the reduction rate of the proposed method and section \ref{sec:reduction 2} that of Louren\c{c}o (2019). The results in these sections are summarized in Table \ref{table:reduction}, where the ``UB\#iter'' column shows the upper bound on the number of iterations required in the basic procedure. The ``UB\#iter'' of Louren\c{c}o (2019) comes from Proposition 14 of \cite{Bruno2019} (where the authors showed their result by substituting $\rho = 2$), whereas that of Algorithm \ref{basic procedure} comes from Proposition \ref{prop:bp2} with $\ell = 1$. The ``Reduction rate'' of Louren\c{c}o (2019) comes from Theorem \ref{theo:reduction 2}, whereas that of Algorithm \ref{basic procedure} comes from (\ref{eq:reduction 1}) with $(w,v) = (Q_{g^{-1}}(e) , Q_g (e) /r ) $.
By setting $\rho=2$ and $\xi = 1/2$, the two bounds in``UB\#iter'' have the same value; in this case, the reduction rates turn out to be
\begin{align*}
\mbox{Louren\c{c}o (2019):} \ \left( \frac{r^r}{\det w} \right)^{\frac{d}{r}} \leq \left( e^{- \varphi (2)} \right)^{\frac{d}{r}} \simeq (0.918)^{\frac{d}{r}},  \hspace{1cm}
\mbox{Algorithm \ref{basic procedure}: }  \left( \xi^N \right) ^{\frac{d}{r}} \leq \left( \frac{1}{2} \right)^{\frac{d}{r}}.
\end{align*}
The above comparison indicates that Algorithm \ref{basic procedure} is superior to the basic procedure in \cite{Bruno2019} in terms of the reduction rate of the feasible region.

\begin{table}
\caption{Comparison of reduction rates of the two algorithms: Theoretical results}
\label{table:reduction}
\begin{center}
\begin{tabular}{c|c|c} \hline
 Basic procedure  & UB\#iter & Reduction rate \\ \hline
Louren\c{c}o (2019) & $\rho^2 r_{\max}^2$ & $ \mbox{\rm vol} (w,v) = \left( \frac{r^r}{\det w} \right)^{\frac{d}{r}} \mbox{\rm vol} (e,e/r) \leq \left( e^{- \varphi (\rho)} \right)^{\frac{d}{r}} \mbox{\rm vol} (e,e/r) $\\ 
Algorithm \ref{basic procedure} & $r_{\max}^2 / \xi^2$  & $ \mbox{\rm vol} (w, v)  = \left( \xi^N \right)^{\frac{d}{r}} \mbox{\rm vol} (e,e/r)  $  \\ \hline
\end{tabular}
\end{center}
\end{table}

\subsubsection{Theoretical reduction rate of Algorithm \ref{basic procedure}}
\label{sec:reduction 1}

Suppose that Algorithm \ref{basic procedure} returns a result such that there exists a nonempty index set $I \subseteq \{1 ,\dots , r\}$ with $|I| = N $ for which 
\begin{equation}
\langle c_i , x \rangle \leq 
\begin{cases}
\xi & i \in I \\
1 & i \notin I
\end{cases} \label{eq:I}
\end{equation}
holds for any feasible solution $x$ of ${\rm P}_{S_\infty(A)}$, where $\{c_1 , \dots , c_r\}$ are primitive idempotents that make up a Jordan frame.
Note that Algorithm \ref{basic procedure} employs the scaling $\bar{x} = Q_{g^{-1}} (x)$ with $g^{-1} = \frac{1}{\sqrt{\xi}} \sum_{i \in I}  c_i + \sum_{i \notin I} c_i$.
Let us find  $w , v \in \mathbb{E}$ which satisfy
\begin{equation}
H(e , e/r) = Q_{g^{-1}} \left( H(w,v) \right).  \label{appendix-C-1}
\end{equation}
Since (\ref{appendix-C-1}) and the scaling $\bar{x} = Q_{g^{-1}} (x)$ imply that
\begin{align*}
H(w,v) &= Q_g \left( H(e , e/r)  \right) \\
&= \{ Q_g (\bar{x}) \in \mathbb{E}  \ \ | \ \  \langle \bar{x} , e \rangle \leq 1 \} \\
&= \{ Q_g (\bar{x}) \in \mathbb{E}  \ \ | \ \  \langle Q_g (\bar{x}) , Q_{g^{-1}}(e) \rangle \leq 1 \} \\
&= \{ x \in \mathbb{E}  \ \ | \ \  \langle x , Q_{g^{-1}}(e) \rangle \leq \langle Q_{g^{-1}}(e) , Q_g (e)  /r \rangle  = 1\} ,
\end{align*}
by setting $w = Q_{g^{-1}}(e)$ and $v = Q_g (e)  /r$, we find that the half space $H(w,v)$ is transformed to $H(e,e/r)$ after the scaling.
Since $Q_{g^{-1}}(e)  \in \mbox{int} \hspace{0.75mm} \mathcal{K}$, we can apply the following proposition to $w = Q_{g^{-1}}(e)$.

\begin{proposition}[Proposition 6 of~\cite{Bruno2019}]
Suppose that $w \in \mbox{\rm int} \hspace{0.75mm} \mathcal{K}$. Then, 
\begin{align*}
Q_{w^{-1/2} \sqrt{\langle w , v \rangle}} \left( H(e , e/r )\right) &= H(w,v) , \\
\mbox{\rm vol} (w,v) &= \left( \frac{\langle w , v \rangle}{\sqrt[r]{\det w}} \right)^d \mbox{\rm vol} (e,e/r) .
\end{align*}
\end{proposition}

Using the above proposition and the assumption $|I| = N$ for the set $I$ in (\ref{eq:I}), we can see how the volume $\mbox{\rm vol} (Q_{g^{-1}}(e) , Q_g (e) /r)$ of $H (Q_{g^{-1}}(e) , Q_g (e) /r) \cap \mathcal{K}$ decreases compared with $\mbox{\rm vol} (e , e/r)$:
\begin{align}
\notag
\mbox{\rm vol} (Q_{g^{-1}}(e)  , Q_g (e)  /r ) &= \left( \frac{1}{\sqrt[r]{\det Q_{g^{-1}}(e) }} \right)^d \mbox{\rm vol} (e,e/r) \\
&= \left( \frac{1}{\sqrt[r]{ \frac{1}{\xi^N} }} \right)^d \mbox{\rm vol} (e,e/r) 
= \left( \xi^N \right)^{\frac{d}{r}} \mbox{\rm vol} (e,e/r) .\label{eq:reduction 1}
\end{align}

\subsubsection{Theoretical reduction rate of the basic procedure of Louren\c{c}o (2019) }
\label{sec:reduction 2}

The following theorem gives the reduction rate of the basic procedure of Louren\c{c}o (2019).

\begin{theorem}[Theorem 10 of~\cite{Bruno2019}]
\label{theo:reduction 2}
Let $\rho>1$ and $y \in \mathcal{K} \setminus \{ 0 \}$ be such that $F_{{\rm P}_{S_1}(A)} \subseteq H(y , e/\rho r)$.
Let $\beta = r - \left(  \frac{1}{\rho} - \frac{1}{\sqrt{\rho(3\rho-2)}} \right), w=\frac{r-\beta}{\langle y,e \rangle} \rho r y + \beta e, v = w^{-1}.$
Then, the following hold:
\begin{enumerate}
\item $F_{{\rm P}_S(A)} \subseteq H(y ,  e / \rho r) \cap H(e ,  e  / r) \subseteq H(w,v)$
\item $Q_{\sqrt{r} w^{-1 / 2}} \left( H(e ,  e/r) \right) = H(w,v) $
\item 
\begin{equation}
\mbox{\rm vol} (w,v) = \left( \frac{r^r}{\det w} \right)^{\frac{d}{r}} \mbox{\rm vol} (e ,  e/r) 
\leq \left( \exp \left( - \varphi(\rho) \right) \right)^{\frac{d}{r}} \mbox{\rm vol} (e , e/r)  \notag
\end{equation}
where $\varphi(\rho) = 2-\frac{1}{\rho} - \sqrt{3-\frac{2}{\rho}} $.
In particular,  if  $\rho \geq 2$, we have $\mbox{\rm vol} (w ,v) < (0.918)^{\frac{d}{r}}  \mbox{\rm vol} (e , e/r)$.
\end{enumerate}
\end{theorem}

\subsubsection{Comparison of reduction rates of the two algorithms in numerical experiments}

To confirm whether similar reduction rates are observed numerically, we conducted an experiment where we used our method (Algorithms \ref{bp-alg-mvn} and \ref{main algorithm 2}) with $\xi = 1/2$ and Louren\c{c}o (2019) with modified von Neumann scheme to solve a weakly feasible instance with $\nu=0.1$. At each iteration of the main algorithms, we recorded the value of $\frac{r^r}{\det w} $ of Louren\c{c}o (2019) and the value of $\xi^N$ of our method and computed the reduction rates of the search region. The results are summarized in Table \ref{tabel:reduction2}.

\begin{table}
\caption{Comparison of reduction rates of the two algorithms: Numerical results}
\label{tabel:reduction2}
\begin{center}
\begin{tabular}{c|c|c|c|c} \hline
Algorithm & \#iter of M-A & Output & Average reduction rate & Final reduction rate  \\ \hline
Louren\c{c}o (2019) : BP = MVN & 3060 & A &0.864 &3.86e-195 \\ 
Algorithms \ref{bp-alg-mvn} and \ref{main algorithm 2}  &618 & C & 0.357 & 9.11e-305  \\ \hline
\end{tabular}
\end{center}
\end{table}
The ``\#iter of M-A'' column shows the number of iterations of the main algorithm.
The ``Average reduction rate'' column shows the average value of $\frac{r^r}{\det w} $ for Louren\c{c}o (2019) and the average value of $\xi^N$ for our method (Algorithms \ref{bp-alg-mvn} and \ref{main algorithm 2}). The ``Final reduction rate'' column shows the value
\begin{equation}
\frac{r^{kr}}{\det w(1) \times \det w(2) \times \dots \times \det w(k)} \notag
\end{equation}
for Louren\c{c}o (2019), where $w(k)$ denotes $w$ computed from the result of the basic procedure at the $k$-th iteration of the main algorithm, or the value 
\begin{equation}
\xi^{N_1 + \dots + N_k} . \notag
\end{equation}
for our method (Algorithms \ref{bp-alg-mvn} and \ref{main algorithm 2}), where $N_k$ denotes the number of cuts obtained from the basic procedure at the $k$-th iteration of the main algorithm.

Here, we observed that our method (Algorithms \ref{bp-alg-mvn} and \ref{main algorithm 2}) terminated at the 618-th iteration of the main algorithm with a reduction rate of 9.11e-305, while Louren\c{c}o (2019) attained a reduction rate of 5.88e-40 at the same iteration of the main algorithm.

\subsection{Detection of an $\varepsilon$-feasible solution}

Here, we discuss the capabilities of our method and Louren\c{c}o (2019) at detecting an $\varepsilon$-feasible solution. Both methods terminate their main algorithms by detecting the existence of an $\varepsilon$-feasible solution. We compared them by computing the reduction in $\log \left( \lambda_{\min} (x_\ell) \right)$ per iteration for parameter settings in which the maximum numbers of iterations of the basic procedures would be the same (i.e., $\rho=2$ in Louren\c{c}o (2019) and $\xi = 1/2$ in our method).

In \cite{Bruno2019}, for each block $\ell$, Lemma 16 ensures that $\log \left( \lambda_{\min} (x_\ell) \right)$ is bounded from above by $\epsilon_\ell$, and Theorem 17 ensures that $\epsilon_\ell$ decreases at least $\frac{\varphi(\rho)}{r_\ell}>0$ if a {\it good} iteration is obtained for the block $\ell$.
For our method, Proposition \ref{prop:lambda-min-upper} ensures that $\log \left( \lambda_{\min} (x_\ell) \right)$ is bounded from above by $\frac{\mbox{num}_\ell}{r_\ell} \log \xi$ and Proposition \ref{prop:iteration-num-ma} ensures that $\frac{\mbox{num}_\ell}{r_\ell} \log \xi$ decreases $-\frac{1}{r_\ell} \log \xi > 0 $ in the same situation.
By substituting $\rho=2$ and $\xi = 1/2$ into $\varphi(\rho)$ and $- \log \xi$ so that the upper bounds for the numbers of iterations of the basic procedures are the same, we obtain
\begin{equation}
\notag
\begin{array}{ll}
\varphi(2) = 2 - \frac{1}{2} - \sqrt{2} \simeq 0.085786, \ \ 
&-\log \frac{1}{2} = \log 2 \simeq 0.693147
\end{array}
\end{equation}
which implies that the rate of reduction in the upper bound $\log \left( \lambda_{\min} (x_\ell) \right)$ of our method is greater than that of Louren\c{c}o (2019).

\end{document}